\documentclass{amsart}

\usepackage{graphicx}

\allowdisplaybreaks
\numberwithin{equation}{section}

\newcommand{\bE}{\mathbb{E}}
\newcommand{\bM}{\mathbb{M}}
\newcommand{\bN}{\mathbb{N}}
\newcommand{\bP}{\mathbb{P}}
\newcommand{\bQ}{\mathbb{Q}}
\newcommand{\bR}{\mathbb{R}}

\newcommand{\bH}{\mathbb{H}}

\newcommand{\cA}{\mathcal{A}}
\newcommand{\cB}{\mathcal{B}}
\newcommand{\cC}{\mathcal{C}}
\newcommand{\cF}{\mathcal{F}}
\newcommand{\cG}{\mathcal{G}}
\newcommand{\cH}{\mathcal{H}}
\newcommand{\cI}{\mathcal{I}}

\newcommand{\cP}{\mathcal{P}}
\newcommand{\cR}{\mathcal{R}}
\newcommand{\cS}{\mathcal{S}}
\newcommand{\cT}{\mathcal{T}}

\newcommand{\EE}{\mathbf{E}}
\newcommand{\II}{\mathbf{I}}
\newcommand{\JJ}{\mathbf{J}}

\newcommand{\SSS}{\mathbf{S}}
\newcommand{\RR}{\mathbf{R}}
\newcommand{\TT}{\mathbf{T}}

\newcommand{\aaa}{\mathbf{a}}
\newcommand{\bbb}{\mathbf{b}}
\newcommand{\ccc}{\mathbf{c}}
\newcommand{\sss}{\mathbf{s}}
\newcommand{\ttt}{\mathbf{t}}
\newcommand{\www}{\mathbf{w}}

\theoremstyle{plain} \newtheorem{theorem}{Theorem}[section]
\theoremstyle{plain} \newtheorem{lemma}[theorem]{Lemma}
\theoremstyle{plain} \newtheorem{corollary}[theorem]{Corollary}
\theoremstyle{plain} \newtheorem{proposition}[theorem]{Proposition}
\theoremstyle{remark} \newtheorem{remark}[theorem]{Remark}
\theoremstyle{remark} \newtheorem{example}[theorem]{Example}
\theoremstyle{definition} \newtheorem{definition}[theorem]{Definition}
\theoremstyle{definition} 
\theoremstyle{definition} \newtheorem{hypothesis}[theorem]{Hypothesis}

\begin{document}

\title[Trickle-down processes]{Trickle-down processes and their boundaries}

\author{Steven N. Evans}
\address{Department of Statistics\\
         University  of California\\ 
         367 Evans Hall \#3860\\
         Berkeley, CA 94720-3860 \\
         U.S.A.}

\email{evans@stat.berkeley.edu}
\thanks{SNE supported in part by NSF grants DMS-0405778 and DMS-0907630}

\author{Rudolf Gr\"ubel}
\address{Institut f\"ur Mathematische Stochastik\\ 
         Leibniz Universit\"at Hannover\\
         Postfach 6009\\ 
         30060 Hannover\\
         Germany}

\email{rgrubel@stochastik.uni-hannover.de}

\author{Anton Wakolbinger}
\address{Institut f\"ur Mathematik \\
         Goethe-Universit\"at \\
         60054 Frankfurt am Main\\
         Germany}

\email{wakolbinger@math.uni-frankfurt.de}



\date{\today}

\begin{abstract}
It is possible to represent each of a number of Markov chains
as an evolving sequence of connected subsets of a directed acyclic graph that grow in
the following way: initially, all vertices of the
graph are unoccupied, particles are fed in one-by-one
at a distinguished source vertex, successive particles proceed along
directed edges according to an appropriate stochastic mechanism,
and each particle comes to rest once it encounters
an unoccupied vertex.  Examples include 
the binary and digital search tree processes, 
the random recursive tree process and generalizations of it
arising from nested instances of Pitman's two-parameter 
Chinese restaurant process,
tree-growth models associated with Mallows' $\phi$ model of random permutations
and with Sch\"utzenberger's non-commutative $q$-binomial theorem, and a 
construction due to Luczak and Winkler that grows uniform random binary trees
in a Markovian manner.
We introduce a framework that encompasses such Markov chains, 
and we characterize 
their asymptotic behavior by analyzing in detail their Doob-Martin compactifications,
Poisson boundaries and tail $\sigma$-fields.
\end{abstract}

\maketitle

\tableofcontents

\section{Introduction}
\label{S:intro}

Several stochastic processes appearing
in applied probability may be viewed
as growing connected subsets of
a directed acyclic graph that evolve according to the following
dynamics: initially, all vertices of the
graph are unoccupied, particles are fed in one-by-one
at a distinguished source vertex, successive particles proceed along
directed edges according to an appropriate stochastic mechanism,
and each particle comes to rest once it encounters
an unoccupied vertex.  If we picture the source vertex as being
at the ``top'' of the graph, then successive particles ``trickle down''
the graph until they find a vacant vertex that they can occupy.

We are interested in the question:  ``What is the asymptotic behavior of such
a (highly transient) set-valued Markov chain?''  For several of the
models we consider, any finite neighborhood of the source vertex
will, with probability one, be eventually  occupied by a particle
and so a rather unilluminating answer to our question is to say in
such cases that the sequence of sets converges to the entire vertex set $V$.
Implicit in the use of the term ``converges'' in this statement is
a particular topology on the collection of subsets of $V$;
we are embedding the space of
finite subsets of $V$ into the Cartesian product
$\{0,1\}^V$ and equipping the product space with the usual product topology.
A quest for more informative answers can therefore be thought
of as a search for an embedding of the state space
of the chain into a topological space with a richer class of possible limits.

An ideal embedding would be one such that the chain converged almost
surely to a limit and the $\sigma$-field generated by the limit coincided
with the tail $\sigma$-field of the chain up to null events. For 
trickle-down processes, the Doob-Martin compactification provides such 
an embedding, and so our aim is to develop a body of theory that enables
us to identify the compactification for at least some interesting examples.  
Moreover, a knowledge of the Doob-Martin compactification
allows us to determine, via the Doob $h$-transform construction, all
the ways in which it is possible, loosely speaking, to condition the
Markov chain to behave for large times. This allows us to construct 
interesting new processes from existing ones or recognize
that two familiar processes are related by such a conditioning.

A prime example of a Markov chain that fits into the
trickle-down framework is the {\em binary search tree (BST) process}, and
so we spend some time describing the BST process
in order to give the reader some concrete motivation for 
the  definitions we introduce later. 
The BST process and the related {\em digital search tree (DST)
processes} that we consider in Section~\ref{S:BST_and_DST}
arise from considering the behavior of tree-based
searching and sorting algorithms.\label{p:BST} 
The trickle-down mechanism is at the heart of both algorithms: the 
vertices of the complete rooted binary tree are regarded as  
potential locations for the storage of data values $x_1,x_2,\ldots$ 
that arrive sequentially in time. 
We interpret these values as labels of particles. 
The particles are fed in at the root vertex, which receives $x_1$, and they
are routed through the tree until a free vertex is found.  
How we travel onwards from an occupied vertex depends on the algorithm: 
in the BST case we assume that the input stream consists of real numbers 
and we compare the value $x$ to be inserted with the content 
$y$ of the occupied vertex, moving to the left 
or right depending on whether $x<y$ or $x>y$, whereas in the DST case the inputs 
$x_i$ are taken to be infinite 0-1 sequences, 
and we move from an occupied vertex of 
depth $k$ to its left or right child if the $k^{\mathrm{th}}$ component 
of $x_i$ is 0 or 1 respectively. If the input is random and we  
ignore the labeling of the vertices by elements of the input data sequence,
then we obtain a sequence of subtrees of the complete binary tree;
the $n$-th element of the sequence is the subtree consisting
of the vertices occupied by the first $n$ particles.

Binary trees in general and their role in the theory and practice of computer 
science are discussed in \cite{MR0286317}.  
Several tree-based sorting and searching
algorithms are described in \cite{MR0445948}. 
In particular, a class of trees 
(generalizing binary search trees as well as digital search
trees) with a construction similar  to our
trickle-down process is introduced in \cite{MR1634354}.
An introduction to the literature on tree-valued stochastic
processes arising in this connection is \cite{MR1140708}.
Historically, real valued functionals such as the path length or 
the insertion depth of the next item were investigated first, 
with an emphasis on the expected value for random input as a 
function of the amount of stored data 
(that is, of the number of vertices in the tree). 
In recent years, 
several infinite-dimensional random quantities related to the shape of 
the trees  such as the node depth profile 
\cite{MR1878289, MR2380900}, the subtree size profile 
\cite{DeGr3, MR2454562} and the silhouette \cite{MR2569807}
have been studied.

In the present paper we develop a  
framework for  trickle-down processes
that contains the BST and DST processes as special cases. 
As a consequence, we obtain limit results for the
sequence of random trees themselves, using a 
topology on the space of finite binary trees 
that is dictated by the underlying stochastic mechanism. 
We also establish distributional relationships; for example, 
we show that the Markov chains generated by the BST and the DST
algorithms are related via $h$-transforms -- see Theorem~\ref{T:DST_is_h-transform}.

In order to motivate our later formal
definition of trickle-down
processes, we now reconsider the BST process from
a slightly different point of view by 
moving away somewhat from the search tree application 
and starting with a bijection 
from classical enumerative combinatorics (see, for example, \cite{MR1442260}) 
between permutations of the finite set $[n] := \{1,2,\ldots,n\}$ 
and certain trees with $n$ vertices labeled  by $[n]$.  

Denote by $\{0,1\}^\star:=\bigsqcup_{k=0}^\infty\{0,1\}^k$ the set of finite tuples
or {\em words} drawn from the alphabet $\{0,1\}$ (with the empty word $\emptyset$
allowed) -- the symbol $\bigsqcup$ emphasizes that this is a disjoint union.  
Write an $\ell$-tuple $(v_1, \ldots, v_\ell) \in \{0,1\}^\star$ more simply as
$v_1 \ldots v_\ell$. Define a directed graph with vertex set
$\{0,1\}^\star$ by declaring that if
$u= u_1 \ldots u_k$ and $v=v_1 \ldots v_\ell$ are two words,
then $(u,v)$ is a directed edge (that is, $u \rightarrow v$)
if and only if $\ell=k+1$ and $u_i=v_i$
for $i=1,\ldots,k$.  Call this directed graph
the {\em complete rooted binary tree}.  Say that $u < v$ for
two words $u= u_1 \ldots u_k$ and $v=v_1 \ldots v_\ell$ if
$k < \ell$ and $u_1 \ldots u_k = v_1 \ldots v_k$;
that is, $u < v$ if there exist words 
$w_0, w_1, \ldots, w_{\ell-k}$ with
$u = w_0 \to w_1 \to \ldots \to w_{\ell-k} = v$.

A {\em finite rooted binary tree}
is a non-empty subset $\ttt$ of 
$\{0,1\}^\star$ with the property that if $v \in \ttt$
and $u \in \{0,1\}^\star$ is such that $u \rightarrow v$, then $u \in \ttt$.
The vertex $\emptyset$ (that is, the empty word)
belongs to any such tree $\ttt$ and
is the {\em root} of $\ttt$.  See Figure~\ref{fig:binary_tree}.

\begin{figure}[htbp]
	\centering
		\includegraphics[width=1.00\textwidth]{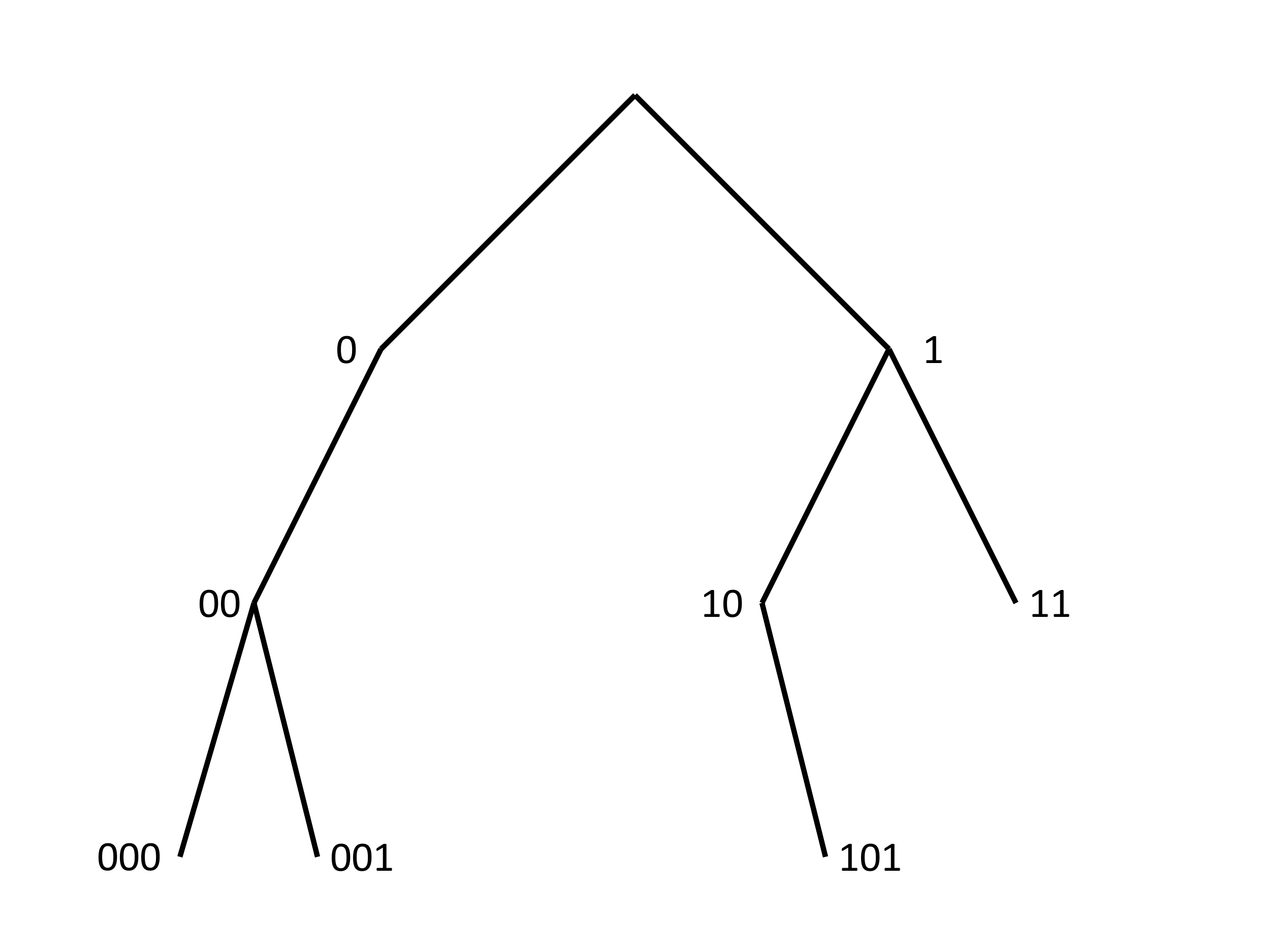}
	\caption{A finite rooted binary tree.} 
	\label{fig:binary_tree}
\end{figure}

If $\# \ttt = n$, then a {\em labeling} of $\ttt$ by $[n]$
is a bijective map $\phi : \ttt \to [n]$.

Suppose that $r(1), \ldots, r(n)$ is an ordered listing of $[n]$.
Define a permutation $\pi$ of $[n]$ by $\pi^{-1}(k) = r(k)$, $k \in [n]$.
%, see Table \ref{table1} for an example.
There is a unique pair  $(\ttt, \phi)$, where 
$\ttt$ is a finite rooted binary tree with
$\# \ttt = n$ and $\phi$
is a {\em labeling} of $\ttt$ by $[n]$, such that
\begin{itemize}
\item
$\phi(\emptyset) = 1$,
\item
if $u,v \in \ttt$ and $u < v$, then $\phi(u) < \phi(v)$,
\item
if $u,v \in \ttt$, $u0 \le v$, then
$\pi \circ \phi(u) > \pi \circ \phi(v)$.
\item
if $u,v \in \ttt$, $u1 \le v$, then
$\pi \circ \phi(u) < \pi \circ \phi(v)$.
\end{itemize}
The labeling may be constructed inductively as follows. If $n=1$, then
we just have the tree consisting of the root $\emptyset$ labeled with $1$.
For $n>1$ we first remove $n$ from the list $r(1), \ldots, r(n)$
and build the labeled tree $(\sss, \psi)$
for the resulting listing of $[n-1]$.  The labeled tree for
$r(1), \ldots, r(n)$ is of the form $(\ttt, \phi)$,
where $\ttt = \sss \cup \{u\}$
for $u \notin \sss$, 
$\phi(u) = n$,
$\phi$ restricted to $\sss$ is $\psi$, and, 
setting $u = u_1 \ldots u_k$, 
\[
u_\ell =
\begin{cases}
0, & \text{if $\pi \circ \psi(u_1 \ldots u_{\ell-1}) < \pi(n)$}, \\
1, & \text{if $\pi \circ \psi(u_1 \ldots u_{\ell-1}) > \pi(n)$}. \\
\end{cases}
\]

To illustrate this construction, take $n=9$ and consider the ordered listing
$r(1), \ldots, r(9)$ of the set $[9]$ to be $8, 7, 9, 4, 1, 3, 5, 2, 6$.
See Table~\ref{tab:permutation} for the resulting permutation, written in the usual
two line format.

\begin{table}[htbp]\label{table1}
	\centering
		\begin{tabular}{c|c|c|c|c|c|c|c|c|c}
		$k$ & 1 & 2 & 3 & 4 & 5 & 6 & 7 & 8 & 9 \\ \hline
$\pi(k)$& 5 & 8 & 6 & 4 & 7 & 9 & 2 & 1 & 3 \\
		\end{tabular}
		\medskip
	\caption{Permutation of $[9]$ with $8, 7, 9, 4, 1, 3, 5, 2, 6$
	as the corresponding ordered listing $r(1), \ldots, r(9)$.}
	\label{tab:permutation}
\end{table}
 
The successive ordered listings of $[1],
[2], \ldots, [9]$ implicit in the recursive construction are
\[
\begin{split}
& 1 \\
& 1, 2 \\
& 1, 3, 2 \\
& \cdots \\
& 8, 7, 4, 1, 3, 5, 2, 6 \\
& 8, 7, 9, 4, 1, 3, 5, 2, 6. \\
\end{split}
\]
As illustrated in Figure~\ref{fig:binary_sort_tree}, 
the label $1$ is inserted at the root, 
the label $2$ trickles down to the vertex $1$, 
the label $3$ trickles down to the vertex $10$, 
the label $4$ trickles down to the vertex $0$, and so on
until
the label $9$ trickles down to the vertex $001$.

\begin{figure}[htbp]
	\centering
		\includegraphics[width=1.00\textwidth]{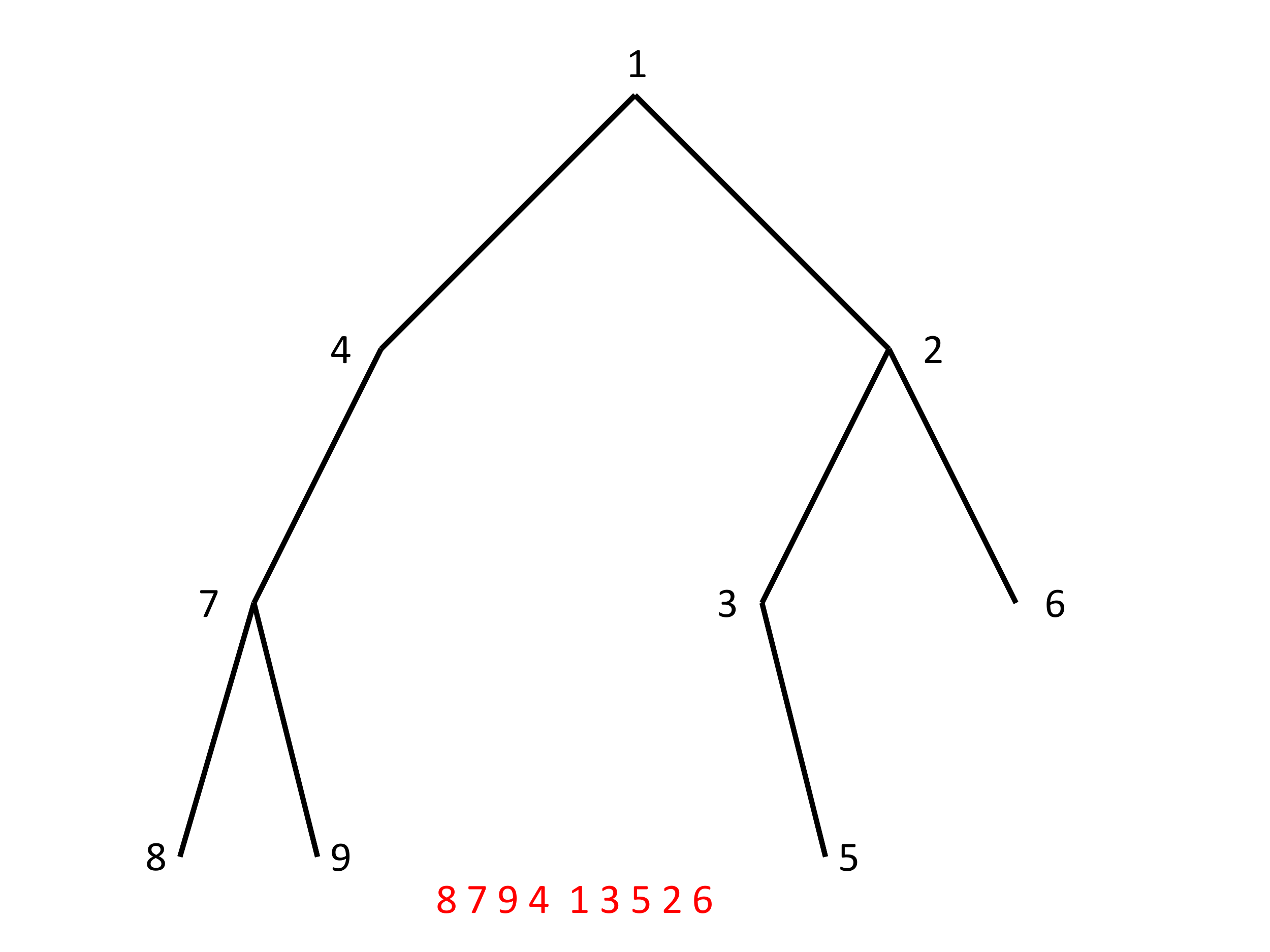}
	\caption{The labeled binary tree corresponding to the  permutation of $[9]$ with 
	$r(1), \ldots, r(9) = 8, 7, 9, 4, 1, 3, 5, 2, 6$.
For the sake of clarity, the 
coding (see Figure~\ref{fig:binary_tree})
of the vertices as elements of $\{0,1\}^\star$ is
%respective 
%$\{0,1\}^\star$-valued formal definitions
%of the vertices are 
not shown.
The correspondence between the labeling by the set $[9]$ and 
the vertices as elements of $\{0,1\}^\star$ is
$1 \leftrightarrow \emptyset$,
$2 \leftrightarrow 1$,
$3 \leftrightarrow 10$,
$4 \leftrightarrow 0$,
$5 \leftrightarrow 101$,
$6 \leftrightarrow 11$,
$7 \leftrightarrow 00$,
$8 \leftrightarrow 000$,
$9 \leftrightarrow 001$.
}
	\label{fig:binary_sort_tree}
\end{figure}

Now let $(U_n)_{n \in \bN}$ be a sequence of independent identically distributed
random variables that each have the uniform distribution on the interval $[0,1]$.
For each positive integer $n$ define a uniformly
distributed random permutation $\Pi_n$
of $[n]$ by requiring that $\Pi_n(i) < \Pi_n(j)$ if and only
if $U_i < U_j$ for $1 \le i,j \le n$.  That is,
$\Pi_n(k) = \# \{1 \le \ell \le n : U_\ell \le U_k\}$
and the corresponding ordered list $R_n(k) := \Pi_n^{-1}(k)$, $1 \le k \le n$,
is such that $U_{R_n(1)} < U_{R_n(2)} < \ldots < U_{R_n(n)}$.
The corresponding ordered list for $\Pi_{n+1}$ is thus obtained by
inserting $n+1$ into one of the $n-1$ ``slots''
between the successive elements of the existing list or into one of
the two ``slots'' at the beginning and end of the list, with
all $n+1$ possibilities being equally likely.

Applying the procedure above for building labeled rooted binary trees 
to the successive permutations $\Pi_1, \Pi_2, \ldots$ produces a
sequence of labeled trees $(L_n)_{n \in \bN}$, where
$L_n$ has $n$ vertices labeled by $[n]$.  
This sequence  is a Markov
chain that evolves as follows. Given $L_n$, there are $n+1$ words of the
form $v = v_1 \ldots v_\ell$ such that $v$ is not a vertex of the tree
$L_n$ but the word $v_1 \ldots v_{\ell-1}$ is.  Pick such a word uniformly
at random and adjoin it (with the label $n+1$ attached) to produce 
the labeled tree $L_{n+1}$.

If we remove the labels from each tree $L_n$, then
the resulting random sequence
 of unlabeled trees is also
a Markov chain that has the same distribution
as the sequence of trees generated by the BST algorithm when
the input stream consists of independent random variables that all 
have the same continuous distribution function. 
In essence, at step $n+1$ of the BST algorithm there are $n+1$ vertices 
that can be added to the existing tree and  the
rank of the input value  $x_{n+1}$ within $x_1, \ldots, x_n, x_{n+1}$ determines the choice of this
``external vertex'':
for i.i.d. continuously
distributed random input, this rank is uniformly distributed on 
$\{1, \ldots, n + 1\}$, resulting in a uniform pick from the
external vertices (see also the discussion following \eqref{E:BST_prob}).
See Figure~\ref{fig:binary_tree_external} for an example showing
the external vertices of the finite rooted binary tree of Figures~\ref{fig:binary_tree} 
and~\ref{fig:binary_sort_tree}.

\begin{figure}[htbp]
	\centering
		\includegraphics[width=1.00\textwidth]{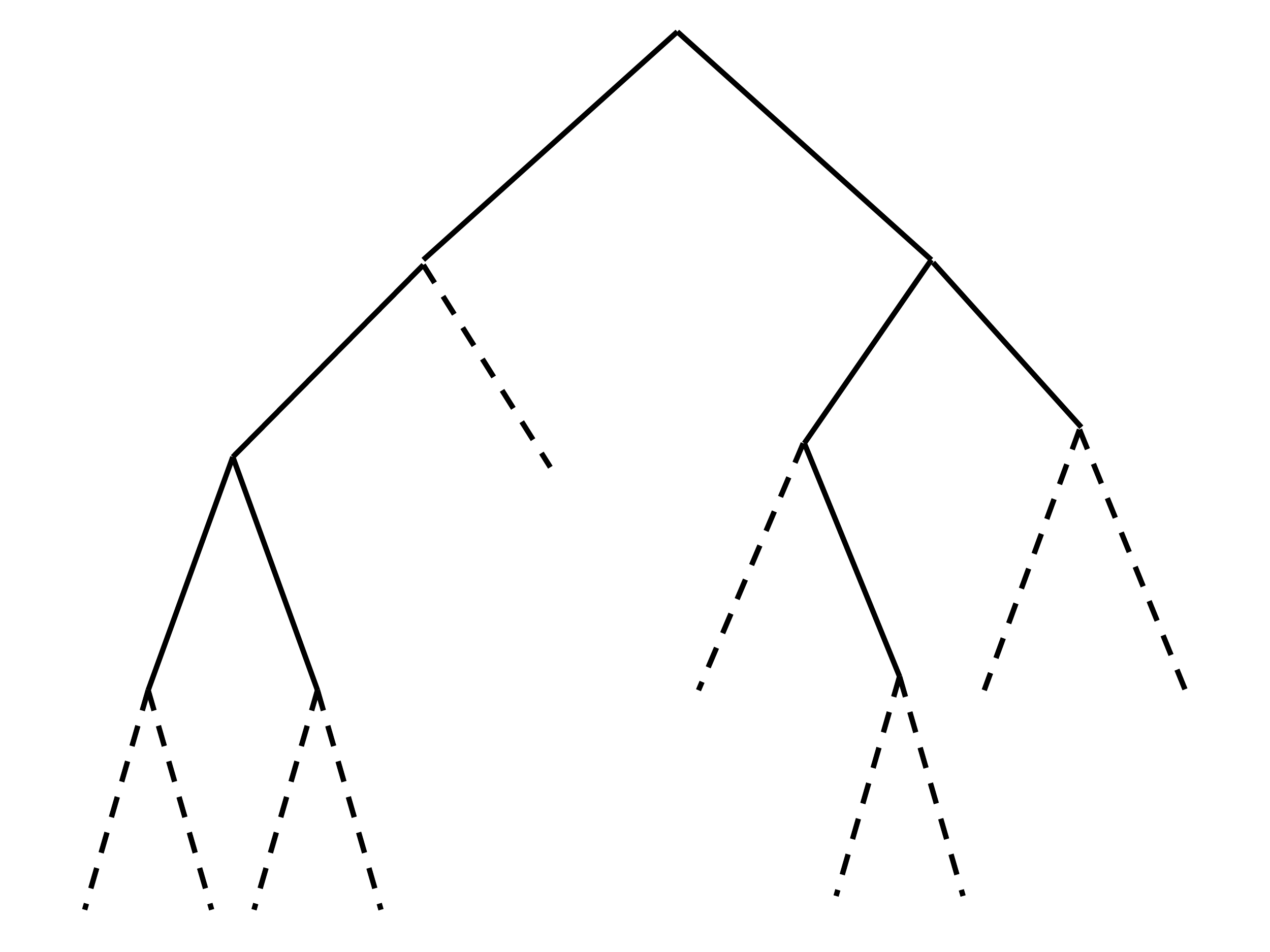}
	\caption{A finite rooted binary tree, the tree with $9$ vertices
	connected by the solid edges, and its $10$ external vertices, the vertices
	connected to the tree by dashed edges. For simplicity, the 
coding of the vertices as elements of $\{0,1\}^\star$ is
not shown.}
	\label{fig:binary_tree_external}
\end{figure}

{}From now on we will refer to any Markov chain on the space
of finite rooted binary trees with this  transition mechanism
as ``the'' BST process and denote it by $(T_n)_{n \in \bN}$.

We note in passing that the labeled permutation trees 
$L_1,\ldots,L_{n-1}$ can be reconstructed from $L_n$, but a similar
reconstruction of the history of the process from its current value 
is not possible if we consider the sequence of
labeled trees obtained by labeling
the vertices of the tree in the binary search tree
algorithm with the input values
$x_1,\ldots,x_n$ that created the tree.

Write $G_n$ (respectively, $D_n$)
for the number of vertices in $T_n$ of the
form $0v_2 \ldots v_\ell$ (resp. $1w_2 \ldots w_m$).
That is, $G_n$ and $D_n$ are the sizes of the 
``left'' and ``right'' subtrees in $T_n$ below the root $\emptyset$.
Then, $G_n + 1$ and $D_n + 1$ are, respectively, the number of
``slots'' to the left and to the right of $1$ in the
collection of $n+1$ slots between successive elements or
at either end of the ordered list $\Pi_n^{-1}(1), \ldots, \Pi_n^{-1}(n)$.
It follows that the sequence of pairs $(G_n + 1, D_n + 1)$, $n \in \bN$,
is itself a Markov chain that evolves as the numbers of black and white
balls in a classical P\'olya urn (that is, as the process describing the
successive compositions of an urn
that initially contains one black and one white ball
and at each stage a ball is drawn uniformly
at random and replaced along with a new ball of the same color).
More precisely, conditional on the past up to time $n$, if 
$(G_n + 1, D_n + 1) = (b,w)$,
then $(G_{n+1} + 1, D_{n+1} + 1)$ takes the values $(b+1,w)$ and $(b,w+1)$ with
respective conditional probabilities $\frac{b}{b+w}$ and $\frac{w}{b+w}$.

More generally,  suppose for a fixed vertex $u = u_1\ldots u_k \in \{0,1\}^*$ 
that we write $G_n^u$ (respectively, $D_n^u$) for
the number of vertices in $T_n$ of the form $u_1\ldots u_k 0 v_2\ldots v_\ell$ 
(resp. $u_1\ldots u_k 1 w_2\ldots w_m$).
That is, $G_n^u$ and $D_n^u$ are the sizes of the 
``left'' and ``right'' subtrees in $T_n$ below the vertex $u$.
Put $C_n^u := \#\{v \in T_n : u \le v\}$ and 
$S_r^u = \inf\{s \in \bN : C_s^u = r \}$ for $r \in \bN$; that is,
$S_r^u$ is the first time that the subtree of $T_n$ rooted at $u$ 
has $r$ vertices.
Then, the sequence $(G_{S_r^u}, D_{S_r^u})$, $r \in \bN$, 
obtained by time-changing the sequence 
$(G_n^u, D_n^u)$, $n \in \bN$, so that we only observe it when it
changes state is a Markov chain with the same distribution as 
$(G_n, D_n)$, $n \in \bN$.
 
It follows from this observation that
we may construct the tree-valued process $(T_n)_{n \in \bN}$ from an
infinite collection of independent, identically distributed P\'olya urns, with
one urn for each vertex of the complete binary tree $\{0,1\}^\star$, by
running the urn for each vertex according to a clock that 
depends on the evolution of the
urns associated with vertices that are on the path from the root to
the vertex.

More specifically, we first equip each vertex $u \in \{0,1\}^\star$ 
with an associated independent $\bN_0 \times \bN_0$-valued 
{\em routing instruction} 
process $(Y_n^u)_{n \in \bN_0}$ such that 
$(Y_n^u + (1,1))_{n \in \bN_0}$ evolves like
the pair of counts in a P\'olya urn with an initial composition of one black
and one white ball. Then,
at each point in time we feed in a new particle at the root $\emptyset$.
At time $0$ the particle simply comes to rest at $\emptyset$. At time $1$ the 
root is occupied and so the particle must be routed to either the vertex
$0$ or the vertex $1$, where it comes to rest,
depending on whether the value of $Y_1^\emptyset$ is $(1,0)$
or $(0,1)$.  We then continue on in this way: at time $n \ge 2$ we feed a particle in
at the root $\emptyset$, it is routed to the vertex $0$ or the vertex $1$ depending
on whether the value of $Y_n^\emptyset - Y_{n-1}^\emptyset$ is $(1,0)$
or $(0,1)$, the particle then trickles down through the tree until it
reaches an unoccupied vertex. 
At each stage of the trickle-down, if the particle is routed to a vertex
$u$ that is already occupied, then it moves on to the vertex $u0$ or the vertex $u1$ depending
on whether the value of $Y_{A_n^u}^u - Y_{A_n^u-1}^u$ is $(1,0)$
or $(0,1)$, where $A_n^u$ is the number of particles that have passed through 
vertex $u$ and been routed onwards by time $n$.  
The resulting sequence of trees is indexed by
$\bN_0$ rather than $\bN$, and if we shift the indices by one
we obtain a sequence indexed by $\bN$ that has the same distribution as 
$(T_n)_{n \in \bN}$.

It is well-known (see \cite{MR0176518}) that the Doob-Martin compactification of 
the state space $\bN^2$ of the classical P\'olya urn
results in a Doob-Martin boundary that is homeomorphic to the unit interval $[0,1]$:
a sequence of pairs $((b_n, w_n))_{n \in \bN}$ from $\bN^2$ converges to a point
in the boundary if and only if $b_n + w_n \to \infty$ and $\frac{w_n}{b_n+w_n} \to z$
for some $z \in [0,1]$.  We can, of course, identify $[0,1]$ with the space of probability
measures on a set with two points, say $\{0,1\}$, by identifying $z \in [0,1]$
with the probability measure that assigns mass $z$ to the point $1$.

It is a consequence of results we prove in Section~\ref{S:general_trickle}
that this result ``lifts'' to the binary search tree process:  the Doob-Martin boundary
is homeomorphic to the space of probability measures on
$\{0,1\}^\infty$ equipped with the weak topology corresponding to the
product topology on $\{0,1\}^\infty$ and a sequence $(\ttt_n)_{n \in \bN}$ of 
finite rooted binary trees converges to the boundary point identified with
the probability measure $\mu$ if and only if $\# \ttt_n \to \infty$ and
for each $u  \in \{0,1\}^\star$
\[
\frac{
\#\{v \in \ttt_n : u \le v\}
}
{
\# \ttt_n
}
\to
\mu\{v \in \{0,1\}^\infty : u \le v\},
\]
where we extend the partial order $\le$ on $\{0,1\}^\star$ to 
$\{0,1\}^\star \sqcup \{0,1\}^\infty$ by declaring that two distinct
elements of $\{0,1\}^\infty$ are not comparable and $u \in \{0,1\}^\star$
is dominated by $v \in \{0,1\}^\infty$ if $u$ is a prefix of $v$.

An outline of the remainder of the paper is the following.
In Section~\ref{S:trickle_general} we give a general version of the trickle-down construction
in which the complete rooted binary tree $\{0,1\}^*$ 
is expanded to a broad class of directed acyclic graphs 
with a unique ``root'' vertex and the independent P\'olya urns
at each vertex are replaced by independent Markov chains that keep a running total
of how many particles have been routed onwards to each of the immediate successors
of the vertex.  For example, we could take the graph to be $\bN_0^2$ with directed edges
of the form $((i,j),(i+1,j))$ and $((i,j),(i,j+1))$ (so that the root is $(0,0)$)
and take the Markov chain at vertex $(i,j)$ to correspond to successive 
particles being routed independently with
equal probability to either $((i,j),(i+1,j))$ or $((i,j),(i,j+1))$.  This gives a
process somewhat reminiscent of Sir Francis Galton's {\em quincunx} --
a device used to illustrate the binomial distribution and central limit theorem in which 
successive balls are dropped onto a vertical board with interleaved rows of horizontal pins that
send a ball striking them downwards to the left or right ``at random''.  We illustrate
the first few steps in the evolution of the set of occupied vertices in Figure~\ref{fig:quincunx_evolve}.

\begin{figure}[htbp]
	\centering
		\includegraphics[width=1.00\textwidth]{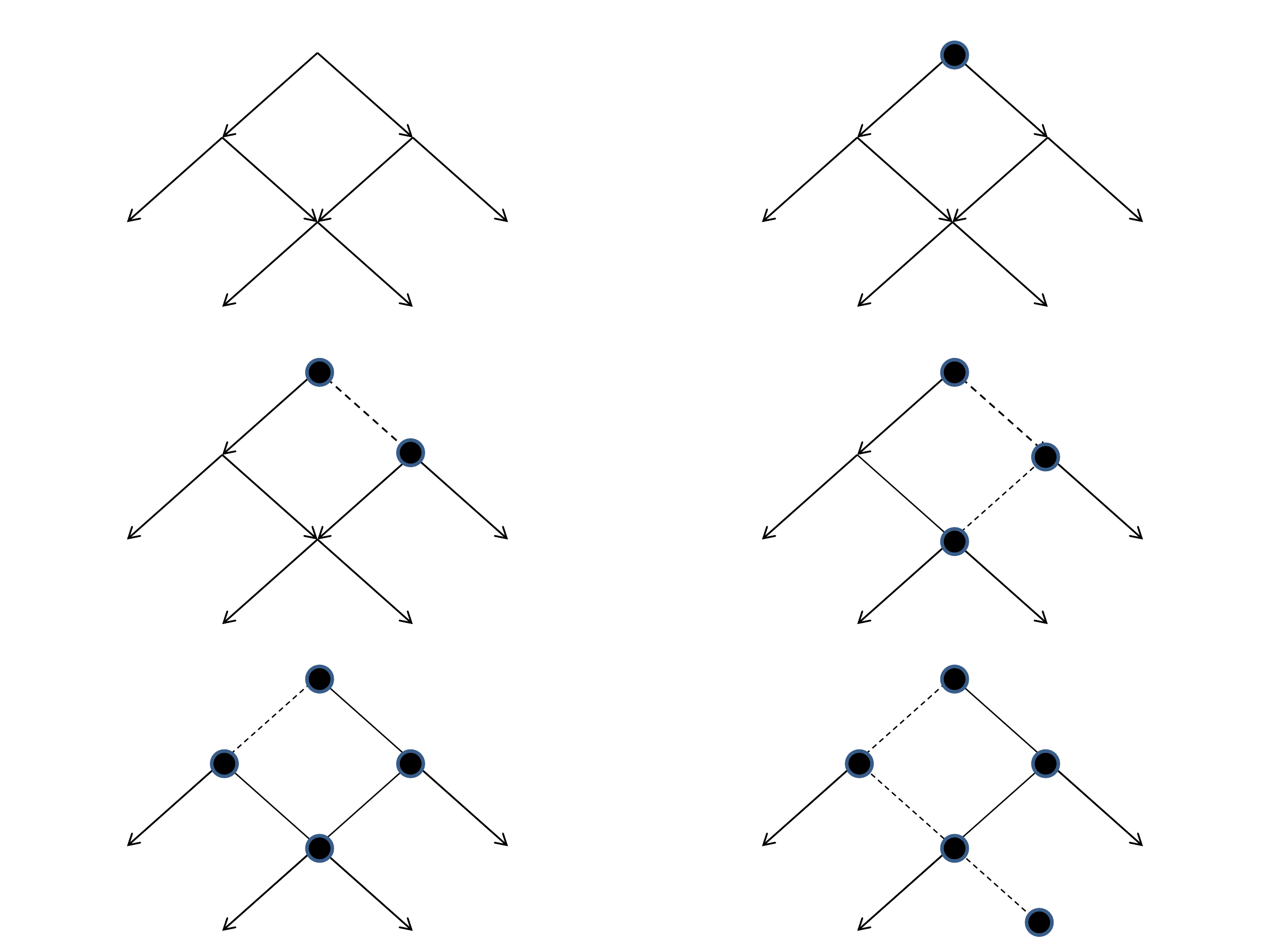}
	\caption{The first five steps in the trickle-down process for the directed acyclic
	graph $\bN_0^2$ with directed edges of the form $((i,j),(i+1,j))$ and $((i,j),(i,j+1))$.  
	The root $(0,0)$ is drawn at the top.  Dashed lines show that paths taken by successive
	particles as they pass through occupied vertices until they come to rest at the
	first unoccupied vertex they encounter.}
	\label{fig:quincunx_evolve}
\end{figure}

We give a brief overview of the theory of Doob-Martin compactifications in
Section~\ref{S:Martin_general}.  We present our main
result, a generalization of the facts 
about the Doob-Martin boundary of the
binary search tree process we have stated above, 
in Section~\ref{S:general_trickle}. 
It says for a large class of trickle-down processes 
that if the convergence of a sequence to a point in the 
Doob-Martin boundary  for each of the component 
Markov chains is determined
by the convergence of the proportions of points that are routed to each of
the immediate successors, then the Doob-Martin boundary of the trickle-down
process is homeomorphic to a space of probability measures on a set
of directed paths from the root that either have infinite length or
are ``killed'' at some finite time.
We then consider special cases of this general result
in Section~\ref{S:BST_and_DST}, where we investigate the binary and digital
search tree processes, and in Section~\ref{S:RRT_and_CRT}, 
where we study random
recursive tree processes that are related to a hierarchy of 
Chinese restaurant processes.

More specifically, we show in Section~\ref{S:BST_and_DST}
that, as we already noted above, the Doob-Martin boundary of the
BST process may be identified with the space of probability measures on
$\{0,1\}^\infty$ equipped with the weak topology corresponding to the
product topology on $\{0,1\}^\infty$, that every boundary point is extremal,
that the {\em digital search tree} process 
is a Doob $h$-transform of the BST process with respect to the extremal harmonic function
corresponding to the fair coin-tossing measure on $\{0,1\}^\infty$, and that
an arbitrary Doob $h$-transform may be constructed from a suitable
``trickle-up'' procedure in which particles come in successively
from the ``leaves at infinity''
of the complete rooted binary tree $\{0,1\}^*$ (that is from $\{0,1\}^\infty$)
and work their way up the tree until they can move no further because their path
is blocked by an earlier particle.

We observe  in Section~\ref{S:RRT_and_CRT} that the 
{\em random recursive tree (RRT) process} 
-- see \cite{smythe-mahmoud} for a review --
can be built from the above sequence $(\Pi_n)_{n \in \bN}$ 
of uniform permutations
in a manner analogous to the construction of the BST process by
using a different bijection between permutations and trees.   
The RRT process is also a trickle-down
process similar to the BST process, with the tree $\{0,1\}^*$
replaced by the tree $\bN^*$ and the P\'olya urn routing instructions replaced
by the Markov chain that gives the block sizes in the simplest {\em Chinese
restaurant process} model of growing random partitions.  We extend this
construction to incorporate Pitman's two-parameter family of
Chinese restaurant processes and then investigate the 
associated Doob-Martin compactification.  We identify the Doob-Martin boundary
as a suitable space of probability measures, show that all boundary points
are extremal, demonstrate that $h$-transform processes may be constructed
via a ``trickle-up'' procedure similar to that described above for the
BST process, and relate the limit distribution to the 
{\em Griffiths--Engen--McCloskey (GEM) distributions}.  Similar nested
hierarchies of Chinese restaurant processes appear in
\cite{MR2288702, MR2561439} and in \cite{MR2279480, MR2606082} in the statistical
context of mixture models, hierarchical models, and nonparametric Bayesian inference.

%Sections \ref{S:Mallows}-\ref{S:Catalan_trees} treat situations that are not covered by the main result in Section~\ref{S:general_trickle}.

A commonly used probability distribution on the set of permutations of a finite
set is the Mallows $\phi$ model --
see \cite{MR0087267, MR818986, MR876847, MR964069, MR1128236, MR1346107} --
for which the uniform distribution is a limiting case.  This distribution extends
naturally to the set of permutations of $\bN$, and applying the obvious generalization
of the above bijection between finite permutations and labeled finite rooted subtrees of the
complete rooted binary tree $\{0,1\}^\star$ leads to an interesting probability
distribution on infinite rooted subtrees of $\{0,1\}^\star$.  In Section~\ref{S:Mallows}
we relate this distribution to yet another model for growing random finite trees that we
call the {\em Mallows tree process}.  We show that the Doob-Martin boundary
of this Markov chain is a suitable space of infinite rooted subtrees of $\{0,1\}^\star$.
We outline a parallel analysis in Section~\ref{S:q_chain} for a somewhat similar process
that is related to Sch\"utzenberger's non-commutative $q$-binomial theorem and its
connection to weighted enumerations of ``north-east'' lattice paths.

The routing instruction processes that appear in the trickle-down construction of the Mallows
tree process have the feature that if we know the state of the chain at some time, then we
know the whole path of the process up to that time.  We observe in Section~\ref{sec:memory}
that such processes may be thought of as Markov chains on a rooted tree with transitions that
always go to states that are one step further from the root.  As one might expect, the
Doob-Martin compactification in this case is homeomorphic to the usual end compactification
of the tree. We use this observation to describe the Doob-Martin compactification of a certain Markov
chain that takes values in the set of compositions of the integers and whose value at time
$n$ is uniformly distributed over the compositions of $n$.

As we have already remarked, 
our principal reason for studying the Doob-Martin compactification 
of a trickle-down chain is to determine the
chain's tail $\sigma$-field.  The Doob-Martin compactification gives
even more information about the asymptotic behavior of the chain, but
it is not always easy to compute.  We describe another approach to
determining the tail $\sigma$-field of certain trickle-down
chains in Section~\ref{S:freezing}.  That result applies to the
Mallows tree process and the model related 
to the non-commutative $q$-binomial theorem.  We also apply it
in Section~\ref{S:Catalan_trees} to yet another Markov chain model
of growing random trees from \cite{MR2060629}.   The latter model, which
turns out to be of the trickle-down type, has
as its state space the set of finite rooted binary trees and is
such that if it is started at time $0$ in the trivial tree $\{\emptyset\}$, then 
the value of the process at time $n$ is equally likely to be any of the
$C_n$ rooted binary trees with $n$ vertices, where $C_n:=\frac{1}{n+1}\binom{2n}{n}$
is the {\em $n^{\mathrm{th}}$ Catalan number}.  Even though we cannot
determine the Doob-Martin compactification of this chain, we are able to show
that its tail $\sigma$-field is generated by 
the random infinite rooted subtree of the complete binary
tree that is the (increasing) union of the
successive values of the chain. Also, knowing the tail $\sigma$-field allows us to
identify the Poisson boundary -- see Section~\ref{S:Martin_general} for a definition
of this object.

We observe that there is some similarity between the trickle-down
description of the binary search tree process and the
{\em internal diffusion limited aggregation model} that
was first named as such in \cite{MR1188055} after it was introduced in
\cite{MR1218674}.  There particles are  fed successively into a fixed
state of some Markov chain and they then execute independent copies of the
chain until they  come to rest at the first unoccupied state they encounter.
The digital search tree  process that we discuss in Section~\ref{S:BST_and_DST}
turns out to be internal diffusion limited aggregation model for the Markov chain
on the complete rooted binary tree that from the state $u$ moves to the states $u0$
and $u1$ with equal probability.

Finally, we note that there are a number of other papers that
investigate the Doob-Martin boundary of Markov chains
on various combinatorial structures such as Young diagrams and partitions --
see, for example, \cite{MR1331221, MR1609628, MR1747061, MR2160320, MR2211157, MR2274860}.

\section{The trickle-down construction}
\label{S:trickle_general}

\subsection{Routing instructions and clocks}
\label{SS:routing}

We begin by introducing a class of directed graphs with features
generalizing those of the complete binary tree $\{0,1\}^\star$
considered in the Introduction.

Let $\II$ be a countable directed acyclic graph.
With a slight abuse of notation, write $u \in \II$
to indicate that $u$ is a vertex of $\II$.
Given two vertices $u,v \in \II$, write $u \rightarrow v$
if $(u,v)$ is a directed edge in $\II$.

Suppose that there is a unique vertex $\hat 0$ such that
for any other vertex $u$ there is at least one finite {\em directed path}
$\hat 0 = v_0 \rightarrow v_1 \rightarrow \ldots \rightarrow v_n = u$
from $\hat 0$ to $u$.
Define a {\em partial order} on $\II$ by declaring that $u \le v$
if $u=v$ or there is a finite directed path 
$u = w_0 \rightarrow w_1 \rightarrow \ldots \rightarrow w_n = v$.
Note that $\hat 0$ is the unique minimal element of $\II$.
Suppose further that the number of directed paths between any
two vertices is finite: this is equivalent to supposing
that the number of directed paths between $\hat 0$ and
any vertex is finite.

For each vertex $u \in \II$,
set
\[
\alpha(u) := \{v \in \II : v \rightarrow u\}
\]
and
\[
\beta(u) := \{v \in \II : u \rightarrow v\}.
\]
That is, $\alpha(u)$ and $\beta(u)$ are, respectively,
the immediate predecessors and the immediate successors of $u$.
Suppose that $\beta(u)$ is non-empty for all $u \in \II$.
Thus, any path $\hat 0 = v_0 \rightarrow v_1 \rightarrow \ldots \rightarrow v_n = u$ is the initial piece of a semi-infinite path 
$v_0 \rightarrow v_1 \rightarrow \ldots \rightarrow v_n 
\rightarrow v_{n+1} \rightarrow \ldots$

We next introduce the notion of {\em routing instructions} that
underlies the construction of a sequence of connected
subsets of $\II$ via a trickle-down mechanism analogous to
that described in the Introduction for the BST: 
at each point in time a particle
is fed into $\hat 0$ and trickles down through $\II$ according
to the routing instructions at the occupied vertices it encounters
until it finds a vacant vertex to occupy.

Let $(\bN_0)^{\beta(u)}$ be the space of functions on the set of successors
of $u\in\II$ that take values in the non-negative integers. Let $e_v$, 
$v\in \beta(u)$, be the function that takes the value $1$ at $v$ and $0$
elsewhere.  That is, if we regard $e_v$ as a vector indexed by $\beta(u)$,
then $e_v$ has $1$ in the $v^{\mathrm{th}}$ coordinate and $0$ elsewhere.
Formally, a routing instruction for the vertex $u \in \mathbf I$ is a
sequence $(\sigma_n^u)_{n \in \mathbb N_0}$ of elements of $(\mathbb N_0)^{\beta(u)}$
with the properties:
\begin{itemize}
\item
$\sigma_0^u = (0,0,\ldots)$, 
\item
for each $n \ge 1$,
$\sigma_{n}^u = \sigma_{n-1}^u + e_{v_n}$ for some $v_n \in \beta(u)$.
\end{itemize}
The interpretation of such a sequence is that, for each $v \in \beta(u)$, the component $(\sigma_n^u)^v$ counts the number
of particles out of the first $n$ to pass through the vertex $u$ that 
are routed onwards to vertex $v \in \beta(u)$.  The equation 
$\sigma_{n}^u = \sigma_{n-1}^u + e_{v_n}$ indicates that the $n^{\mathrm{th}}$ 
such particle is routed onwards to the vertex $v_n \in \beta(u)$.

For $s = (s^v)_{v \in \beta(u)}\in (\mathbb N_0)^{\beta(u)}$ we put
\begin{equation}
\label{E:def_norm}
|s| := \sum_{v \in \beta(u)} s^v.
\end{equation}
Note that  a routing instruction 
$(\sigma_n^u)_{n \in \mathbb N_0}$ for the vertex
$u$ satisfies $|\sigma_{n}^u| = n$ for all $n \in \mathbb N_0$.

For each vertex $u \in \mathbf I$, suppose that we have  a non-empty set $\Sigma^u$ 
of routing instructions for $u$. Put $\Sigma := \prod_{u \in \mathbf I} \Sigma^u$.  Depending
on convenience, we write a generic element
of $\Sigma$ in the form $((\sigma_n^u)_{n \in \mathbb N_0})_{u \in \mathbf I}$ or 
the form $((\sigma^u(n))_{n \in \mathbb N_0})_{u \in \mathbf I}$.
Recall that $\sigma_n^u = \sigma^u(n)$ is an element of $(\mathbb N_0)^{\beta(u)}$,
and so it has coordinates $(\sigma_n^u)^w = (\sigma^u(n))^w$ for $w \in \beta(u)$.

Given $\sigma \in \Sigma$, 
each vertex $u$ of $\mathbf I$ has an associated {\em clock} 
$(a_n^u(\sigma))_{n \in \mathbb N_0}$ such that
$a_n^u(\sigma)$ counts the number of particles
that have passed through $u$ by time $n$
and been routed onwards to some vertex in $\beta(u)$. 
For each $n\in\bN$ and $\sigma\in\Sigma$ the integers $a_n^u(\sigma)$, 
$u\in{\mathbf I}$, are defined recursively (with respect 
to the partial order on $\II$) as follows: 
%The family $a_n: \Sigma \rightarrow (\mathbb N_0)^\mathbf I$
%is defined inductively as follows:
\begin{itemize}
\item[(a)]
$a_n^{\hat 0}(\sigma) := n$,
%for all $\sigma \in \Sigma$ and $n \in \mathbb N_0$,
\item[(b)]
$a_n^u(\sigma) := (\sum_{v \in \alpha(u)} (\sigma^v(a_n^v(\sigma)))^u - 1)_+$, $u \ne \hat 0$.
\end{itemize}
In particular,
$a_0(\sigma) = (0,0,\ldots) \mbox{ for all } \sigma \in \Sigma$.
The equation in (b) simply says that the number of particles that have been
routed onwards from the vertex $u$ by time $n$ 
is equal to the number of particles that have passed through
vertices $v$ with $v \rightarrow u$ and have been 
routed in the direction of $u$, excluding the
first particle that reached the vertex $u$ and occupied it.

We say that the sequence $(x_n)_{n \in \mathbb N_0} = ((x_n^u)_{u \in \mathbf I})_{n \in \mathbb N_0}$ 
given by
\begin{equation}
\label{E:def_trickle_down}
x_n^u := \sigma^u(a_n^u(\sigma))
\end{equation}
is the 
{\em result of the trickle-down construction for the routing instruction} 
$\sigma \in \Sigma$.

\begin{example}
\label{E:quincunx}
Suppose that the directed graph $\II$ has $\bN_0^2$ as its set of vertices and directed edges
of the form $((i,j),(i+1,j))$ and $((i,j),(i,j+1))$. The root is $(0,0)$.
\begin{itemize}
\item[(a)]
Figure~\ref{fig:quincunx} shows the state at time $n=12$
(that is, the values of $x_{12}^u$ for $u = (i,j) \in \II = \bN_0^2$)
generated by routing instructions whose initial pieces are 
\[
\sigma_1^{(0,0)}=(0,1),
\ \sigma_2^{(0,0)}=(0,2), 
\ \sigma_3^{(0,0)}=(1,2), 
\ \sigma_4^{(0,0)}=(2,2),
\]
\[
\sigma_1^{(0,1)}=(0,1), 
\ \sigma_1^{(1,0)}=(1,0), 
\ \sigma_2^{(1,0)}=(2,0), 
\ \sigma_1^{(1,1)}=(0,1),
\]
when the states $(i+1,j)$ and $(i,j+1)$
that comprise $\beta(u)$, the immediate successors of $u$,  
are taken in that order.
\item[(b)]
The clock $a^{(0,1)}$, which translates from ``real time'' to the ``local time'' at the vertex $(0,1) \in \II = \bN_0^2$ by counting the particles that pass through this vertex, has a corresponding sequence of states that begins
$a^{(0,1)}_0 = a^{(0,1)}_1=a^{(0,1)}_2=a^{(0,1)}_3 = 0$, $a^{(0,1)}_4 = a^{(0,1)}_5 =1$.
\item[(c)]
The configuration $x_5$ consists of a pair $x_5^u = x_5^{(i,j)} \in \bN_0^2$
for every $u = (i,j) \in \II = \bN_0^2$.  Each such pair records the
onward routings by time $5$ to the immediate successors
$\beta(u) = \{(i+1,j), (i,j+1)\}$ of $u$.  Following
through the construction gives $x^{(0,0)}_5=(2,2)$, 
$x^{(1,0)}_5 = (0,1)$, $x^{(0,1)}_5 = (2,0)$, 
$x^{(1,1)}_5 = (0,1)$, with all the other components of 
$x_5$ being $(0,0)$. For example, the value $x^{(0,1)}_5 = (2,0)$
indicates that by time $5$ the 
vertex $(0,1)$ has been occupied, $2$ particles have been
sent onwards to the vertex $(1,1)$, and $0$ particles
have been sent onwards to the other immediate successor $(0,2)$.
\item[(d)]
Looking at the state $x_{12}^u$, $u \in \II$, at time $n=12$ we cannot
reconstruct the relevant initial segments of the routing instructions 
but we can see, for example, that
\begin{itemize}
\item
$13$ particles have been fed into the root $(0,0)$: the first of these stayed at the
root, $6$ of the remainder were routed onwards to $(1,0)$ and the other $6$ 
were routed onwards to $(0,1)$ (that is, $a_{12}^{(0,0)}(\sigma) = 12$ and
$\sigma_{12}^{(0,0)} = (6,6)$);
\item 
of the $6$ particles routed from the root towards $(1,0)$, the first stayed there, $2$ of the
remainder were routed onwards to $(2,0)$ and the other $3$ were routed onwards to $(1,1)$
(that is, $a_{12}^{(1,0)}(\sigma) = 5$ and $\sigma_{5}^{(1,0)} = (2,3)$);
\item
of the $6$ particles routed from the root towards $(0,1)$, the first stayed there, $3$ of the
remainder were routed onwards to $(1,1)$ and the other $2$ were routed onwards to $(0,2)$
(that is, $a_{12}^{(0,1)}(\sigma) = 5$ and $\sigma_{5}^{(0,1)} = (3,2)$).
\end{itemize}
\end{itemize}
\end{example}

\begin{figure}[htbp]
	\centering
		\includegraphics[width=1.00\textwidth]{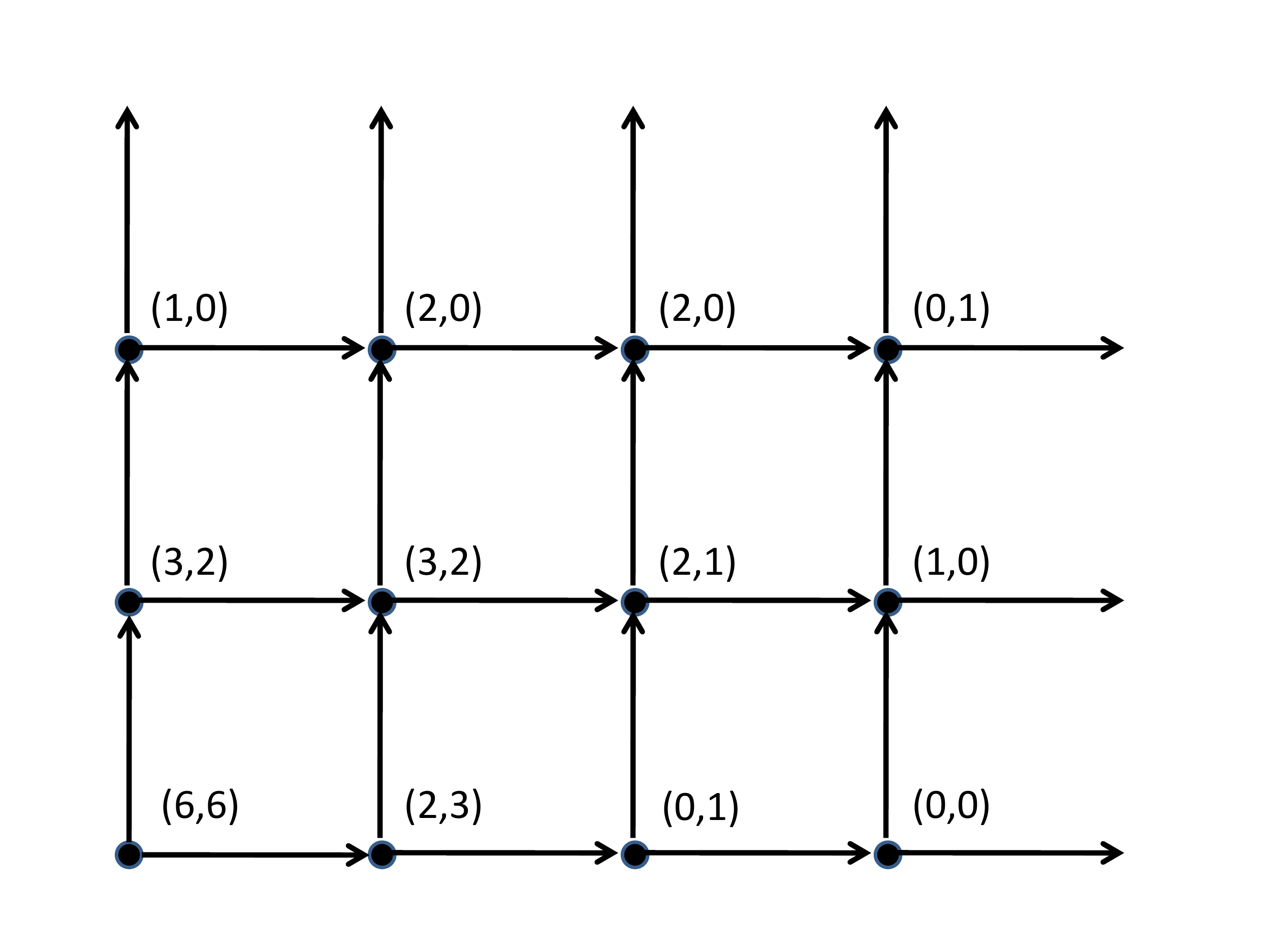}
	\caption{A possible result of the trickle-down construction at time $n=12$ on $\II = \bN_0^2$.  See the text for details.}
	\label{fig:quincunx}
\end{figure}

For each vertex $u \in \mathbf I$, write
$\mathbf S^u \subseteq (\mathbb N_0)^{\beta(u)}$ for the set of vectors that
can appear as an entry in an element of $\Sigma^u$.  
That is, $s \in \mathbf S^u$ if
and only if $s = \sigma_m$ for some sequence $(\sigma_n)_{n \in \mathbb N_0} \in \Sigma^u$,
where, of course, $m = |s|$.  
Note that the set $\mathbf S^u$ is countable.

Let $\SSS$ denote the subset of $\prod_{u \in \II} \SSS^u$ 
consisting of points $x = (x^u)_{u \in \II}$ that can be constructed as 
$(x^u)_{u\in \II} = (\sigma^u(a_m^u(\sigma)))_{u\in \II}$ for some $m \in \bN_0$ and
some $\sigma = ((\sigma_n^v)_{n \in \bN_0})_{v \in \II} \in \Sigma$\,;
that is, $x$ appears as the value at time $m$ in the
result of the trickle-down construction for the routing instruction $\sigma$.
Clearly, if a sequence 
$(x^u)_{u \in \II} \in \prod_{u \in \II} \SSS^u$ belongs to $\SSS$, then
\begin{equation}
\label{E:consistency_condition}
\left(\sum_{v \in \alpha(u)} (x^v)^u - 1\right)_+
=
\sum_{w \in \beta(u)} (x^u)^w.
\end{equation} 

Given two points $x,y \in \SSS$, say that 
$x \preceq y$ if for some $m,n \in \bN_0$ with $m \le n$ and
some $\sigma \in \Sigma$ we have
$x^u = \sigma^u(a_m^u(\sigma))$ and 
$y^u = \sigma^u(a_n^u(\sigma))$
for all $u \in \II$.

\begin{remark}
\label{R:hit_condition}
Note that if $x \preceq y$, then
$(x^u)^v \le (y^u)^v$ for all $u \in \II$ and $v \in \beta(u)$.
Moreover, if $x \preceq y$,  then
\[
\begin{split}
&  \left \{\sigma \in \Sigma : 
\left (\sigma^u(a_m^u(\sigma))\right)_{u \in \II} = x \; \text{and} \; \left(\sigma^u(a_n^u(\sigma))\right)_{u \in \II} = y 
\; \text{for some $m \le n \in \bN_0$}
\right \} \\
& \quad =
 \left \{\sigma \in \Sigma : 
\left(\sigma^u\left(\sum_{v \in \beta(u)} (x^u)^v\right)\right)_{u \in \II} = x \; \text{and} \; \left(\sigma^u\left(\sum_{v \in \beta(u)} (y^u)^v\right)\right)_{u \in \II} = y 
\right \} \\
& \quad =
\prod_{u \in \II} 
 \left \{\sigma^u \in \Sigma^u : 
\sigma^u\left(\sum_{v \in \beta(u)} (x^u)^v\right) = x^u \; \text{and} \; \sigma^u\left(\sum_{v \in \beta(u)} (y^u)^v\right) = y^u
\right \} \\
& \quad =
\prod_{u \in \II} 
 \left \{\sigma^u \in \Sigma^u : 
\sigma^u(p) = x^u \; \text{and} \; \sigma^u(q) = y^u
\; \text{for some $p \le q \in \bN_0$}
\right \}. \\
\end{split}
\]
\end{remark}

\begin{example}
\label{E:dag_tree}
Suppose that $\II$ is a {\em tree}.  This amounts to imposing the extra condition
that for each vertex $u \in \II$ there is a {\bf unique} directed path from
$\hat 0$ to $u$. For each $u \in \II$ 
take $\Sigma^u$ to be the set of all allowable routing 
instructions for $u$, so that the corresponding set 
$\SSS^u$ is $(\bN_0)^{\beta(u)}$.
In this case, there is a bijection between $\SSS$ and
finite subtrees of $\II$ that contain the root $\hat 0$.  An element
$x \in \SSS$ determines a finite rooted subtree $\ttt$ by
\[
\ttt = 
\{\hat 0\} 
\cup 
\{v \in \II \setminus \{\hat 0\}: (x^u)^v > 0 \; \text{for some $u \in \alpha(v)$}\}.
\]
In other words, the tree $\ttt$ consists of those vertices of $\II$  that are occupied by the first $\sum_{v\in \beta(\hat 0)}(x^{\hat 0})^v$ particles.

Conversely, if $\ttt$ is a finite subtree of $\II$ that contains $\hat 0$, then
the corresponding element of $\SSS$ is
\[
x = \left( \left (\# \{w \in \ttt : v \le w\} \right)_{v \in \beta(u)} \right)_{u \in \II};
\]
that is, $x$ appears as the result of the trickle down
construction at some time $n$ and for each pair of vertices
$u \in \II$ and $v \in \beta(u)$ the integer
$\# \{w \in \ttt : v \le w\}$ gives the number of particles that have been routed onwards
from vertex $u \in \II$ to vertex $v \in \beta(u)$ by time $n$.
The partial order $\preceq$ on $\SSS$ is equivalent to containment of the
associated subtrees.  From now on, when $\II$ 
is a tree we sometimes do not mention
this bijection explicitly and abuse terminology slightly by speaking of $\SSS$ as
the set of finite subtrees of $\II$ that contain the root $\hat 0$.
\end{example}

\begin{example}
\label{E:trail}
In Example~\ref{E:dag_tree}, the set $\SSS^u$ of states
for the routing instructions at any vertex $u \in \II$ is all of $(\bN_0)^{\beta(u)}$. At the other
extreme we have what we call the \emph{single trail} routing: as
always, the first item is put into the root, but now, in the step from $n$
to $n+1$, the new item follows the trail $u_0,\ldots,u_{n-1}$ left by the
last one and then chooses $u_n$ from $\beta(u_{n-1})$. In this case, 
$\SSS^u = \{0\} \sqcup \bigsqcup_{v \in \beta(u)} \bN e_v$,
where $0$ is the zero vector in $(\bN_0)^{\beta(u)}$. 
Examples of this type appear in Section~\ref{sec:memory}.
\end{example}

\begin{remark}
\label{R:dag_tree}
In the setting of Example~\ref{E:dag_tree},
the sequence $(x_n)_{n \in \bN_0}$
in $\SSS$ constructed by setting $x_n^u = \sigma^u(a_n^u(\sigma))$ 
for some $\sigma \in \Sigma$ corresponds to a sequence of
growing subtrees that begins with the trivial tree $\{\hat 0\}$ 
and successively add a single vertex that is connected by a directed edge 
to a vertex present in the current subtree, 
and this correspondence is bijective. 
In Example~\ref{E:trail}, a sequence $(x_n)_{n \in \bN_0}$ in $\SSS$
corresponds to the sequence of initial segments of 
some infinite directed path,
$\hat 0 =u_0\to u_1\to u_2\to\cdots$ through $\II$, 
and this correspondence is also bijective.
\end{remark}

\subsection{Trickle-down chains}
\label{SS:trickle_down_chains}

We now choose the routing instructions randomly in order to
produce an $\SSS$-valued stochastic process.

For each $u \in  \mathbf I$, let $Q^u$ be a transition matrix whose rows and columns are indexed by some subset $\mathbf{R}^u \subseteq (\bN_0)^{\beta(u)}$ such that $(0,0,\ldots) \in \mathbf{R}^u$,  
and $Q^u(s',s'')>0$ for $s',s'' \in \mathbf{R}^u$ implies that
$s'' = s'+e_v$ for some $v \in \beta(u)$.  Let  $\Sigma^u$ be the set of sequences $\sigma^u = (\sigma_n^u)_{n \in \bN_0}$ in  $\mathbf{R}^u$ that satisfy $\sigma_0^u = (0,0,\ldots)$ and
$ Q^u(\sigma_n^u, \sigma_{n+1}^u)>0$ for all $n \in \bN_0$. Then $\Sigma^u$ is a set of routing instructions for the vertex $u$. Define, as in the
previous subsection,
$\mathbf{S}^u$  to be the set of elements of
 $\bN_0^{\beta(u)}$ that can appear as an entry in an element of $\Sigma^u$.
Note that $\mathbf{S}^u \subseteq \mathbf{R}^u $: the set
$\mathbf{S}^u$ consists of the states 
that are reachable by a Markov chain with transition matrix
$Q^u$ started from the state $(0,0,\ldots)$.  We will suppose from now on
that $\mathbf{R}^u = \mathbf{S}^u$.

%Let $\SSS^u$ and $\Sigma^u$ be as above for every $u \in \II$.
%Assume for each $u \in \II$ that there is a transition
%matrix $Q^u$ with rows and columns indexed
%by $\SSS^u$ such that $\sigma = (\sigma_n)_{n \in \bN_0} \in \Sigma^u$ if and only if
%$Q^u(\sigma_n, \sigma_{n+1})>0$ for all $n \in \bN_0$.

Write $(Y_n^u)_{n \in \bN_0}$
for the corresponding $\SSS^u$-valued Markov chain
with its associated collection of probability measures $\bQ^{u,\xi}$, $\xi \in \SSS^u$.  A realization of the process $Y^u$ starting
from the zero vector in $(\bN_0)^{\beta(u)}$ will serve as the routing
instruction  for the vertex $u$; that is, the $n^{\mathrm{th}}$
particle that trickles down to $u$ and finds $u$ occupied will be
routed onward to the immediate successor $v \in \beta(u)$
specified by  $e_v = Y_n^u - Y_{n-1}^u$.
By assumption, and with $0$ the zero vector in $(\bN_0)^{\beta(u)}$, $Y^u$  has positive 
probability under $\bQ^{u,0}$
of hitting any given state in $\SSS^u$. We will refer to $Y^u$ as the {\em routing chain for the vertex} $u$. Let $Y:=(Y^u)_{u\in\II}$, where the component
processes $Y^u$ are independent and have distribution $\bQ^{u,0}$.

With $a_0,a_1,\ldots$ the clocks defined in Section~\ref{SS:routing}, set
\[
A_n
:=
\begin{cases}
a_n(Y), & \text{if $Y_0 = (0,0,\ldots)$},\\
0, & \text{otherwise}.
\end{cases}
\]
Thus,
$(A_n)_{n \in \bN_0}$
is an $(\bN_0)^\II$-valued 
stochastic process with non-decreasing paths and initial
value $(0,0,\ldots)$.  When $Y_0 = (0,0,\ldots)$,
the value of the process $A_n$ at time $n$ is
a vector $(A_n^u)_{u \in \II}$: 
the non-negative integer $A_n^u$ records
the number of particles that have trickled down
to the vertex $u$ by time $n$, found $u$ already occupied,
and have been routed onwards.

Define
\[
Z_n^u := Y_{A_n^u}^u, \quad u \in \II, \, n \in \bN_0.
\]
By construction, $Z := (Z_n)_{n \in \bN_0} = ((Z_n^u)_{u \in \II})_{n \in \bN_0}$ 
is a Markov chain on the countable state space  $\SSS$
under the probability measure $\bigotimes_{u \in \II} \bQ^{u,0}$.
The paths of $Z$ 
start from the state $(0,0,\ldots)$ and
increase strictly in the natural partial order on $\SSS$.  
The random vector $Z_n^u$ gives for each immediate successor 
$v \in \beta(u)$ of $u$ the number of particles that have trickled down to 
$u$ by time $n$,
found $u$ already occupied, and have been routed onwards towards $v$.

By standard arguments, we can construct a measurable  space $(\Omega, \cF)$,
a family of probability measures $(\bP^x)_{x \in \SSS}$
and an $\SSS$-valued stochastic process $X = (X_n)_{n \in \bN_0}$
such that $X$ under $\bP^x$ is a Markov chain with $X_0 = x$
and the same transition mechanism as $Z$.

\begin{remark}
\label{R:hereditary}
Note that if
$\JJ$ is a subset of $\II$ with the property that 
$\{v \in \II : v \le u\} \subseteq \JJ$ for all $u \in \JJ$, 
then $((X_n^u)_{u \in \JJ})_{n \in \bN_0}$ is
a Markov chain under $\bP^x$.  
Moreover, the law of the latter process under $\bP^x$
agrees with its law under $\bP^y$ for any $y \in \SSS$ with
$x^u = y^u$ for all $u \in \JJ$.  
\end{remark}

\section{Doob-Martin compactification background}
\label{S:Martin_general}

We restrict the following sketch of Doob-Martin 
compactification theory for discrete time Markov chains
to the situation of interest in the present paper. The
primary reference is \cite{MR0107098}, but useful reviews
may be found in 
\cite[Chapter 10]{MR0407981},
\cite[Chapter 7]{MR0415773}, 
\cite{MR1463727},
\cite[Chapter IV]{MR1743100},
\cite[Chapter III]{MR1796539}.

Suppose that  $(X_n)_{n \in \bN_0}$ is a discrete time Markov chain
with countable state space $E$ and transition matrix $P$.  Define
the {\em Green kernel} or {\em potential kernel} $G$ of
$P$ by $G(i,j) := \sum_{n=0}^\infty P^n(i,j)$ for
$i,j \in E$ and assume that there is a reference state $e \in E$
such that $0 < G(e,j) < \infty$ for all $j \in E$.  This implies
that any state can be reached from $e$ and that every state is
transient.  For the chains to which we apply the theory, the
state space $E$ is  a partially ordered set with unique
minimal element $e$ and transition matrix
$P$ such that $P(k,\ell) = 0$ unless $k < \ell$, so that the
sample paths of the chain are {\em increasing} and
\[
G(i,j) 
= \bP^i\{X_n = j \; \text{for some $n \in \bN_0$}\}
=:\bP^i\{\text{$X$  hits $j$}\}
\]
for all $i,j \in E$.

A function $f: E \rightarrow \bR_+$ is said to be {\em excessive}
(respectively, {\em regular}) if 
$\sum_{j \in E} P(i,j) f(j) =: P f(i) \le f(i)$ for all
$i \in E$ (respectively,
$P f(i) = f(i)$ for all $i \in E$).  
Excessive functions are also called {\em non-negative
superharmonic functions}.  Similarly, regular functions
are also called {\em non-negative harmonic functions}. 
Given a finite measure $\mu$ on $E$, define a function $G \mu: E \to \bR_+$ by 
$G \mu(i) := \sum_{j \in I}
G(i,j) \mu(\{j\})$ for $i \in E$.
The function $G \mu$ is excessive and
is called the {\em potential} of the measure $\mu$.
The Riesz decomposition  says that
any excessive function $f$ has a unique decomposition
$f = h + p$, where $h$ is regular and $p = G \nu$ is the potential
of a unique measure $\nu$.

Note for any excessive function $f$ that
$f(e) \ge \sup_{n \in \bN_0} P^n(e,j) f(j)$,
and so $f(e) = 0$ implies that $f = 0$.  
Therefore, any excessive function is a constant multiple of an
element of the set $S$ of excessive functions that take the value
$1$ at $e$.  The set $S$ is a compact convex metrizable subset of
the locally convex topological vector space $\bR^E$.

The {\em Martin kernel with reference state $e$} is given by
\[
   K(i,j) := \frac{G(i,j)}{G(e,j)}= 
             \frac{\bP^i\{\text{$X$  hits $j$}\}}{\bP^e\{\text{$X$  hits $j$}\}};
\]
that is, $K(\cdot,j)$ is the potential of the unit point mass at $j$ normalized
to have value $1$ at the point $e \in E$.
For each $j \in E$ the function $K(\cdot,j)$ belongs to $S$
and is non-regular.
Moreover, $K(\cdot,j)$ is an extreme point of $S$ and 
any extreme point of $S$ that is not of the form
$K(\cdot,j)$ for some $j \in E$ is regular.
It also follows from the Riesz decomposition that the map
$\phi: E \to S$ given by $\phi(j) := K(\cdot,j)$
is injective.  Therefore, we can identify $E$ with its image
$\phi(E) \subset S$ that sits densely inside the compact closure
$F$ of $\phi(E)$ in $S$.  With the usual slight abuse of terminology,
we treat $E$ as a subset of $F$ and use the alternative
notation $\bar E$ for $F$.  The construction of the
compact metrizable space $\bar E$ from $E$
using the transition matrix $P$ and the reference state $e$
is the {\em Doob-Martin compactification} of $E$ and the set
\[
\partial E := \bar E \setminus E
\]
is the {\em Doob-Martin boundary} of $E$.

By definition, a sequence $(j_n)_{n \in \bN}$ in $E$ converges
to a point in $\bar E$
if and only if the sequence of real numbers
$(K(i,j_n))_{n\in\bN}$ converges for all $i\in E$. Each function
$K(i,\cdot)$ extends continuously to $\bar E$ and we call the
resulting function $K: E \times \bar E \rightarrow \bR$
the {\em extended Martin kernel}.

The set of extreme points $F_{\mathrm{ex}}$ of the convex set $F$ is a $G_\delta$
subset of $F$ and any regular function $h \in S$ (that is,
any regular function $h$ with $h(e) = 1$) has the representation
\[
h = \int K(\cdot, y) \, \mu(dy)
\]
for some unique probability measure on $F$ that assigns all of
its mass to $F_{\mathrm{ex}} \cap E^c \subseteq \partial E$.

The primary probabilistic consequence of the Doob-Martin compactification
is that for any initial state $i$ the limit 
$X_\infty:=\lim_{n \rightarrow \infty} X_n$ exists $\bP^i$-almost surely in the topology of $F$
and the limit belongs to $F_{\mathrm{ex}}$, 
$\bP^i$-almost surely.

If $h$ is a regular function (not identical to $0$), then the corresponding
{\em Doob $h$-transform} is the Markov chain $(X_n^{(h)})_{n \in \bN_0}$
with state space $E^h := \{i \in E : h(i) > 0\}$ and transition matrix
\[
P^{(h)}(i,j) := h(i)^{-1} P(i,j) h(j), \quad i,j \in E^h.
\]
When $h$ is strictly positive, 
the Doob-Martin compactification of $E$ and its set of extreme points are the
same for $P$ and $P^{(h)}$.  

The regular function $h$ is extremal 
if and only if the limit  $\lim_{n \rightarrow \infty} X_n^h$
is almost surely equal to a single point $y$ for some
$y \in F$, in which case $y \in F_{\mathrm{ex}}\cap E^c$ and $h = K(\cdot,y)$.
In particular, $h$ is extremal if and only if the tail $\sigma$-field
of $(X_n^{(h)})_{n \in \bN_0}$ is trivial.
In this case, the transformed chain $(X_n^{(h)})_{n \in \bN_0}$ may be thought of as the
original chain $(X_n)_{n \in \bN_0}$ conditioned to converge to $y$.
The original chain is a mixture of such conditioned chains, where
the mixing measure is the unique probability measure $\nu$ 
supported on $F_{\mathrm{ex}} \cap E^c \subseteq \partial E$ such that
$1 = \int K(\cdot, y) \, \nu(dy)$. Further, $\nu$ is the distribution of $X_\infty$
under $\bP^e$.

The Doob-Martin boundary provides a representation of the non-negative harmonic
functions.  We close this review section with a brief discussion of a measure 
theoretic boundary concept that has a more direct relation to tail $\sigma$-fields 
in the trickle-down case.
 
The set $\bH$ of all \emph{bounded} harmonic functions
is a linear space and indeed a Banach space when endowed with the supremum norm.
The \emph{Poisson boundary} is a measure space $(M,\cA,\mu)$ with the property
that $L^\infty(M,\cA,\mu)$ and $\bH$ are isomorphic as Banach spaces.
The Doob-Martin boundary $\partial E$ together with its Borel $\sigma$-field 
and the distribution $\nu$ of $X_\infty$ under $\bP^e$ provides such a measure space. 

Our models have the specific feature that, loosely speaking, `time is a function
of space': the state space $E$ of a trickle-down chain $(X_n)_{n\in\bN_0}$ may be
written as the disjoint union of the sets
\begin{equation*}
   E_n  :=  \{x\in E:\, \bP^e\{X_n=x\}>0\}.
\end{equation*}
Let $\cT$ be the tail $\sigma$-field of the chain. Consider now the map that takes 
a bounded, $\cT$-measurable random variable $Z$ to the function $h:E\to\bR$ defined by
\begin{equation*}
  h(x) :=  \frac{1}{\bP^e\{X_n=x\}}\int_{\{X_n=x\}} Z\, d\bP^e,
\end{equation*}
for all $x\in E_n$, on each $E_n$ separately. Note that $h(X_n)=\bE^e[Z|X_n]$. 
Using martingale convergence and the Markov property, it follows that this map is a 
Banach space isomorphism between $L^\infty(\Omega,\cT,\bP^e)$ and $\bH$. 

For any embedding in which the chain converges to a limit $X_\infty$, 
this limit is $\cT$-measurable.  The limit in the Doob-Martin
compactification of a transient chain generates the invariant $\sigma$-field 
up to null sets,
where for a chain $(X_n)_{n \in \bN_0}$ with state space $E$, an event $A$ is invariant if
there is a product measurable subset $B \subseteq E^{\bN_0}$ such that for all $n \in \bN_0$
the symmetric difference $A\, \triangle \,\{(X_n, X_{n+1}, \ldots) \in B\}$ has zero 
probability.
In our models, the limit $X_\infty$ in the Doob-Martin compactification generates the tail
$\sigma$-field, because it is possible to reconstruct the value of the time parameter from
the state of the process at an unspecified time. 
Conversely, from the tail $\sigma$-field we may
obtain the Poisson boundary but not, in general,  the Doob-Martin boundary.

\section{Compactification for trickle-down processes}
\label{S:general_trickle}

For each $ u \in \mathbf I$, let $Q^u$ be a transition matrix on
$\mathbf S^u \subseteq \mathbb N_0^{\beta(u)}$ 
with the properties described in Section \ref{SS:trickle_down_chains}.
The following result is immediate from the construction of the trickle-down
chain $X$ and Remark~\ref{R:hit_condition}.

\begin{lemma}
\label{L:hit_prob_is_product}
Consider elements $x = (x^u)_{u \in \II}$
and $y = (y^u)_{u \in \II}$ of $\SSS$.  
Write $m^u = \sum_{v \in \beta(u)} (x^u)^v$
and 
$n^u = \sum_{v \in \beta(u)} (y^u)^v$.
Then,
\[
\bP^x\{X \; \text{hits} \; y\}
=
\prod_{u \in \II}
\bQ^{u,x^u}\{Y_{n^u - m^u}^u = y^u\}
=
\prod_{u \in \II}
\bQ^{u,x^u}\{Y^u \; \text{hits} \; y^u\}.
\]
The product is zero unless $x \preceq y$
(equivalently, $x^u \le y^u$ for all $u \in \II$).
Only finitely many terms in the
product differ from $1$, because
$x^u = y^u = (0,0,\ldots)$ (equivalently, $m^u = n^u = 0$)
for all but finitely many values of $u \in \II$.
\end{lemma}

\begin{corollary}
\label{C:Martin_kernel_is_product}
The Martin kernel of the Markov chain $X$ with respect 
to the reference state $\hat 0$ is given by
\[
K(x,y) 
=
\prod_{u \in \II} K^u(x^u, y^u),
\]
where $K^u$ is the Martin kernel of the Markov chain $Y^u$
with respect to reference state $(0,0,\ldots) \in \SSS^u$.
The product is zero unless $x \preceq y$
(equivalently, $x^u \le y^u$ for all $u \in \II$).
Only finitely many terms in the product differ from $1$, because
$x^u = (0,0,\ldots)$
for all but finitely many values of $u \in \II$.
\end{corollary}

\begin{proof}
It suffices to note that
\[
K(x,y) = \frac{\bP^x\{X \; \text{hits} \; y\}}{\bP^{\hat 0}\{X \; \text{hits} \; y\}}
\]
and
\[
K^u(\xi,\zeta) = \frac{\bQ^{u,\xi}\{Y^u \; \text{hits} \; \zeta\}}{\bQ^{u,0}\{Y^u \; \text{hits} \; \zeta\}},
\]and then apply Lemma~\ref{L:hit_prob_is_product}.
\end{proof}

\begin{example}
\label{E:BST}
Consider the BST process from the Introduction.  Recall
that in this case the directed graph $\II$ is the complete binary tree
$\{0,1\}^\star$ and each of the processes $(Y_n^u + (1,1))_{n \in \bN_0}$
is the classical P\'olya urn in which we have an urn consisting
of black and white balls, we draw a ball uniformly at random
at each step and replace it along with one of the same
color, and we record the number of black and white balls
present in the urn at each step.  Note that if we start
the P\'olya urn with $b$ black and $w$ white balls,
then the probability that we ever see $B$ black balls
and $W$ white balls is the probability that after
$(B+W)-(b+w)$ steps we have added $B-b$ black balls
and $W-w$ white balls.  The probability of adding the extra balls
in a particular specified order is
\[
\frac{
b(b+1) \cdots (B-1) w(w+1) \cdots (W-1)
}
{
(b+w)(b+w+1) \cdots (B+W-1)
}
\]
(the fact that this probability is the same for all orders
is the fundamental exchangeability fact regarding the P\'olya urn).
The probability of adding the required extra balls of each color
in some order is therefore
\[
\frac{((B+W)-(b+w))!}{(B-b)! (W-w)!}
\frac{
b(b+1) \cdots (B-1) w(w+1) \cdots (W-1)
}
{
(b+w)(b+w+1) \cdots (B+W-1)
}.
\]
Hence,
\[
\begin{split}
& \bQ^{u,\xi}\{Y^u \; \text{hits} \; \zeta\} \\
& \quad =
\frac
{
((\zeta^{u0} + \zeta^{u1}) - (\xi^{u0} + \xi^{u1}))!
}
{
(\zeta^{u0} - \xi^{u0})! (\zeta^{u1} - \xi^{u1})!
}
\frac
{
(\xi^{u0} +1) \ldots \zeta^{u0} \times (\xi^{u1} +1) \ldots \zeta^{u1}
}
{
(\xi^{u0} + \xi^{u1} + 2) (\xi^{u0} + \xi^{u1} + 1) \ldots (\zeta^{u0} + \zeta^{u1} + 1)
} \\
\end{split}
\]
for $\xi \le \zeta$, and so
\[
\begin{split}
K^u(\xi,\zeta)
&  =
\frac
{
(\xi^{u0} + \xi^{u1} + 1)!
}
{
\xi^{u0}! \xi^{u1}!
}
\frac
{
(\zeta^{u0} - \xi^{u0} + 1) \ldots \zeta^{u0} \times (\zeta^{u1} - \xi^{u1} + 1) \ldots \zeta^{u1}
}
{
((\zeta^{u0} + \zeta^{u1}) - (\xi^{u0} + \xi^{u1}) + 1) \ldots (\zeta^{u0} + \zeta^{u1})
}
\\
& =
\frac
{(\xi^{u0} + \xi^{u1} + 1)!
}
{
\xi^{u0}! \xi^{u1}!
}
\frac
{
\zeta^{u0}! \zeta^{u1}!
}
{
(\zeta^{u0} + \zeta^{u1} + 1)!
} 
\\
& \quad \times
\frac 
{
((\zeta^{u0} + \zeta^{u1}) - (\xi^{u0} + \xi^{u1}))!
}
{
(\zeta^{u0} - \xi^{u0})! (\zeta^{u1} - \xi^{u1})!
}
(\zeta^{u0} + \zeta^{u1} + 1).
\\
\end{split}
\]

Suppose that $x,y \in \SSS$ with $x \preceq y$.
It follows from Corollary~\ref{C:Martin_kernel_is_product} that
\[
\begin{split}
K(x,y) 
& =
\prod_{u \in \II}
\frac
{
((x^u)^{u0} + (x^u)^{u1} + 1)!
} 
{
(x^u)^{u0}! (x^u)^{u1}!
}
\frac
{
(y^u)^{u0}! (y^u)^{u1}!
}
{
((y^u)^{u0} + (y^u)^{u1} + 1)!
}
\\
& \quad \times
\frac 
{
(((y^u)^{u0} + (y^u)^{u1}) - ((x^u)^{u0} + (x^u)^{u1}))!
}
{
((y^u)^{u0} - (x^u)^{u0})! ((y^u)^{u1} - (x^u)^{u1})!
}
((y^u)^{u0} + (y^u)^{u1} + 1).
\\
\end{split}
\]

Recall from Example~\ref{E:dag_tree} that we may associate $x$ and $y$
with the two subtrees
\[
\sss = \{\emptyset\} \cup \{v \in \II: (x^u)^v > 0 \; 
              \text{for the unique $u \in \alpha(v)$}\}
\]
and
\[
\ttt = \{\emptyset\} \cup \{v \in \II: (y^u)^v > 0 \; 
              \text{for some the unique $u \in \alpha(v)$}\},
\]
in which case $(x^u)^v = \#\{w \in \sss : v \le w\} =: \#\sss(v)$ for
$v \in \sss \setminus \{\emptyset\}$ and $u \in \alpha(v)$
(respectively, $(y^u)^v = \#\{w \in \ttt : v \le w\} =: \#\ttt(v)$ for
$v \in \ttt \setminus \{\emptyset\}$ and $u \in \alpha(v)$).
Note for $\varepsilon = 0,1$ that
\[
(x^u)^{u\varepsilon} 
=
\begin{cases}
\#\sss(u \varepsilon),& \text{if $u \in \sss$},\\
0,& \text{otherwise},
\end{cases}
\]
and that
\[
(x^u)^{u0} + (x^u)^{u1} + 1 
=
\begin{cases}
\#\sss(u),& \text{if $u \in \sss$},\\
1,& \text{otherwise}.
\end{cases}
\]
Similar relations exist for $y$ and $\ttt$.  
It follows that
\[
\prod_{u \in \II}
\frac{
((x^u)^{u0} + (x^u)^{u1} + 1)!
} 
{
(x^u)^{u0}! (x^u)^{u1}!
}
= 
\#\sss!,
\]
\[
\prod_{u \in \II}
\frac
{
(y^u)^{u0}! (y^u)^{u1}!
}
{
((y^u)^{u0} + (y^u)^{u1} + 1)!
}
=
\frac{1}{\#\ttt!},
\]
\[
\prod_{u \in \II}
\frac 
{
(((y^u)^{u0} + (y^u)^{u1}) - ((x^u)^{u0} + (x^u)^{u1}))!
}
{
((y^u)^{u0} - (x^u)^{u0})! ((y^u)^{u1} - (x^u)^{u1})!
}
=
\frac
{
(\#\ttt - \#\sss)!
}
{
\prod_{u \in \ttt \setminus \sss} \#\ttt(u)
},
\]
and
\[
\prod_{u \in \II}
((y^u)^{u0} + (y^u)^{u1} + 1)
=
\prod_{u \in \ttt} \#\ttt(u),
\]
so we arrive at the simple formula
\begin{equation}
\label{E:Martin_kernel_BST}
K(x,y)
=
\binom{\#\ttt}{\#\sss}^{-1} \prod_{u \in \sss} \#\ttt(u).
\end{equation}

This formula may also be obtained
without using Corollary~\ref{C:Martin_kernel_is_product}
as follows.  With a slight abuse of notation, we think of the process
$(X_n)_{n \in \bN_0}$ as taking values in the set of finite subtrees
of $\{0,1\}^\star$ containing the root $\emptyset$.  
We first want a formula for $\bP^{\sss}\{X \; \text{hits} \; \ttt\}$ 
when $\sss$ and $\ttt$ are two such trees with 
$\sss \subseteq \ttt$.
For ease of notation, set $k := \#\sss$ and $n :=\#\ttt$.
It is known (see, for example, \cite[p.316]{SedFlaj}) that
\begin{equation}
\label{E:BST_prob}
\bP^{\{\emptyset\}}\{X \; \text{hits} \; \ttt\}
=\bP^{\{\emptyset\}}\{X_n = \ttt\} 
= \prod_{u \in \ttt} (\# \ttt(u))^{-1},
\end{equation}
Write $v_1,\ldots,v_{k+1}$ for the ``external 
vertices'' of $\sss$; that is, the elements of $\{0,1\}^\star$ that are connected
to a vertex of  $\sss$ by a directed edge, but are not vertices of
$\sss$ themselves (recall Figure~\ref{fig:binary_tree_external}). Denote by $\ttt(v_j)$, $j=1,\ldots,k+1$ 
the subtrees of $\ttt$ that are rooted 
at these vertices; that is, the  $\ttt(v_j)$ are the connected components
of $\ttt \setminus \sss$.  In order for the
BST process to pass from $\sss$ to $\ttt$ it
needs to place the correct number $n_j :=\# \ttt(v_j)$ 
of vertices into each of these subtrees 
and, moreover, the subtrees have to be equal to $\ttt(v_j)$, for $j=1,\ldots,k+1$.  
The process that tracks the number of vertices in each subtree
is, after we add the vector $(1, \ldots, 1)$, a multivariate P\'olya urn model starting 
with $k+1$ balls, all of different colors. Thus, the probability that
each subtree has the correct number of vertices is
\[
\binom{n - k}{n_1,\ldots,n_{k+1}}\,\frac{\prod_{i=1}^{k+1} n_i!}{(k+1)\cdot\ldots\cdot (n-1)\cdot n} 
= 
\binom{\# \ttt}{\# \sss}^{-1},
\]
using a standard argument for the P\'olya urn \cite[Chapter 4.5]{MR0488211}. 
Moreover, it is apparent from the recursive structure of the BST process that,
conditional on  $k+1$ subtrees receiving the correct number of vertices, the probability the subtrees are actually $\ttt(v_1), \ldots, \ttt(v_{k+1})$ 
is 
\[
\prod_{i=1}^{k+1} \prod_{v \in \ttt(v_i)} (\# \ttt(v))^{-1}
=
\prod_{v \in \ttt \setminus \sss} (\# \ttt(v))^{-1}.
\]
Thus,
\begin{equation}
\label{E:transition_probs_BST}
\bP^{\sss}\{X \; \text{hits} \; \ttt\} = \binom{\# \ttt}{\# \sss}^{-1} \prod_{v \in \ttt \setminus \sss} (\# \ttt(v))^{-1},
\end{equation}
and \eqref{E:Martin_kernel_BST} follows upon taking the appropriate ratio.
\end{example}

With Example~\ref{E:BST} in mind, we now begin to build
a general framework for characterizing the Doob-Martin
compactification of a trickle-down chain in terms of the
compactifications of each of the routing chains.
 
\begin{proposition}
\label{P:component_convergence}
Suppose $(y_n)_{n \in \bN_0}$ is a sequence in $\SSS$
such that $y_\infty^u := \lim_{n \rightarrow \infty} y_n^u$ exists
in the Doob-Martin topology of $\bar \SSS^u$ for each $u \in \II$.  Then,
$(y_n)_{n \in \bN_0}$ converges in the Doob-Martin topology
of $\SSS$ to a limit $y_\infty$ and the value at $(x,y_\infty)$
of the extended Martin kernel is
$K(x,y_\infty) = \prod_{u \in \II} K^u(x^u, y_\infty^u)$.
\end{proposition}

\begin{proof} 
The assumption that 
$y_\infty^u := \lim_{n \rightarrow \infty} y_n^u$
exists in the Doob-Martin topology of $\bar \SSS^u$ for each $u \in \II$
implies that $\lim_{n \rightarrow \infty} K^u(\xi,y_n^u)$ exists
for each $u \in \II$ and $\xi \in \SSS^u$.  This limit is, by definition,
the value $K^u(\xi, y_\infty^u)$ of the extended Martin kernel.  
We need to show for all $x \in \SSS$ that 
$\lim_{n \rightarrow \infty} K(x,y_n)$ exists and is given by
$\prod_{u \in \II} K^u(x^u, y_\infty^u)$.  It follows from 
Corollary~\ref{C:Martin_kernel_is_product} that
$K(x,y_n) = \prod_{u \in \II} K^u(x^u, y_n^u)$.
We also know from that result that we may restrict the
product to the fixed, finite set of $u$ for which $x^u \ne (0,0,\ldots)$,
and hence we may interchange the
limit and the product.
\end{proof}

\begin{remark} 
Proposition~\ref{P:component_convergence} shows that if 
the sequence $(y_n)_{n \in \bN_0}$ in $\SSS$ is
such that for each $u \in \II$ the component sequence 
$(y_n^u)_{n \in \bN_0}$ converges in the Doob-Martin
compactification of $\SSS^u$, then $(y_n)_{n \in \bN_0}$
converges in the Doob-Martin compactification of $\SSS$.

Establishing results in the converse direction is somewhat tricky,
since $K(x,y_n) = \prod_{u \in \II} K^u(x^u, y_n^u)$
might converge because $K^v(x^v, y_n^v)$ converges to $0$ for
some particular $v \in \II$, and so we are not able to conclude that
$K^u(x^u, y_n^u)$ converges for all $u \in \II$.  Instances
of this possibility appear in Section~\ref{S:Mallows} and Section~\ref{S:q_chain}.
\end{remark}

The following set of hypotheses gives one quite general setting in which
it is possible to characterize the Doob-Martin compactification
of $\SSS$ in terms of the compactifications of the component spaces
$\SSS^u$.  These
hypotheses are satisfied by a number of interesting examples such
as the binary search tree and the random recursive tree processes
(see Example~\ref{Ex:BST_ratios} and Example~\ref{Ex:space_time_RW} below
as well as Section~\ref{S:BST_and_DST} and Section~\ref{S:RRT_and_CRT}).
The key condition is part (iii) of the following set of hypotheses: 
it requires that the Doob-Martin boundary of the routing chain for
the vertex $u$ may be thought of as a set of subprobability measures on $\beta(u)$ that arise as the vector of limiting proportions of particles
that have been routed onward to the various elements of $\beta(u)$.

\begin{hypothesis}
\label{H:convergence_equivalent_ratios}
Suppose that the following hold for all $u \in \II$.
\begin{itemize}
\item[(i)]
Writing $|\xi| = \sum_{v \in \beta(u)} \xi^v$ for $\xi \in \SSS^u$,
the sets $\{\xi \in \SSS^u : |\xi| = m\}$ are finite for all 
$m \in \bN_0$, so that if $(\zeta_n)_{n \in \bN_0}$ 
is a sequence from $\SSS^u$, then the two conditions 
\begin{equation}
\label{E:finitely_many_visits}
\#\{n \in \bN_0 : \zeta_n = \zeta\} < \infty \; \text{ for all $\zeta \in \SSS^u$}
\end{equation}
and
\begin{equation}
\label{E:size_to_infinity}
\lim_{n \rightarrow \infty} |\zeta_n| = \infty
\end{equation}
are equivalent.
\item[(ii)]
In order that a sequence $(\zeta_n)_{n \in \bN_0}$ from $\SSS^u$
is such that $K^u(\xi, \zeta_n)$ converges as $n \rightarrow \infty$
for all $\xi \in \SSS^u$, it is necessary and sufficient that
either 
\[
\#\{n \in \bN_0 : \zeta_n \ne \zeta\} < \infty
\;
\text{for some $\zeta \in \SSS^u$}
\] 
or that
the equivalent conditions \eqref{E:finitely_many_visits}
and \eqref{E:size_to_infinity} hold and, in addition,
\begin{equation}
\label{E:existence_limit_proportions}
\lim_{n \rightarrow \infty} \frac{\zeta_n^{v}}{ |\zeta_n|} \; \text{exists for all $v \in \beta(u)$}.
\end{equation}
\item[(iii)]
If $(\zeta_n')_{n \in \bN_0}$ and $(\zeta_n'')_{n \in \bN_0}$ 
are two sequences from $\SSS^u$ such that 
$\#\{n \in \bN_0 : \zeta_n' = \zeta\} < \infty$ 
and 
$\#\{n \in \bN_0 : \zeta_n'' = \zeta\} < \infty$
for all $\zeta \in \SSS^u$ and both $K^u(\xi, \zeta_n')$ and
$K^u(\xi, \zeta_n'')$ converge for all $\xi \in \SSS^u$, then
\[
\lim_{n \rightarrow \infty} K^u(\xi, \zeta_n') 
= 
\lim_{n \rightarrow \infty} K^u(\xi, \zeta_n'')
\]
for all $\xi \in \SSS^u$ if and only if
\[
\lim_{n \rightarrow \infty} 
\frac{(\zeta_n')^{v}}{|\zeta_n'|}
=
\lim_{n \rightarrow \infty} 
\frac{(\zeta_n'')^{v}}{|\zeta_n''|}
\]
for all $v \in \beta(u)$.  It follows that there is a natural bijection
between $\partial \SSS^u := \overline{\SSS^u} \setminus \SSS^u$, where $\overline{\SSS^u}$
is the Doob-Martin compactification of $\SSS^u$, and the set $\cS^u$  of subprobability
measures on $\beta(u)$ that are limits in the vague topology of probability
measures of the form 
\[
\frac{1}{|\zeta_n|} \sum_{v \in \beta(u)} \zeta_n^v \delta_v,
\]
where $(\zeta_n)_{n \in \bN_0}$ is a sequence from $\SSS^u$
that satisfies \eqref{E:finitely_many_visits}.
\item[(iv)]
The bijection between $\partial \SSS^u$ and
$\cS^u$ is a homeomorphism if
the former set is equipped with the trace of the Doob-Martin topology
and the latter set is equipped with the trace of the vague topology.
\item[(v)]
There is a collection 
$\RR^u \subseteq \{0,1\}^{\beta(u)} \cap \SSS^u$
such that if $(\zeta_n)_{n \in \bN_0}$ is a sequence from $\SSS^u$ that satisfies \eqref{E:finitely_many_visits}
and $\lim_{n \rightarrow 0} K^u(\eta, \zeta_n)$ exists for all $\eta \in \RR^u$,
then $\lim_{n \rightarrow 0} K^u(\xi, \zeta_n)$ exists for all $\xi \in \SSS^u$.  Moreover, if
$(\zeta_n')_{n \in \bN_0}$ and $(\zeta_n'')_{n \in \bN_0}$ 
are two sequences from $\SSS^u$ 
that both satisfy \eqref{E:finitely_many_visits} and 
\[
\lim_{n \rightarrow \infty} K^u(\eta, \zeta_n')
=
\lim_{n \rightarrow \infty} K^u(\eta, \zeta_n'')
\]
for all $\eta \in \RR^u$, then
\[
\lim_{n \rightarrow \infty} K^u(\xi, \zeta_n') 
= 
\lim_{n \rightarrow \infty} K^u(\xi, \zeta_n'')
\]
for all $\xi \in \SSS^u$.
\item[(vi)]
Suppose that
$(\zeta_n)_{n \in \bN_0}$ is a sequence from $\SSS^u$ such that \eqref{E:finitely_many_visits} holds
and $K^u(\xi, \zeta_n)$ converges as $n \rightarrow \infty$
for all $\xi \in \SSS^u$.  Let $\rho = (\rho^v)_{v \in \beta(u)}$
be the subprobability vector of limiting proportions defined
by \eqref{E:existence_limit_proportions}.  
The extended Martin kernel is such that $K^u(\xi,\rho) = 0$
whenever $\xi^v \ge 2$ for some $v \in \beta(u)$ with $\rho^v = 0$,
whereas if $\rho_v > 0$ for some $v \in \beta(u)$, then there exists
a sequence $(\xi_m)_{m \in \bN}$ from $\SSS^u$ such that
$\xi_m^v = m$, $\xi_m^w \in \{0,1\}$ for $w \ne v$, and
$K(\xi_m, \rho) > 0$.
\item[(vii)]
A subprobability vector $\rho$ belongs to $\cS^u$ if and only if
there is a sequence $(\sigma_n^u)_{n \in \bN_0} \in \Sigma^u$ such
that
\[
\lim_{n \rightarrow \infty} 
\frac{\sigma_n^u}{|\sigma_n^u|}
=
\lim_{n \rightarrow \infty} 
\frac{\sigma_n^u}{n}
=
\rho.
\]
\end{itemize}
\end{hypothesis}

\begin{example}
\label{Ex:BST_ratios}
Hypothesis~\ref{H:convergence_equivalent_ratios} 
holds if $\# \beta(u) = 2$ for all $u \in \II$
(for example, if $\II = \{0,1\}^\star$), $\SSS^u = (\bN_0)^{\beta(u)}$,
and the Markov chains
$Y^u = (Y_n^u)_{n \in \bN_0}$ are such that
$(Y_n^u + (1,1))_{n \in \bN_0}$ are all P\'olya's urns
starting with one black ball and one white ball.
This is a consequence of the results in \cite{MR0176518}.
Indeed, the same is true if for arbitrary $\II$
with $\beta(u)$ finite for all $u \in \II$ we take
$\SSS^u = (\bN_0)^{\beta(u)}$ and let $Y^u$ be an
urn scheme of the sort considered in \cite{MR0362614} where there
is a (not necessarily integer-valued)
finite measure $\nu_u$ on $\beta(u)$ that describes the initial
composition of an urn with balls whose ``colors'' are identified with the
elements of $\beta(u)$, balls are drawn at random and replaced
along with a new ball of the same color, and $Y_n^u$ records the number of 
balls of the various colors that have been drawn by time $n$.
In this general case, the extended Martin kernel is given by
\[
\frac
{(|\nu_u| + |\xi| - 1) (|\nu_u| + |\xi| - 2) \cdots |\nu_u|}
{\prod_{v \in \beta(u)} [(\nu_u^v + \xi^v - 1) (\nu_u^v + \xi^v - 2) \cdots \nu_u^v]}
\prod_{v \in \beta(u)} (\rho^v)^{\xi^v},
\]
where $|\nu_u| = \sum_{v \in \beta(u)} \nu_u^v$,
$|\xi| = \sum_{v \in \beta(u)} \xi^v$, and
$(\rho^v)^{\xi^v}$ denotes the value $\rho^v$ 
that the probability measure $\rho$ assigns to $\{v\}$
raised to the power $\xi^v$.  We may take the set $\RR^u$ in this case
to be the coordinate vectors $e_v$, $v \in \beta(u)$,
where $e_v$ has a single $1$ in the $v^{\mathrm{th}}$ component and $0$
elsewhere.  The set $\cS^u$ consists of all the probability measures
on the finite set $\beta(u)$.
\end{example}

\begin{example}
\label{Ex:space_time_RW}
Hypothesis~\ref{H:convergence_equivalent_ratios} also holds if the set
$\beta(u)$ is finite for all $u \in \II$, $\SSS^u = (\bN_0)^{\beta(u)}$, 
and the routing chain $Y^u$ is given by
$Y^u := (\sum_{k=1}^n W_k^u)_{n \in \bN_0}$,
where the $W_k^u$ are independent, identically distributed
$\SSS^u$-valued random variables with distribution that has
support the set of coordinate vectors.  If $p_u^v$ is the probability
that the common distribution of the $W_k^u$ assigns to 
the coordinate vector $e_v$, then
the extended Martin kernel is given by
\[
K^u(\xi,\rho) = \prod_{v \in \beta(u)} \left(\frac{\rho^v}{p_u^v}\right)^{\xi^v}.
\]
Results of this type go back to \cite{MR0120683} and are described in
\cite{MR0407981}.  Once again, we may take $\RR^u$ 
to be the set of coordinate vectors, and once again
$\cS^u$ consists of all the probability measures
on the finite set $\beta(u)$. 
\end{example}

In order to state a broadly applicable result in the converse
 direction of Proposition \ref{P:component_convergence}
we first
need to develop some more notation and collect together some auxiliary results.

Adjoin a point $\diamond$ to $\II$
and write $\II_\infty$ for the set of
sequences of the form $(u_n)_{n \in \bN_0}$
where either $u_n \in \II$ for all $n \in \bN_0$
and $\hat 0 = u_0 \rightarrow u_1 \rightarrow \ldots$
or, for some $N \in \bN_0$, $u_n \in \II$ for $n \le N$,
$\hat 0 = u_0 \rightarrow \ldots \rightarrow u_N$, and
$u_n = \diamond$ for $n>N$.  We think of $\II_\infty$
as the space of directed paths through $\II$ that start at $\hat 0$
and are possibly ``killed'' at some time and sent to the
``cemetery'' $\diamond$.  

Write $\cC_\infty$ for the countable collection of
subsets of $\II_\infty$ of the form 
$\{(v_n)_{n \in \bN_0} \in \II_\infty : v_k = u_k, \, 0 \le k \le n\}$,
where $n \in \bN_0$, $u_k \in \II$ for $0 \le k \le n$,
and $\hat 0 = u_0 \rightarrow \ldots \rightarrow u_n$.
Denote by $\cI_\infty$ the $\sigma$-field generated by $\cC_\infty$.
The following result is elementary and we leave its proof to the reader.

\begin{lemma}
\label{L:probs_on_I_infty}
Any probability measure on the measurable space
$(\II_\infty, \cI_\infty)$ is specified by its values on the sets in
$\cC_\infty$.  The space of such probability measures equipped
with the coarsest topology that makes each of the maps
$\mu \mapsto \mu(C)$, $C \in \cC_\infty$, continuous
is compact and metrizable.
\end{lemma}

Consider the case of Lemma~\ref{L:probs_on_I_infty} where 
the measure $\mu$ describes the dynamics of a Markov process.  That is, 
for each $u \in \II$ there is a subprobability measure $r^u$ on $\beta(u)$ 
such that if the process is in state $u$, then the next step is with 
probability $(r^u)^v$ to $v$, and with probability $1-\sum_{v\in \beta(u)}(r^u)^v$ 
to $\diamond$.

Label $u \in \II$  
with $\downarrow$ if $u$ is {\em reachable} from $\hat 0$ (in the
classical sense of Markov chains), and with $\dagger$ otherwise. Denote by $\mathbf
J^\downarrow$ and $\mathbf J^\dagger$ the sets of vertices labeled with $\downarrow$ and
$\dagger$, respectively.

Clearly, in order to specify the distribution $\mu$ of the Markovian path 
starting from $\hat 0$ it
suffices to have the subprobability measures $r^u$ only for $u\in \mathbf J^\downarrow$.

Note that the labeling  $(\mathbf J^\downarrow, \mathbf J^\dagger)$ has the two properties
\begin{itemize}
   \item the vertex $\hat 0$ is labeled with $\downarrow$;
   \item if for some $v \ne \hat 0$ every vertex $u \in \alpha(v)$ is labeled with $\dag$,
         then $v$ is also labeled with $\dag$.
\end{itemize}
Let us now switch perspectives and start from a labeling instead of a collection 
of subprobability measures.

\begin{definition}
\label{D:admissible}
Say that a labeling of $\II$ with the symbols $\downarrow$
and $\dag$ is {\em admissible} if it satisfies the above two properties.
Write $\II^\downarrow$ (resp. $\II^\dag$) for the subset of vertices labeled
with $\downarrow$ (resp. $\dag$). 
\end{definition}

Note that if $(\II^\downarrow, \II^\dag)$ is an admissible labeling
of $\II$, $(u_n)_{n \in \bN_0}$ is a directed path in $\II$ with $u_0 = \hat 0$,
and we define a sequence $(\tilde u_n)_{n \in \bN_0}$ in $\II \cup \{\diamond\}$ by 
\[
\tilde u_n :=
\begin{cases} 
u_n, & \text{if $u_n \in \II^\downarrow$}, \\
\diamond, & \text{if $u_n \in \II^\dag$},
\end{cases}
\]
then $(\tilde u_n)_{n \in \bN_0}$ is an element of $\II_\infty$.

\begin{definition}
\label{D:compatible}
Given an admissible labeling  $(\II^\downarrow, \II^\dag)$ of $\II$,
say that a collection 
$(r^u)_{u \in \II^\downarrow}$, where $r^u$ is a subprobability measure on
$\beta(u)$ for $u \in  \II^\downarrow$, is {\em compatible} with 
the labeling if a vertex $v \in \II \backslash \{\hat 0\}$ is in $\II^\dag$ if and only if 
$\alpha(v) \cap \II^\downarrow = \emptyset$ or 
$(r^u)^v = 0$ for $u \in \alpha(v) \cap \II^\downarrow$.
\end{definition}

\begin{remark}
For an admissible labeling $(\II^\downarrow, \II^\dag)$ of $\II$ and a collection of
 subprobability measures as in Definition \ref{D:compatible}, compatibility of the
subprobability measures with the labeling is equivalent to the equality $\II^\downarrow =
\mathbf J^\downarrow$, where $\mathbf J^\downarrow$ is the set of vertices that are reachable
from $\hat 0$ under the Markovian dynamics specified by the subprobability measures.
\end{remark}

The assertions (i), (ii) and (iii) in the following lemma, with $\mathbf J^\downarrow$ and
$\mathbf J^\dagger$ instead of $\mathbf I^\downarrow$ and $\mathbf I^\dagger$, are
obvious. The proof of the lemma is then clear from the previous remark.

\begin{lemma}
\label{L:equiv_compat_and_probs}
Consider an admissible labeling of $\II$ 
with the symbols $\downarrow$ and $\dag$ and a compatible collection of 
subprobability measures
$(r^u)_{u \in \II^\downarrow}$.
\begin{itemize}
\item[(i)]
There is a unique probability measure $\mu$ on 
$(\II_\infty, \cI_\infty)$ for which the mass assigned to the set
$\{(v_n)_{n \in \bN_0} \in \II_\infty : v_k = u_k, \, 0 \le k \le n\} \in \cC_\infty$
is 
\[
\begin{cases} 
\prod_{k=0}^{n-1} (r^{u_k})^{u_{k+1}}, 
      & \quad \text{if  $u_k \in \II^\downarrow$ for $0 \le k \le n$},\\
   0, & \quad otherwise.
\end{cases}
\]
\item[(ii)]
The vertex $u$ belongs to $\II^\dag$ if and only if
$\mu\{(v_n)_{n \in \bN_0} \in \II_\infty : v_k = u_k, \, 0 \le k \le n\} = 0$
whenever $\hat 0 = u_0 \rightarrow \ldots \rightarrow u_n = u$.
\item[(iii)]
If $u \in \II^\downarrow$ and $v \in \beta(u)$, then
\[
(r^u)^v
=
\frac
{
\mu\{(v_n)_{n \in \bN_0} \in \II_\infty : v_k = u_k, \, 0 \le k \le n+1\}
}
{
\mu\{(v_n)_{n \in \bN_0} \in \II_\infty : v_k = u_k, \, 0 \le k \le n\}
}
\]
for any choice of 
$\hat 0 = u_0 \rightarrow \ldots \rightarrow u_n = u \rightarrow u_{n+1} = v$
such that the denominator is positive.  In particular, it is possible
to recover the labeling and the collection $(r^u)_{u \in \II^\downarrow}$ from 
the probability measure $\mu$.
\end{itemize}
\end{lemma}

\begin{theorem}
\label{T:main}
Suppose that Hypothesis~\ref{H:convergence_equivalent_ratios} holds.
Denote by $\cR_\infty$ the set of pairs $((\II^\downarrow, \II^\dag), (r^u)_{u \in \II^\downarrow})$,
such that $(\II^\downarrow, \II^\dag)$ is an admissible labeling of $\II$ and
$(r^u)_{u \in \II^\downarrow} \in \prod_{u \in \II^\downarrow} \cS^u$ is a compatible
collection of subprobability measures.
\begin{itemize}
\item[(i)]
If a sequence $(y_n)_{n \in \bN_0}$ in $\SSS$
converges to a point in the Doob-Martin boundary $\partial \SSS = \bar \SSS \backslash \SSS$, then
there exists $((\II^\downarrow, \II^\dag), (r^u)_{u \in \II^\downarrow}) \in \cR_\infty$ satisfying
\begin{equation}
\label{conv_ratio_to_ru}
\lim_{n \rightarrow \infty} 
\frac{y_n^u}{|y_n^{u}|}
= r^u \in \cS^u, \quad \text{for all $u \in \II^\downarrow$}.
\end{equation}
Moreover, if two such sequences converge to the same point then the corresponding elements of $\cR_\infty$ coincide.
\item[(ii)]
Conversely, if $((\II^\downarrow, \II^\dag), (r^u)_{u \in \II^\downarrow}) \in \cR_\infty$, then
there is a sequence $(y_n)_{n \in \bN_0}$ in $\SSS$ that
converges to a point in the Doob-Martin boundary $\partial \SSS = \bar \SSS \backslash \SSS$
and satisfies \eqref{conv_ratio_to_ru}.  Moreover, any two such sequences converge to the same point,
establishing a bijection between $\cR_\infty$ and $\partial \SSS$.
\item[(iii)]
For $x \in \SSS$ and $((\II^\downarrow, \II^\dag), (r^u)_{u \in \II^\downarrow}) \in \cR_\infty \cong \partial \SSS$, 
the value of the extended Martin kernel is
\[
\begin{cases}
\prod_{u \in \II^\downarrow} K^u(x^u, r^u),& \; \text{if $x^v = (0,0,\ldots)$
for all $v \notin \II^\downarrow$}, \\
0,& \; \text{otherwise.}
\end{cases}
\]
\item[(iv)]
Let $\cP_\infty$ be the set of
probability measures on $\II_\infty$
constructed from elements of $\cR_\infty$
via the bijection of Lemma~\ref{L:equiv_compat_and_probs}. 
Equip $\cP_\infty$ with the trace of the metrizable topology introduced in Lemma~\ref{L:probs_on_I_infty}.
The composition of the bijection between $\cP_\infty$ and $\cR_\infty$  and the bijection
between $\cR_\infty$ and $\partial \SSS$ is a homeomorphism between $\cP_\infty$ and $\partial \SSS$.
\end{itemize}
\end{theorem}

\begin{proof}
Consider part (i).
Suppose that
the sequence $(y_n)_{n \in \bN_0}$ 
converges to a point in $\partial \SSS$; that is, 
\begin{equation}
\label{E:assume_convergence}
\lim_{n \rightarrow \infty} K(x,y_n) \; \text{exists for all $x \in \SSS$}
\end{equation}
and no subsequence converges 
in the discrete topology on $\SSS$
to a point of $\SSS$. Thus,
\begin{equation}
\label{E:no_finite_convergence}
\#\{n \in \bN_0 : y_n = y\} < \infty \; \text{for any $y \in \SSS$}.
\end{equation}

Because of \eqref{E:no_finite_convergence} and 
Hypothesis~\ref{H:convergence_equivalent_ratios}(i), it follows that
\begin{equation}
\label{E:infinitely_many_particles}
\lim_{n \rightarrow \infty} |y_n^{\hat 0}| = \infty.
\end{equation}

Consider $\eta \in \RR^{\hat 0}$.   Define
$x \in \SSS$ by setting
$x^{\hat 0} = \eta$.  
By the consistency
condition \eqref{E:consistency_condition}, this
completely specifies $x$.  Note that $x^w = 0$ if $w \ne \hat 0$.
By Corollary~\ref{C:Martin_kernel_is_product},  
\[
K(x,y_n) 
=
K^{\hat 0}(\eta, y_n^{\hat 0}),
\]
and so $\lim_{n \rightarrow \infty} K^{\hat 0}(\eta, y_n^{\hat 0})$
exists.  Since this is true for all $\eta \in \RR^{\hat 0}$, it follows
from Hypothesis~\ref{H:convergence_equivalent_ratios}(v) 
that $\lim_{n \rightarrow \infty} K^{\hat 0}(\xi, y_n^{\hat 0})$
exists for all $\xi \in \SSS^{\hat 0}$.  Hence, by
Hypothesis~\ref{H:convergence_equivalent_ratios}(ii)
\[
\lim_{n \rightarrow \infty} 
\frac{(y_n^{\hat 0})^v}{|y_n^{\hat 0}|}
\]
exists for all $v \in \beta(\hat 0)$.  Write 
$r^{\hat 0} = ((r^{\hat 0})^v)_{v \in \beta(\hat 0)} \in \cS^{\hat 0}$ 
for the subprobability vector defined by the limits.  

If $(r^{\hat 0})^v = 0$ for some $v \in \beta(\hat 0)$, then,
from Hypothesis~\ref{H:convergence_equivalent_ratios}(vi), 
$\lim_{n \rightarrow \infty} K(x,y_n) = 0$ for any $x \in \SSS$
with $(x^{\hat 0})^v \ge 2$ -- no matter what the values of $y_n^u$
are for $u > \hat 0$.  Consequently, in order to understand
what further constraints are placed on the sequence 
$(y_n)_{n \in \bN_0}$ by the assumption that
\eqref{E:assume_convergence} holds, we need only consider
choices of $x \in \SSS$ with the property that
$(x^{\hat 0})^v = \{0,1\}$ for all $v \in \beta(\hat 0)$
such that $(r^{\hat 0})^v = 0$.  Note from the consistency condition
\eqref{E:consistency_condition} that for this restricted class
of $x$ we must have $x^w = 0$ for all $w \in \II$ such that all
directed path from $\hat 0$ to $w$ necessarily passes through 
$v \in \beta(\hat 0)$ with $(r^{\hat 0})^v = 0$.

Suppose that $r^{\hat 0} \ne 0$.
Fix a vertex $u \in  \beta(\hat 0)$ such that $(r^{\hat 0})^u > 0$
and  $\eta \in \RR^u$. 
From Hypothesis~\ref{H:convergence_equivalent_ratios}(vi), there
exists $\theta \in \SSS^{\hat 0}$ such that $\theta^u = |\eta| + 1$,
and $\theta^w \in \{0,1\}$ for $w \ne u$.
Define $x \in \SSS$ by setting
$x^{\hat 0} = \theta$ and $x^u = \eta$.  By the consistency
condition \eqref{E:consistency_condition}, this
completely specifies $x$.  Note that $x^w = 0$ if $w \notin \{\hat 0, u\}$.
By Corollary~\ref{C:Martin_kernel_is_product}, 
\[
K(x,y_n) 
=
K^{\hat 0}(\theta, y_n^{\hat 0}) K^u(\eta, y_n^u),
\]
and, by the choice of $\theta$,
$K^{\hat 0}(\theta, y_n^{\hat 0})$ converges to a non-zero value as 
$n \rightarrow \infty$.  Therefore, 
$\lim_{n \rightarrow \infty} K^u(\eta, y_n^u)$
exists.  Since this is true for all $\eta \in \RR^u$, it follows
from Hypothesis~\ref{H:convergence_equivalent_ratios}(v) 
that $\lim_{n \rightarrow \infty} K^u(\xi, y_n^u)$
exists for all $\xi \in \SSS^u$.  Hence, by
Hypothesis~\ref{H:convergence_equivalent_ratios}(ii),
\[
\lim_{n \rightarrow \infty} 
\frac{(y_n^u)^v}{|y_n^u|}
\]
exists for all $v \in \beta(u)$.  Write 
$r^u \in \cS^u$ for the resulting subprobability measure.

Continuing in this way, we see that,
under the assumption \eqref{E:no_finite_convergence}, 
if \eqref{E:assume_convergence} holds then 
there is a labeling of $\II$ with the symbols $\downarrow$
and $\dag$ such that the following are true:
\begin{itemize}
\item
the vertex $\hat 0$ is in $\II^\downarrow$;
\item
if a vertex $u$ is in $\II^\downarrow$, then
the limiting subprobability measure
\[
\lim_{n \rightarrow \infty} 
\frac{y_n^u}{|y_n^{u}|}
=: r^u \in \cS^u
\]
exists;
\item
a vertex $v \ne \hat 0$ belongs to $\II^\dag$ if and only if
every vertex $u \in \alpha(v)$ belongs to $\II^\dag$ or
$(r^u)^v = 0$ for every vertex $u \in \alpha(v) \cap \II^\downarrow$.
\end{itemize}
Thus, the labeling $(\II^\downarrow, \II^\dag)$ is admissible
and the collection $(r^u)_{u \in \II^\downarrow} \in \prod_{u \in \II^\downarrow} \cS^u$
are compatible, so $(I^\downarrow, I^\dag), (r^u)_{u \in \II^\downarrow})$ is 
an element of $\cR_\infty$.

Suppose that
$(y_n)_{n \in \bN_0}$ and $(z_n)_{n \in \bN_0}$
are two sequences from $\SSS$
that converge to the same point in $\partial \SSS$. Then,
$|y_n^{\hat 0}| \to \infty$ and $|z_n^{\hat 0}| \to \infty$
as $n \to \infty$, 
\[
\lim_{n \rightarrow \infty} K(x,y_n) \text{ exists for all } x \in \SSS,
\]
\[
\lim_{n \rightarrow \infty} K(x,z_n) \text{ exists for all } x \in \SSS,
\]
and
\[
\lim_{n \rightarrow \infty} K(x,z_n) = \lim_{n \rightarrow \infty} K(x,z_n)
\text{ for all } x \in \SSS.
\]
It is clear that the vertices of $\II$ that are labeled with the symbol $\downarrow$ (resp. $\dag$)
for the sequence $(y_n)_{n \in \bN_0}$ must coincide with
the vertices of $\II$ that are labeled with the symbol $\downarrow$ (resp. $\dag$)
for the sequence $(z_n)_{n \in \bN_0}$, and 
\[
\lim_{n \rightarrow \infty} 
\frac{y_n^{u}}{|y_n^{u}|}
=
\lim_{n \rightarrow \infty} 
\frac{z_n^{u}}{|z_n^{u}|}
\]
for the common set of vertices
$u \in \II$ labeled with $\downarrow$. 
This completes the proof of part (i).

Moreover, it follows from what we have just done that
if $x \in \SSS$ and the convergent sequence $(y_n)_{n \in \bN_0}$ 
is associated with $(\II^\downarrow, \II^\dag), (r^u)_{u \in \II^\downarrow})$, then 
\begin{equation}
\label{E:extended_Martin_kernel}
\lim_{n \rightarrow \infty} K(x,y_n) 
=
\begin{cases}
\prod_{u \in \II^\downarrow} K^u(x^u, r^u),& \; \text{if $x^v = (0,0,\ldots)$
for all $v \notin \II^\downarrow$}, \\
0,& \; \text{otherwise.}
\end{cases}
\end{equation}
This establishes part (iii) once we show part (ii).

Now consider part (ii).  Fix $(I^\downarrow, I^\dag), (r^u)_{u \in \II^\downarrow}) \in \cR_\infty$.
By Hypothesis~\ref{H:convergence_equivalent_ratios}(vii), for each
$u \in \II^\downarrow$ there
is a sequence $(\sigma_n^u)_{n \in \bN_0} \in \Sigma^u$ such that
\[
\lim_{n \rightarrow \infty} 
\frac{\sigma_n^u}{|\sigma_n^u|} 
= 
r^u.
\]
Choose sequences $(\sigma_n^u)_{n \in \bN_0} \in \Sigma^u$ for
$u \notin \II^\downarrow$ arbitrarily and set 
$\sigma = (\sigma^u)_{u \in \II} \in \Sigma$.  Define a sequence
$(y_n)_{n \in \bN_0}$ from $\SSS$ by setting
$y_n^u = \sigma^u(a_n^u(\sigma))$ for $n \in \bN_0$ and $u \in \II$.
It is clear from the arguments for part (i) that $(y_n)_{n \in \bN_0}$ converges to a point
in $\partial \SSS$ and \eqref{conv_ratio_to_ru} holds.  Moreover, it follows from the same
arguments  that any two convergent sequences satisfying \eqref{conv_ratio_to_ru}
must converge to the same point.  This establishes (ii).

The proof of (iv) is straightforward and we omit it.
\end{proof}

\section{Binary search tree and digital search tree processes}
\label{S:BST_and_DST}

Recall the binary search tree (BST) process from the Introduction.
We observed in Example~\ref{Ex:BST_ratios} that
Hypothesis~\ref{H:convergence_equivalent_ratios} holds for
the BST process.  Recall from Example~\ref{E:dag_tree} that we can identify
$\SSS$ in this case with the set of finite subtrees of the
complete binary tree $\{0,1\}^\star$ that contain the root
$\emptyset$.  Moreover, it follows
from the discussion in Section~\ref{S:general_trickle}
that $\partial \SSS$ is homeomorphic to the set of
probability measures on $\{0,1\}^\infty$ equipped with the weak topology
corresponding to the usual product topology on $\{0,1\}^\infty$.

We therefore abuse notation slightly and take $\SSS$ to be set
of finite subtrees of $\{0,1\}^\star$ rooted at $\emptyset$ and 
take $\partial \SSS$ to be the probability measures on $\{0,1\}^\infty$.

With this identification the 
partial order $\preceq$ on $\SSS$ is just subset containment and the
Martin kernel is given by
\begin{equation}\label{eq:BSTkernel}
K(\sss, \ttt) = 
\begin{cases}
\binom{\# \ttt}{\# \sss}^{-1} \prod_{u \in \sss} \# \ttt(u),& \quad \text{if $\sss \subseteq \ttt$}, \\
0,& \quad \text{otherwise},
\end{cases}
\end{equation}
where we recall from Example~\ref{E:BST} that $\#\ttt(u) = \#\{v \in \ttt : u \le v\}$.

A sequence $(\ttt_n)_{n \in \bN}$ in $\SSS$ with
$\# \ttt_n \rightarrow \infty$ converges in the
Doob-Martin compactification of $\SSS$ if and only if 
$\# \ttt_n(u) / \# \ttt_n$ converges for all $u \in \{0,1\}^\star$.
Moreover, if the sequence converges, then the limit can be identified with
the probability measure $\mu$ on $\{0,1\}^\infty$ such that
\[
\mu\{v\in \{0,1\}^\infty :  u < v\} = \lim_{n \rightarrow \infty} \frac{\# \ttt_n(u)}{ \# \ttt_n}
\]
for all $u \in \{0,1\}^\star$.  

Recall that the partial order on $\{0,1\}^\star$ is such that if 
$u= u_1 \ldots u_k$ and $v=v_1 \ldots,v_\ell$ are two words, then
$u \le v$ if and only if $u$ is an initial segment of $v$, 
that is, if and only if
$k \le \ell$ and $u_i=v_i$
for $i=1,\ldots,k$.  Extend this partial order to 
$\{0,1\}^\star \sqcup \{0,1\}^\infty$ by declaring that any two
elements of $\{0,1\}^\infty$ are not comparable and $u < v$ for
$u= u_1 \ldots u_k \in \{0,1\}^\star$ and 
$v=v_1 v_2 \ldots \in \{0,1\}^\infty$ when
$u_i=v_i$ for $i=1,\ldots,k$. Given $\mu \in \partial \SSS$, 
set 
\begin{equation}\label{eq:convention}
   \mu_u := \mu\{v\in \{0,1\}^\infty :  u < v\}.
\end{equation}
That is, $\mu_u$ is the mass assigned by $\mu$
to the set of infinite paths in the complete binary tree
that begin at the root and that pass through the vertex $u$.
The extended Martin kernel is given by
\begin{equation}
\label{eq:BSTkernel_extended}
K(\sss,\mu) 
= 
(\#\sss)!  \prod_{u\in \sss} \mu_u,
\quad \sss \in \SSS, \, \mu \in \partial \SSS.
\end{equation}

Note from the construction of the BST
process that its transition matrix is
\[
P(\sss, \ttt) 
= 
\begin{cases}
\frac{1}{\#\sss + 1}, \quad \text{if $\sss \subset \ttt$ and $\# (\ttt \setminus \sss) = 1$},\\
0,\quad  \text{otherwise},
\end{cases}
\]
(this is also apparent from \eqref{E:transition_probs_BST}). 
Set $h_\mu := K(\cdot, \mu)$ for $\mu \in \partial \SSS$. 
The Doob $h$-transform process corresponding to the regular function $h_\mu$
has state space 
\[
\{\ttt \in \SSS :  \text{$\mu_u > 0$ for all $u \in \ttt$}\}
\]
and transition matrix
\[
P^{(h_\mu)}(\sss, \ttt) 
= 
\begin{cases}
\mu_u, \quad \text{if $\ttt = \sss \sqcup \{u\}$},\\
0, \quad  \text{otherwise}.
\end{cases}
\]

It follows that the $h$-transformed process results from a trickle-down
construction.  For simplicity, we only verify this in
the case when $\mu_u > 0$ for all $u \in \{0,1\}^\star = \II$,
so that the state-space of the $h$-transformed process is all
of $\SSS$, and leave the formulation of
the general case to the reader.  The routing chain on 
$\SSS^u = \bN_0^{\{u0,u1\}}$ has transition matrix $Q^u$ given by
\[
Q^u((m,n),(m+1,n)) = \frac{\mu_{u0}}{\mu_u}
\]
and
\[
Q^u((m,n),(m,n+1)) = \frac{\mu_{u1}}{\mu_u}.
\]
In other words, we can regard the routing chain as the space-time chain
corresponding to the one-dimensional simple random walk that has probability
$\mu_{u0}/\mu_u$ of making a $-1$ step and probability $\mu_{u1}/\mu_u$ of 
making a $+1$ step.  

We have the following ``trickle-up'' construction of the $h$-transformed process.
Suppose on some probability space that there is a sequence of independent
identically distributed $\{0,1\}^\infty$-valued random variables 
$(V^n)_{n \in \bN}$ with common distribution $\mu$.  For an
initial finite rooted subtree $\www$ in the state space
of the $h$-transformed process, define a sequence
$(W_n)_{n \in \bN_0}$ of random finite subsets of $\{0,1\}^\star$ inductively
by setting $W_0 := \www$ and 
$W_{n+1} := W_n \cup \{V_1^{n+1} \ldots V_{H(n+1)+1}^{n+1}\}$, $n \ge 0$,
where 
$H(n+1) := \max\{l \in \bN : V_1^{n+1} \ldots V_l^{n+1} \in W_n\}$
with the convention $\max \emptyset = 0$.
That is, at each point in time
we start a particle at a ``leaf'' of the complete binary tree $\{0,1\}^\star$
picked according to $\mu$ and then let that particle trickle up the tree
until it can go no further because its path is blocked by previous particles
that have come to rest.  It is clear that $(W_n)_{n \in \bN_0}$ is
a  Markov chain with state space the appropriate set
of finite rooted subtrees of $\{0,1\}^\star$, 
initial state $\www$, and transition matrix
$P^{(h_\mu)}$.

It follows from the trickle-up construction and Kolmogorov's zero-one law
that the tail $\sigma$-field of the $h$-transformed process is trivial,
and hence $\mu$ is an extremal point of $\bar \SSS$.  Alternatively,
$\mu$ is extremal because
it is clear from the strong law of large numbers that the
$h$-transformed process converges to $\mu$.

Consider the special case of the $h$-transform construction 
when the boundary point
$\mu$ is the ``uniform'' or ``fair coin-tossing'' measure on $\{0,1\}^\infty$;
that is, $\mu$ is the infinite product of copies of the measure on $\{0,1\}$ that
assigns mass $\frac{1}{2}$ to each of the subsets $\{0\}$ and $\{1\}$.
In this case, the transition matrix of the $h$-transformed process is
\[
P^{(h_\mu)}(\sss, \ttt) 
= 
\begin{cases}
2^{-|u|}, \quad \text{if $\ttt = \sss \sqcup \{u\}$},\\
0, \quad  \text{otherwise},
\end{cases}
\]
where we write $|u|$ for the length of the word $u$; that is,
$|u| = k$ when $u=u_1 \ldots u_k$.  This transition mechanism is
that of the {\em digital search tree (DST) process}.  We have therefore
established the following result.

\begin{theorem} 
\label{T:DST_is_h-transform}
The digital search tree process is the Doob $h$-transform of the
binary search tree process associated with the regular function
$h(\sss) := (\#\sss)!  \prod_{u\in \sss} 2^{-|u|}$, $\sss \in \SSS$.
The regular function $h$ is extremal and corresponds to the
uniform probability measure on $\{0,1\}^\infty$ thought of
as an element of the Doob-Martin compactification of the state
space $\SSS$ of the BST process.  Consequently, the
Doob-Martin compactification of the DST process coincides with
that of the BST process.
\end{theorem}

\begin{remark}
The digital search tree (DST) algorithm is discussed
in~\cite[p.496ff]{MR0445948} and
in~\cite[Chapter~6]{MR1140708}. The process in
Theorem~\ref{T:DST_is_h-transform} appears as the
output of the DST algorithm if the input is a sequence
of independent and identically distributed random 0-1
sequences with distribution $\mu$, where $\mu$ is the
fair coin tossing measure. In the literature this
assumption is also known as the symmetric Bernoulli
model; in the general Bernoulli model the probability
$1/2$ for an individual digit~1 is replaced by an
arbitrary $p\in(0,1)$. In our approach we do not need
any assumptions on the internal structure of the random
0-1 sequences and we can work with a general
distribution $\mu$ on $\{0,1\}^\infty$. Any such DST
processes ``driven by $\mu$'' is an $h$-transform of
the BST process, provided that $\mu_u>0$ for all $u\in\II$,
and the trickle-up construction shows that the
conditional distribution of the BST process, given that
its limit is $\mu$, is the same as the distribution of
the DST process driven by $\mu$.
\end{remark}

In the symmetric Bernoulli model, the sample paths of
the DST process converge almost surely to the single
boundary point $\mu$ in the Doob-Martin topology, where
$\mu$ is the uniform measure on $\{0,1\}^\infty$.  We
now investigate the distribution of the limit of the
sample paths of the BST process.  There are several
routes we could take.

Recall that the routing chains for the BST process are essentially P\'olya urns;
that is, the routing chain
$Y^u = ((Y^u)^{u0}, (Y^u)^{u1})$ for the vertex $u \in \{0,1\}^\star$
makes the transition
$(g,d) \rightarrow (g+1,d)$ with probability $(g+1)/(g+d+2)$ and the
transition $(g,d) \rightarrow (g,d+1)$ with probability $(d+1)/(g+d+2)$.  
It is a well-known fact about the P\'olya urn that, 
when $((Y_0^u)^{u0} , (Y_0^u)^{u1}) = (0,0)$, the sequence
$((Y_n^u)^{u0} + (Y_n^u)^{u1})^{-1} ((Y_n^u)^{u0}, (Y_n^u)^{u1})$, $n\in\bN_0$,
converges almost surely to a random variable of the form $(U,1-U)$,
where $U$ is uniformly distributed on $[0,1]$.  It follows that if
we write $(T_n)_{n \in \bN}$ for the BST process, then almost surely
\[
\frac{\# T_n(u)}{\# T_n} \rightarrow \prod_{\emptyset < v \le u} U_v,
\quad u  \in \{0,1\}^\star,
\]
where the pairs $(U_{u0}, U_{u1})$, $u \in \{0,1\}^\star$, are independent,
the random variables $U_{u0}$ and $U_{u1}$ are uniformly distributed on $[0,1]$, and
$U_{u0} + U_{u1} = 1$.  Thus, the limit of the BST chain is
the random measure $M$ on $\{0,1\}^\infty$ such that
$M_u = \prod_{\emptyset < v \le u} U_v$ for all $u \in \{0,1\}^\star$.

Another approach is to observe that, from the trickle-up description of the
$h$-transformed processes described above and the extremality
of all the boundary points, we only need to find a
random measure on $\{0,1\}^\infty$ such that if we perform the
trickle-up construction from a realization of the random measure, then
we produce the BST process.  It follows from the main result of
\cite{MR0362614} that the random measure $M$ has the correct properties.

Yet another perspective is to observe that, by the 
general theory outlined in Section~\ref{S:Martin_general}, the distribution
of the limit is the unique probability measure 
$\bM$ on $\partial \SSS$
such that 
\[
1 \, =\, \int_{\partial \SSS} K(\sss, \mu) \, \bM(d\mu).
\]
In the present situation the right hand side evaluates to
\[
\int_{\partial \SSS} (\#\sss)! \, \prod_{u\in \sss} \mu_u \, \bM(d\mu) 
 \, =\, 
(\#\sss)!  \, \bE\left[\prod_{u\in \sss} \tilde M_u \right], 
\]
where $\tilde M$ is a random measure on $\{0,1\}^\infty$ with distribution $\bM$.
% \[
% \begin{split}
% 1 
% & = 
% \int_{\partial \SSS} K(\sss, \mu) \, \bM(d\mu) \\
% & =
% \int_{\partial \SSS} (\#\sss)! \, \prod_{u\in \sss} \mu_u \, \bM(d\mu) \\
% & =
% (\#\sss)!  \, \bE\left[\prod_{u\in \sss} \tilde M_u \right], \\
% \end{split}
% \]
% where $\tilde M$ is a random measure on $\{0,1\}^\infty$ with distribution $\bM$.
Rather than simply verify that taking $\tilde M = M$, where
$M_u =  \prod_{\emptyset < v \le u} U_v$ as above, has the requisite property, we
consider a more extensive class of random probability measures with similar
structure, compute the corresponding regular functions, and identify
the transition matrices of the resulting $h$-transform processes.
  
Let the pairs $(R_{u0}, R_{u1})$, $u \in \{0,1\}^\star$,  be independent
and take values in the set $\{(a,b) : a,b  \ge 0, \, a+b = 1\}$.
Define a random probability measure $N$ on $\{0,1\}^\infty$ by
setting $N_u := \prod_{\emptyset < v \le u} R_v$ for all $u \in \{0,1\}^\star$.
The corresponding regular function is
\[
\begin{split}
h(\sss)
& =
\bE\left[K(\sss, N)\right] \\
& =
(\#\sss)!  \, \bE\left[\prod_{u\in \sss} N_u \right] \\
& =
(\#\sss)! \,
\bE\left[\prod_{u\in \sss}  \prod_{\emptyset < v \le u} R_v \right] \\
& =
(\#\sss)! \, 
\bE\left[\prod_{u \in \sss \setminus \{\emptyset\}}  R_u^{\# \sss(u)} \right] \\
& =
(\#\sss)! \, \prod_u A_u(\# \sss(u0), \# \sss(u1)), \\
\end{split}
\]
where the last product is over $\{0,1\}^\star$ and 
\[
A_u(j,k) := \bE \left[ R_{u0}^j R_{u1}^k \right].
\]

With this notation, the probability that the resulting $h$-transform
of the BST process makes a transition from $\sss$ to 
$\ttt :=\sss \sqcup \{v\}$ is
\begin{equation}
\label{eq:transGen}
\begin{split}
h(\sss)^{-1} \frac{1}{\# \sss + 1} h(\ttt)
& =
\frac{1}{\#s+1} 
\frac{
(\# s +1)! \prod_u A_u(\# \ttt(u0),\# \ttt(u1))
}
{
(\# s)!  \prod_u A_u(\# \sss(u0),\# \sss(u1))
} \\
& = \prod_u 
\frac{
A_u(\# \ttt(u0),\# \ttt(u1))
}
{ 
A_u(\# \sss(u0),\# \sss(u1))
} \\
& =
\prod_{\emptyset \le u < v} 
\frac{
A_u(\# \ttt(u0),\# \ttt(u1))
}
{ 
A_u(\# \sss(u0),\# \sss(u1))
}, \\
\end{split}
\end{equation}
because $\#\sss(u)=\#\ttt(u)$ unless $u\le v$.

The ratios in \eqref{eq:transGen} have a simple form: 
if $\#\sss(u0)=j$ and  $\#\sss(u1)=k$, then 
\begin{equation}
\label{eq:totheleft}
\frac{
A_u(\# \ttt(u0),\# \ttt(u1))
}
{ 
A_u(\# \sss(u0),\# \sss(u1))
} 
= 
\frac{A_u(j+1,k)}{A_u(j,k)}, \quad\text{if $u0 \le v$},
\end{equation}
and 
\begin{equation}
\label{eq:totheright}
\frac{
A_u(\# \ttt(u0),\# \ttt(u1))
}
{ 
A_u(\# \sss(u0),\# \sss(u1))
}     
= 
\frac{A_u(j,k+1)}{A_u(j,k)}, \quad\text{if $u1 \le v$}.
\end{equation}

Suppose now that each $R_{u0}$ has a beta distribution with parameters $\theta_u$ and $\eta_u$, (so that $R_{u1} = 1 - R_{u0}$ has a beta distribution with parameters $\eta_u$ and $\theta_u$ and the pair $(R_{u0}, R_{u1})$ has a Dirichlet distribution
with parameters $\theta_u$ and $\eta_u$). 
Then,
\[
A_u(j, k) = \frac{\theta_u(\theta_u + 1) \cdots (\theta_u + j - 1)
                            \times \eta_u(\eta_u + 1) \cdots (\eta_u + k - 1)}
                        {(\theta_u+\eta_u)(\theta_u+\eta_u + 1) 
                                    \cdots (\theta_u+\eta_u + j + k - 1)},
\]
and the factors in \eqref{eq:totheleft} and \eqref{eq:totheright} are
\[
  \frac{\theta_u+ j}{\theta_u + \eta_u + j + k} \; \text{ and } \;
       \frac{\eta_u+ k}{\theta_u + \eta_u + j + k}, 
\]
respectively.  As expected, the
BST chain arises as the special case $\theta_u=\eta_u=1$ for all $u$. 

\begin{remark}
The chain with
$\theta_u=\eta_u=\ell$ for some fixed $\ell\in\bN$  appears in 
connection with the median-of-$(2\ell-1)$ version of the algorithms 
Quicksort and Quickselect (Find) -- 
see \cite{MR1707969}.
\end{remark}   

A special case of the above construction arises in connection with 
Dirichlet random measures. 
Recall that a {\em Dirichlet random measure} (sometimes called a {\em Ferguson random measure})
{\em directed by a finite measure} $\nu$ on $\{0,1\}^\infty$
is a random probability measure $\tilde N$ on 
$\{0,1\}^\infty$ with the property that, 
for any Borel partition $B_1, \ldots, B_k$ 
of $\{0,1\}^\infty$, the random vector $(\tilde N(B_1),\ldots, \tilde N(B_k))$ has a Dirichlet distribution with parameters 
$(\nu(B_1), \ldots, \nu(B_k))$. In particular, $\tilde N(B)$ has a beta distribution with parameters $\nu(B)$ and $\nu(B^{\mathrm{c}})$. 
It follows easily from Lemma~\ref{L:Dirichlet_fact} below that if 
$\theta_u = \nu_{u0}$ and $\eta_u = \nu_{u1}$ for all $u$ in the
above construction of a random probability measure using beta distributed
weights, then the result is a Dirichlet random measure directed by $\nu$.
(We note that the random measures that appear as the limit of the BST and median-of-($2\ell-1$) processes are {\bf not} Dirichlet.)

\begin{lemma}
\label{L:Dirichlet_fact}
Suppose that $(D_1, D_2, D_3, D_4)$ is a Dirichlet distributed
random vector with parameters $(\alpha_1, \alpha_2, \alpha_3, \alpha_4)$.
Then, the three pairs
\[
 \left(\frac{D_1}{D_1+D_2}, \frac{D_2}{D_1+D_2}\right),
\, 
\left(\frac{D_3}{D_3+D_4}, \frac{D_4}{D_3+D_4} \right)
\text{ and }
   (D_1+D_2, D_3+D_4)
\]
are independent Dirichlet distributed random vectors
 with respective parameters $(\alpha_1, \alpha_2)$, $(\alpha_3, \alpha_4)$,
and $(\alpha_1+\alpha_2, \alpha_3+\alpha_4)$.  
\end{lemma}

\begin{proof}
Note that $(D_1, D_2, D_3, D_4)$ has the same distribution as
\[
   \left(\frac{G_1}{G_1 + \cdots + G_4},\ldots, \frac{G_4}{G_1 + \cdots + G_4}\right),
\]
where the $G_1, \ldots, G_4$ are independent and $G_i$ has the Gamma distribution with parameters 
$(\alpha_i,1)$. Moreover, the latter random vector is independent of the sum $G_1 + \cdots + G_4$.

Now,
\[
   \left(D_1+D_2, D_3+D_4,\frac{D_1}{D_1+D_2}, \frac{D_2}{D_1+D_2},
           \frac{D_3}{D_3+D_4}, \frac{D_4}{D_3+D_4} \right)
\]
has the same distribution as
\[
   \left(\frac{G_1+G_2}{G_1 + \cdots + G_4},\frac{G_3+G_4}{G_1 + \cdots + G_4},
       \frac{G_1}{G_1+G_2}, \frac{G_2}{G_1+G_2},\frac{G_3}{G_3+G_4}, \frac{G_4}{G_3+G_4}\right).
\]
By the fact above, 
\[
G_1+G_2,
\]
\[
G_3+G_4,
\] 
\[
\left(\frac{G_1}{G_1+G_2}, \frac{G_2}{G_1+G_2}\right),
\] 
and
\[
\left(\frac{G_3}{G_3+G_4}, \frac{G_4}{G_3+G_4}\right)
\] are independent, and so
\[
\left(\frac{G_1+G_2}{G_1 + \cdots + G_4},\frac{G_3+G_4}{G_1 + \cdots + G_4}\right),
\]
\[
\left(\frac{G_1}{G_1+G_2}, \frac{G_2}{G_1+G_2}\right),
\] 
and 
\[
\left(\frac{G_3}{G_3+G_4}, \frac{G_4}{G_3+G_4}\right)
\]
are independent.
\end{proof}

\section{Random recursive trees and nested Chinese restaurant processes}
\label{S:RRT_and_CRT}

\subsection{Random recursive trees from another encoding of permutations}
\label{SS:RRT_and_perms}

Recall from the Introduction how the binary search tree process
arises from a classical bijection between permutations of 
$[n] := \{1,2, \ldots, n\}$ and a suitable class of
labeled rooted trees.  The random recursive tree process
arises from a similar, but slightly less well-known,  bijection
that we now describe.

We begin with a definition similar to that of the complete binary tree in the Introduction.
Denote by $\bN^\star:=\bigsqcup_{k=0}^\infty \bN^k$ the set of finite tuples
or {\em words} drawn from the alphabet $\bN$ (with the empty word $\emptyset$
allowed).  
Write an $\ell$-tuple $(v_1, \ldots, v_\ell) \in \bN^\star$ more simply as
$v_1 \ldots v_\ell$. Define a directed graph with vertex set
$\bN^\star$ by declaring that if
$u= u_1 \ldots u_k$ and $v=v_1 \ldots v_\ell$ are two words,
then $(u,v)$ is a directed edge (that is, $u \rightarrow v$)
if and only if $\ell=k+1$ and $u_i=v_i$
for $i=1,\ldots,k$.  Call this directed graph
the {\em complete Harris-Ulam tree}.  
A {\em finite rooted Harris-Ulam tree} is a subset $\ttt$ of
$\bN^\star$ with properties:
\begin{itemize}
\item
$\emptyset \in \ttt$,
\item
if $v = u_1 \ldots u_k \in \ttt$, then
$u_1 \ldots u_j \in \ttt$ for $1 \le j \le k-1$
and $u_1 \ldots u_{k-1} m \in \ttt$ for $1 \le m \le u_k - 1$.
\end{itemize}
As in the binary case there is a canonical way to draw a
finite rooted Harris-Ulam tree in the plane, see 
Figure~\ref{fig:Harris-Ulam_tree} for an example.
Further, we can similarly define a vertex $u\in\bN^\star$ to be an 
external vertex of the tree $\ttt$ if $u\notin\ttt$ and if $\ttt\sqcup\{u\}$ 
is again a Harris-Ulam tree. Note that, in contrast to the binary case,
external vertices are now specified by their immediate predecessor;
in particular, a Harris-Ulam tree with $n$ vertices has $n$ external 
vertices.

\begin{figure}[htbp]
	\centering
		\includegraphics[width=1.00\textwidth]{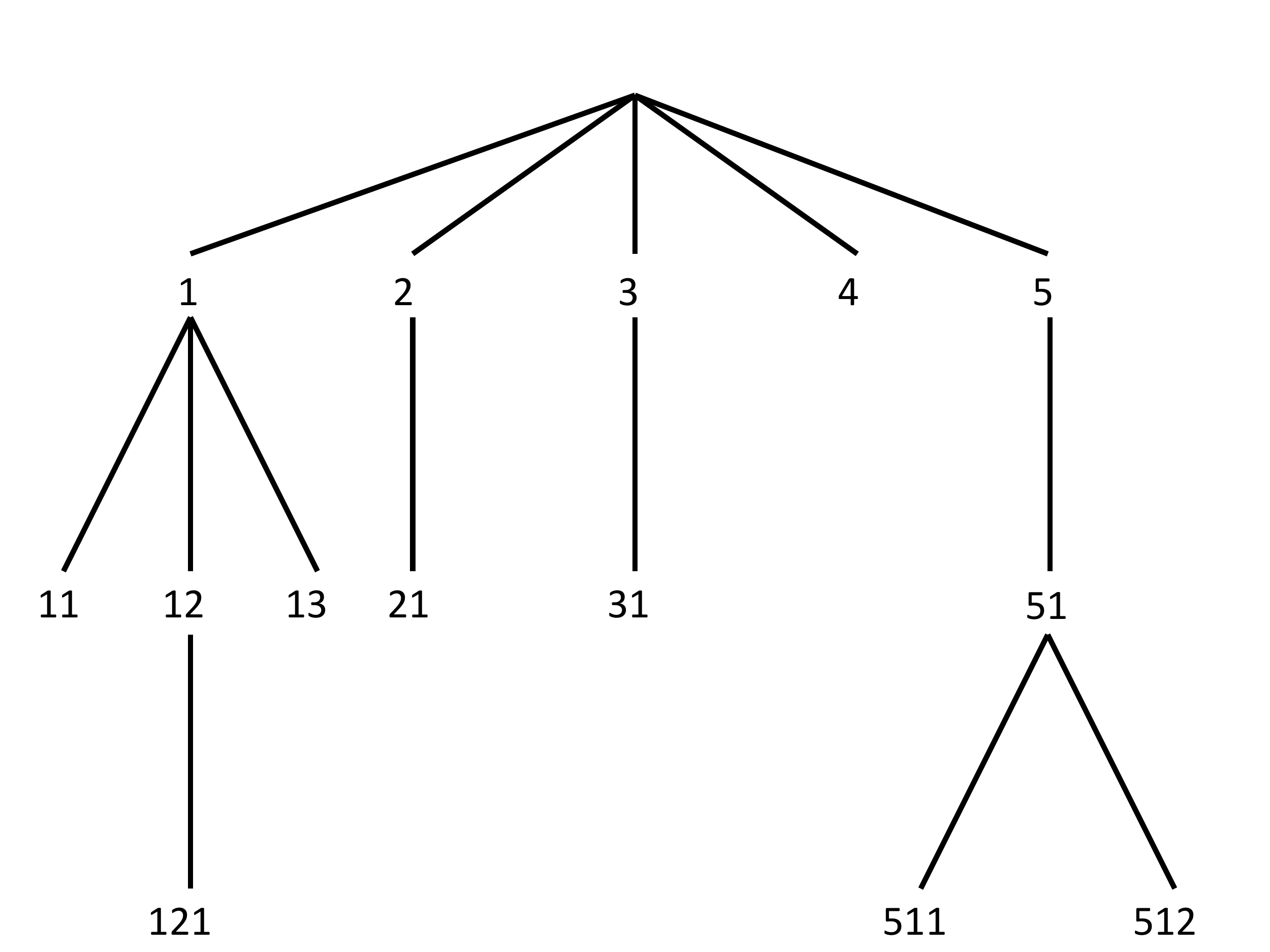}
	\caption{An example of a finite rooted Harris-Ulam tree.}
	\label{fig:Harris-Ulam_tree}
\end{figure}

Given a permutation $\pi$ of $[n]$,
set $r(i) = \pi^{-1}(i)$ for $1 \le i \le n$. Construct a
finite rooted Harris-Ulam tree with $n+1$ vertices labeled by
$[n] \cup \{0\}$ from $r(1), \ldots, r(n)$ recursively, as follows.  
Denote by $\ttt_0$ the tree consisting of just the  root $\emptyset$ labeled with $0$.
Suppose for $1 \le i \le n-1$ that a tree $\ttt_i$ with $i$ vertices labeled by $\{0,\ldots, i-1\}$
has already been defined.
Assume that $i = r(\ell)$.
If $\{ j : 1 \le j < \ell, \,  r(j) < i\} = \emptyset$, set $s := 0$.
Otherwise, set $s := r(k)$, where $k := \max\{ j : 1 \le j < \ell, \,  r(j) < i\}$.
Let $u$ be the vertex in $\ttt_i$ labeled
by $s$.  Put $q := \max\{p \in \bN : u p \in \ttt_i\} + 1$, adjoin the vertex $u q$
to $\ttt_i$ to create the tree $\ttt_{i+1}$, and label this new vertex with $i$.

For example, $1$ is always the first child of $0$ (occupying the vertex $1$
in the complete Harris-Ulam tree)
and $2$ is either the
second child of $0$ (occupying the vertex $2$ in the complete Harris-Ulam tree) or the first child
of $1$ (occupying the vertex $11$ in the complete Harris-Ulam tree), 
depending on whether $2$ appears before or after $1$ in the
list $r(1), \ldots, r(n)$.  See Figure~\ref{fig:recursive_tree} for
an instance of the construction with $n=9$.

Clearly, $\pi$ can be reconstructed from the tree and its vertex labels.

\begin{figure}[htbp]
	\centering
		\includegraphics[width=1.00\textwidth]{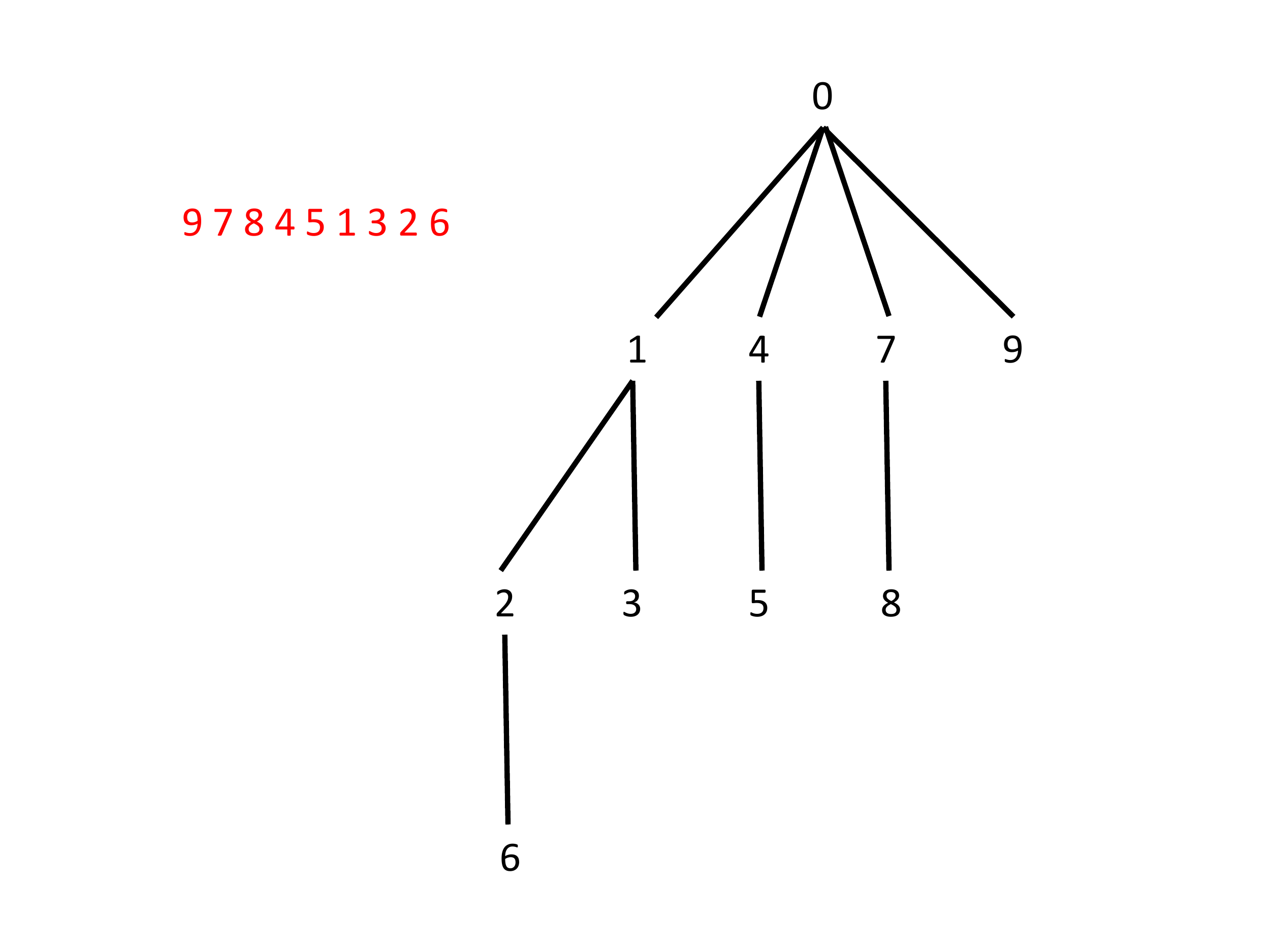}
	\caption{The labeled Harris-Ulam 
tree corresponding to the permutation of $[9]$ with 
$r(1), \ldots, r(9) = 9, 7, 8, 4, 5, 1, 3, 2, 6$.
For the sake of clarity, the 
%$\bN^\star$-valued formal definitions
%of the vertices are 
Harris-Ulam coding of the vertices as elements of $\bN^\star$ is not shown.                      
The correspondence between the labels from $[9] \cup \{0\}$ and the coding
of the vertices by elements of $\bN^\star$ is
$0 \leftrightarrow \emptyset$,
$1 \leftrightarrow 1$,
$2 \leftrightarrow 11$,
$3 \leftrightarrow 12$,
$4 \leftrightarrow 2$,
$5 \leftrightarrow 21$,
$6 \leftrightarrow 111$,
$7 \leftrightarrow 3$,
$8 \leftrightarrow 31$,
$9 \leftrightarrow 4$
.}
	\label{fig:recursive_tree}
\end{figure}

As in the Introduction, 
given a sequence $(U_n)_{n \in \bN}$ 
of independent identically distributed random variables
that are uniform on the unit interval $[0,1]$, define a random permutation
$\Pi_n$ of $[n]$ for each positive integer $n$ by setting
$\Pi_n(k) = \# \{1 \le \ell \le n : U_\ell \le U_k\}$.
Applying the bijection to $\Pi_n$, we obtain
a  random labeled rooted tree and a corresponding unlabeled rooted
tree that we again denote by $L_n$ and $T_n$, respectively.  
Both of these processes are Markov
chains with simple transition probabilities.  For example,
given $T_n$ we
pick one of its $n+1$ vertices uniformly at random and connect a new
vertex to it to form $T_{n+1}$.
Thus, $(T_n)_{n \in \bN}$ is the simplest {\em random recursive tree}
process (see, for example, \cite{smythe-mahmoud} for a survey
of such models).  

As with the BST and DST processes, 
we think of building the sequence $(T_n)_{n \in \bN}$
by first building a growing sequence of finite rooted
Harris-Ulam trees labeled with the values of the input
sequence $U_1,U_2,\ldots$ and then ignoring the
labels. The transition rule for
the richer process takes a simple form: attach a new vertex
labeled with
$U_{n+1}$ to the root if $U_{n+1}$ is smaller than each  
of the previous variables $U_1,\ldots,U_n$; if not, then attach a new vertex
labeled with $U_{n+1}$ to the existing vertex that
is the labeled with the rightmost of the smaller elements. 
In contrast to the binary search tree situation,
the labeled versions of the
trees $T_1,\ldots,T_{n-1}$ can now be determined from the
labeled version of $T_n$. However, if we remove the labels 
then we are in the same situation as in the BST case: the next 
tree is obtained by choosing an external vertex of the current tree 
uniformly at random and attaching it to the current tree.

\subsection{Chinese restaurant processes}
\label{SS:CRPs}

Suppose that in the tree $T_n$ the root has
$k$ offspring.  Let
$n_1, \ldots, n_k$ denote the number of vertices
in the subtrees rooted at each of these offspring,
so that $n_1 + \cdots + n_k = n$.  Note that in constructing $T_{n+1}$
from $T_n$,
either a new vertex is attached to the $j^{\mathrm{th}}$ subtree
with probability $n_j/(n+1)$ 
or it is attached to the root and begins a new subtree with
probability $1/(n+1)$. Thus, the manner in which the number
and sizes of subtrees rooted at offspring of the root evolve
is given by the number and sizes of tables in the simplest
{\em Chinese restaurant process}: the $n^{\mathrm{th}}$ customer
to enter the restaurant finds $k$ tables in use with
respective numbers of occupants $n_1, \ldots, n_k$ and 
the customer either sits at the $j^{\mathrm{th}}$ table
with probability $n_j/(n+1)$ or starts a new table with probability
$1/(n+1)$.

It is clear from the construction of $(T_n)_{n \in \bN}$ 
that if we begin observing
the subtree below one of the offspring of the root at the time
the offspring first appears and only record the state
of the subtree at each time it grows, then the resulting tree-valued process
has the same dynamics as $(T_n)_{n \in \bN}$.  Iterating this observation,
we see that we may think of $(T_n)_{n \in \bN}$ 
as an infinite collection of hierarchically nested Chinese
restaurant processes  and, in particular, that $(T_n)_{n \in \bN}$ arises
as an instance of the  trickle-down construction.

Rather than just investigate the Doob-Martin compactification 
of $(T_n)_{n \in \bN}$ we first recall the definition of
Pitman's two-parameter family of processes to which the simple Chinese
restaurant process belongs
-- see \cite{MR2245368} for background and an extensive
treatment of the properties of these processes.    
We then apply the trickle-down construction to 
build a tree-valued Markov chain that uses these more general
processes as routing instructions.  Analogous
finitely nested Chinese restaurant processes have been used
in hierarchical Bayesian inference \cite{MR2279480}.

A member of the family of Chinese restaurant processes
is specified by two parameters $\alpha$ and $\theta$
that satisfy the constraints
\[
\text{$\alpha < 0$ and $\theta = - M \alpha$ for some $M \in \bN$}
\]
or
\[
\text{$0 \le \alpha < 1$ and $\theta > - \alpha$.}
\]
At time $p$ the state of the process is a partition $\cB$ of the
set $[p]$ with $\# \cB$ blocks that are thought of as describing
the composition of $\# \cB$ occupied tables.  
The next customer arrives at time $p+1$ and
decides either to sit at an empty table with probability
\[
\dfrac{\theta + \alpha \#\cB }{p + \theta},
\]
thereby adjoining an extra block $\{p+1\}$ to the
partition and increasing the number of blocks by $1$,
or else to
sit at an occupied table $B \in \cB$ of size $\# B $ with probability
\[
\dfrac{\# B -  \alpha}{p + \theta},
\]
thereby replacing the block $B$ by the block $B \cup \{p+1\}$
and leaving the number of blocks unchanged.

The probability that the partition of $[q]$  we see at time $q$
is $\cB = \{B_1, \ldots, B_n\}$ with block sizes $b_k = \# B_k$ is
\[
\frac
{
(\theta+\alpha)(\theta+2\alpha) \cdots (\theta+(n-1)\alpha)
}
{
(\theta+1) (\theta+2) \cdots (\theta+q-1)
}
\prod_{k=1}^n (1-\alpha) (2-\alpha) \cdots (b_k-1-\alpha).
\]
Note that if $\alpha < 0$ and $\theta = - M \alpha$ for some $M \in \bN$,
then, with probability one, the number of blocks in the partition is 
always at most $M$.

We are only interested in the process that records the number
and size of the blocks.  This process is also Markov.
The probability that the random partition  at time
$q$ has block sizes $b_1, b_2, \ldots, b_n$ is
\[
\begin{split}
& \frac
{
(\theta+\alpha)(\theta+2\alpha) \cdots (\theta+(n-1)\alpha)
}
{
(\theta+1) (\theta+2) \cdots (\theta+q-1)
}
\prod_{k=1}^n (1-\alpha) (2-\alpha) \cdots (b_k-1-\alpha) \\
& \quad \times \binom{q-1}{b_1 - 1} \binom{q-b_1-1}{b_2-1} \cdots \binom{q-b_1-\cdots-b_{n-2}-1}{b_{n-1}-1}.\\
\end{split}
\]
The ordering of the blocks in this formula is their order of appearance: $b_1$ is the size
of the initial table, $b_2$ is the size of the table that began receive customers next, and so on.

More generally, the probability that we go from the partition $\cA = \{A_1, \ldots, A_m\}$ at time $p$
to the partition $\cB = \{B_1, \ldots, B_n\}$ at time $q>p$ is
\[
\begin{split}
& \frac
{
(\theta+m\alpha)(\theta+2\alpha) \cdots (\theta+(n-1)\alpha)
}
{
(\theta+p) (\theta+p+1) \cdots (\theta+q-1)
}  \\
& \quad \times \prod_{k=1}^m (a_k - \alpha) (a_k+1-\alpha) \cdots (b_k-1-\alpha)
\prod_{k=m+1}^n (1-\alpha) (2-\alpha) \cdots (b_k-1-\alpha). \\
\end{split}
\]
The corresponding
probability that we go from a partition with block sizes $a_1, \ldots, a_m$
at time $p$ to one with block sizes $b_1, \ldots, b_n$ at time $q>p$ is
\[
\begin{split}
& \frac
{
(\theta+m\alpha)(\theta+2\alpha) \cdots (\theta+(n-1)\alpha)
}
{
(\theta+p) (\theta+p+1) \cdots (\theta+q-1)
} \\
& \quad \times \prod_{k=1}^m (a_k - \alpha) (a_k+1-\alpha) \cdots (b_k-1-\alpha) 
\prod_{k=m+1}^n (1-\alpha) (2-\alpha) \cdots (b_k-1-\alpha)\\
& \quad \times \binom{q-p}{b_1-a_1} \binom{(q-b_1)-(p-a_1)}{b_2-a_2} \\
& \quad \times \cdots \binom{(q-b_1-\cdots-b_{m-1}) - (p-a_1-\cdots-a_{m-1})}{b_m-a_m} \\
& \quad \times \binom{q-b_1-\cdots-b_m-1}{b_{m+1}-1} \binom{q-b_1-\cdots-b_{m+1}-1}{b_{m+2}-1} \\ 
& \quad \times \cdots \binom{q-b_1-\cdots-b_{n-2}-1}{b_{n-1}-1}.\\
\end{split}
\]

We can think of the block size process as a Markov chain with state space
\[
\EE := 
\{(0,0,\cdots)\} 
\sqcup
\bigsqcup_{m \in \bN} \bN^m \times \{0\} \times \{0\} \times \cdots
\subset 
\bN_0^\bN
\]
when $0 \le \alpha < 1$ and $\theta > - \alpha$, or
\[
\EE := 
\{(0,0,\cdots)\} 
\sqcup
\bigsqcup_{m =1}^M \bN^m \times \{0\}^{M-m}
\subset 
\bN_0^M
\]
when $\alpha < 0$ and $\theta = - M \alpha$ for some $M \in \bN$.
For two states $\aaa=(a_1, \ldots, a_m, 0, 0, \ldots) \in \EE$ and 
$\bbb=(b_1, \ldots, b_n, 0, 0, \ldots) \in \EE$
with $1 \le m \le n$, $b_i \ge a_i > 0$ when $1 \le i \le m$,  $b_j > 0$
when $m+1 \le j \le n$, $\sum_{i=1}^m a_i = p$, and $\sum_{j=1}^n b_j = q$,
the Martin kernel is
\[
\begin{split}
K(\aaa,\bbb)
& =
\frac
{
(\theta+1)(\theta+2) \cdots (\theta+p-1)
}
{
(\theta+\alpha)(\theta+2 \alpha) \cdots (\theta+(m-1)\alpha)
} \\
& \quad \times 
\left[\prod_{k=1}^m (1-\alpha) (2-\alpha) \cdots (a_k-1-\alpha)\right]^{-1} \\
& \quad \times
\frac
{(q-p)!}
{(b_1-a_1)! ((q-p)-(b_1-a_1))!}
\frac
{(b_1-1)! (q-b_1)!}
{(q-1)!} \\
& \quad \times
\frac
{((q-p) - (b_1-a_1))!}
{(b_2-a_2)! ((q-p)-(b_1-a_1) - (b_2-a_2))!}
\frac
{(b_2-1)! (q-b_1-b_2)!}
{(q-b_1-1)!} \\
& \qquad \cdots \\
& \quad \times
\frac
{((q-p) - (b_1-a_1) - \cdots - (b_{m-1}-a_{m-1}))!}
{(b_m-a_m)! ((q-p)-(b_1-a_1) - \cdots - (b_m-a_m))!} \\
& \quad \times
\frac
{(b_m-1)! (q-b_1- \cdots -b_m)!}
{(q-b_1-\cdots-b_{m-1}-1)!}.\\
\end{split}
\]
This expression can be rearranged to give
\[
\begin{split}
& \frac
{
(\theta+1)(\theta+2) \cdots (\theta+p-1)
}
{
(\theta+\alpha)(\theta+2 \alpha) \cdots (\theta+(m-1)\alpha)
}
\left[\prod_{k=1}^m (1-\alpha) (2-\alpha) \cdots (a_k-1-\alpha)\right]^{-1} \\
& \quad \times 
\frac
{(q-p)!}
{((q-p)-(b_1-a_1) - \cdots - (b_m-a_m))!} \\
& \qquad \times
\frac
{(b_1-1)!}
{(b_1-a_1)!}
\cdots
\frac
{(b_m-1)!}
{(b_m-a_m)!} \\
& \qquad \times
\frac
{(q-b_1)!}
{(q-1)!} 
\frac
{(q-b_1-b_2)!}
{(q-b_1-1)!}
\cdots
\frac
{(q-b_1-b_2-\cdots -b_m)!}
{(q-b_1-b_2-\cdots-b_{m-1}-1)!} \\
& \quad =
\frac
{
(\theta+1)(\theta+2) \cdots (\theta+p-1)
}
{
(\theta+\alpha)(\theta+2 \alpha) \cdots (\theta+(m-1)\alpha)
}
\left[\prod_{k=1}^m (1-\alpha) (2-\alpha) \cdots (a_k-1-\alpha)\right]^{-1} \\
& \qquad \times 
\left[(q-p+1) (q-p+2) \cdots (q-1)\right]^{-1} \\
& \qquad \times 
\prod_{k=1}^m (b_k-a_k+1) \cdots (b_k-1) 
\prod_{k=1}^{m-1} \left(q - \sum_{\ell=1}^k b_\ell\right).\\
\end{split}
\]

If $(\bbb_N)_{N \in \bN} = ((b_{N,1}, b_{N,2}, \ldots))_{N \in \bN}$ 
is a sequence from $\EE$
such that $\#\{N \in \bN : \bbb_N = \bbb\} < \infty$
for all $\bbb \in \EE$, then 
$\lim_{N \rightarrow \infty} \sum_{k=1}^\infty b_{N,k} = \infty$.
In this case, it is not hard to see that
$\lim_{N \rightarrow \infty} K(\aaa, \bbb_N)$ exists
for $\aaa \in \EE$ if and only if
\[
\lim_{N \rightarrow \infty}
\frac
{b_{N,k}}
{\sum_{\ell=1}^\infty b_{N,\ell}}
=: \rho_k
\]
exists for all $k \in \bN$.  Furthermore, for 
$\aaa = (a_1, a_2, \ldots, a_m, 0, \ldots)$ as above
\[
\begin{split}
\lim_{N \rightarrow \infty}
K(\aaa, \bbb_N)
& =
\frac
{
(\theta+1)(\theta+2) \cdots (\theta+p-1)
}
{
(\theta+\alpha)(\theta+2 \alpha) \cdots (\theta+(m-1)\alpha)
} \\
& \quad \times
\left[\prod_{k=1}^m (1-\alpha) (2-\alpha) \cdots (a_k-1-\alpha)\right]^{-1} \\
& \quad \times
\rho_1^{a_1-1} \cdots \rho_m^{a_m-1} \\
& \quad \times
(1 - \rho_1)(1-\rho_1-\rho_2) \cdots (1-\rho_1-\rho_2-\cdots-\rho_{m-1}) \\
& =: K(\aaa, \rho). \\
\end{split}
\]

Note that $\lim_{N \rightarrow \infty} K(\aaa, \bbb_N)$
exists for all $\aaa \in \EE$ 
if and only if the limit exists for all $\aaa \in \EE$ of the form $(1,\ldots,1,0,0,\ldots)$
(that is, for all $\aaa \in \EE$ with entries in $\{0,1\}$).
Note also that the extended Martin kernel has the property that
\[
K(\aaa, \rho) = 0
\Leftrightarrow
\begin{cases}
& a_k \ge 1 \; \text{for some $k$ with $\sum_{j=1}^{k-1} \rho_k = 1$}, \\
& a_k \ge 2 \; \text{for some $k$ with $\rho_k = 0$}.
\end{cases}
\]

Recall that if $\aaa$ is as above, then the transition probabilities
of the block size process are given by
\[
P(\aaa, \bbb)
=
\begin{cases}
\frac{\theta + \alpha m}{\theta+p},
& \quad \text{if $\bbb = (a_1, \ldots, a_m, 1, 0,\ldots)$},\\
\frac{a_k - \alpha}{\theta+p},
& \quad \text{if $\bbb = (a_1, \ldots, a_{k+1}, \ldots, a_m,0,\ldots)$}.
\end{cases}
\]

The Doob $h$-transform corresponding to the regular function
$h_\rho := K(\cdot, \rho)$ therefore
has transition probabilities given by
\[
\begin{split}
& P^{(h_\rho)}(\aaa, \bbb) \\
& \quad =
\begin{cases}
\frac{\theta + \alpha m}{\theta+p}
\frac{\theta + p}{\theta + m\alpha} 
(1 - \rho_1 - \cdots - \rho_m),
& \quad \text{if $\bbb = (a_1, \ldots, a_m, 1, 0,\ldots)$},\\
\frac{a_k - \alpha}{\theta+p}
(\theta+p) (a_k - \alpha)^{-1} \rho_k, 
& \quad \text{if $\bbb = (a_1, \ldots, a_k + 1, \ldots, a_m,0,\ldots)$}.
\end{cases} \\
\end{split}
\]
That is,
\[
\begin{split}
& P^{(h_\rho)}(\aaa, \bbb) \\
& \quad =
\begin{cases}
(1 - \rho_1 - \cdots - \rho_m),
& \quad \text{if $\bbb = (a_1, \ldots, a_m, 1, 0,\ldots)$},\\
\rho_k, 
& \quad \text{if $\bbb = (a_1, \ldots, a_k + 1, \ldots, a_m,0,\ldots)$}.
\end{cases} \\
\end{split}
\]

Note that the parameters $\alpha$ and $\theta$ do not appear 
in this expression for the transition probabilities.
It follows that for a given $M$ the block size
chains all arise as Doob $h$-transforms of each other.

We can build a Markov chain $(W_n)_{n \in \bN_0}$ with transition matrix
$P^{(h_\rho)}$ and initial state $\ccc$ as follows.  
Let $(V_n)_{n \in \bN}$ be a sequence of independent
identically distributed random variables taking
values in $[M] \cup \{\infty\}$ with $\bP\{V_n = k\} = \rho_k$
for $k \in [M]$ and $\bP\{V_n = \infty\} = 1 - \sum_\ell \rho_\ell$
(the latter probability is always $0$ when $M$ is finite).  
Define $(W_n)_{n \in \bN_0}$ inductively by setting $W_0 = \ccc$
and, writing $N_n := \inf\{j \in [M] : W_{nj} = 0\}$ with the
usual convention that $\inf \emptyset = \infty$,
\[
W_{n+1} 
=
\begin{cases}
(W_{n1}, \ldots, W_{nN_n}, 1, 0, \ldots),
& \quad \text{if $V_{n+1} > N_n$},\\
(W_{n1}, \ldots, W_{nk}+1, \ldots W_{nN_n}, 0, \ldots),
& \quad \text{if $V_{n+1} = k \le N_n$},\\
\end{cases}
\]
for $n \ge 0$.  It is clear from this construction and Kolmogorov's
zero-one law that the tail $\sigma$-field of the chain is trivial,
and so the regular function $h_\rho$ is extremal.

\subsection{Chinese restaurant trees}
\label{SS:CRP_and_trees}

Fix an admissible pair of parameters $\alpha$ and $\theta$ for
the two-parameter Chinese restaurant process.  Set $M := \infty$
when $0 \le \alpha < 1$ and
$M:= - \theta/\alpha \in \bN$ when $\alpha < 0$.  Put
$[M] := \bN$ for $M=\infty$ and $[M]:=\{1, \ldots, M\}$
otherwise.

Consider the trickle-down construction with the following
ingredients. The underlying directed acyclic graph $\II$ has vertex set
$[M]^\star:=\bigsqcup_{k=0}^\infty [M]^k$, the set of finite tuples
or words drawn from the alphabet $[M]$ (with the empty word $\emptyset$
allowed) and directed edges are defined in a manner analogous to that in Subsection~\ref{SS:RRT_and_perms} -- when $[M] = \bN$ we just
recover the complete Harris-Ulam tree of Subsection~\ref{SS:RRT_and_perms}.
Thus, $\II$ is a tree rooted at $\emptyset$ in which we may identify
$\beta(u)$, the set of offspring of vertex $u \in \II$, with $[M]$
for every vertex $u$.  With this identification, we take the routing
chain for every vertex to be the Chinese restaurant block size process 
with parameters $\alpha$ and $\theta$.

We may think of the state space $\SSS$
of the trickle-down chain $(X_n)_{n \in \bN_0}$ as
the set of finite subsets $\ttt$ of
$\II$ with the property that if a word 
$v = v_1 \ldots v_\ell \in \ttt$, then 
$v_1 \ldots v_{\ell-1} \in \ttt$ and 
$v_1 \ldots v_{\ell-1} k \in \ttt$ for $1 \le k < v_\ell$.
That is, when $[M] = \bN$ we may think of $\SSS$ as the set
of finite rooted Harris-Ulam trees from Subsection~\ref{SS:RRT_and_perms}, and when $M$ is finite
we get an analogous collection in which each individual
has at most $M$ offspring.

The partial order on $\II = [M]^\star$ is the one we get by
declaring that $u \le v$ for two words $u,v \in \II$ if and only
if $u = u_1 \ldots u_k$ and $v = v_1 \ldots v_\ell$
with $k \le \ell$ and $u_1 \ldots u_k = v_1 \ldots v_\ell$, 
just as for the complete binary tree.
By analogy with the notation introduced in
Example~\ref{E:BST} for finite rooted binary trees, 
write  $\#\ttt(u) := \#\{v \in \ttt : u \le v\}$
for $\ttt \in \SSS$ and $u \in [M]^\star$.

It follows from the discussion in Subsection~\ref{SS:CRPs} that 
Hypothesis~\ref{H:convergence_equivalent_ratios} holds.
We may identify the set $\II_\infty$ with 
\[
[M]^\infty
\sqcup
\bigsqcup_{k=0}^\infty ([M]^k \times \{\diamond\}^\infty)
\]
For each vertex $u \in \II$ the collection $\cS^u$
consists of all probability measures on $\beta(u)$
when $M$ is finite and all
subprobability measures on $\beta(u)$
when $M = \infty$.  
We may therefore 
identify $\partial \SSS$ with the probability measures
on $\II_\infty$ that assign all of their mass
to $[M]^\infty$ when $M$ is finite and with the set of all
probability measures on $\II_\infty$
when $M = \infty$.  We may extend the partial order by declaring that
$u < v$ for $u \in \II = [M]^\star$ and 
$v \in \II_\infty
=
[M]^\infty
\sqcup
\bigsqcup_{k=0}^\infty ([M]^k \times \{\diamond\}^\infty)
$
if and only if 
$u = u_1 \ldots u_k$ and $v = v_1 v_2 \ldots$
with $u_1 \ldots u_k = v_1 \ldots v_k$.

The following result summarizes 
the salient conclusions of the above discussion.

\begin{theorem}
Consider the Chinese restaurant tree process with parameters $(\alpha,\theta)$,
where $\alpha < 0$ and $\theta = -M \alpha$ for some $M \in \bN$
or $0 \le \alpha < 1$ and $\theta > -\alpha$, in which case we
define $M = \infty$.  We may identify the state space $\SSS$ of this
process as the set of finite rooted Harris-Ulam trees 
where the vertices
are composed of digits drawn from $[M]$.
When $M<\infty$ (resp. $M=\infty$),
the Doob-Martin boundary $\partial \SSS$ is homeomorphic 
to the space of probability  measures
on $[M]^\infty$ 
(resp. 
$[M]^\infty
\sqcup
\bigsqcup_{k=0}^\infty ([M]^k \times \{\diamond\}^\infty)$)
equipped with the topology of weak convergence.
With this identification, a sequence $(\ttt_n)_{n \in \bN}$
of finite rooted Harris-Ulam trees converges in
the topology of the Doob-Martin compactification
$\bar \SSS$ to the (sub)probability
measure $\mu$ in the Doob-Martin boundary  $\partial \SSS$ if and only if
$\lim_{n \to \infty} \# \ttt_n = \infty$ and
\[
\lim_{n \to \infty}
\frac{\# \ttt_n(u)}
{\# \ttt_n}
=
\mu\left\{v \in  
[M]^\infty
\sqcup
\bigsqcup_{k=0}^\infty ([M]^k \times \{\diamond\}^\infty): 
u < v\right\}
\]
for all $u \in [M]^\star$.
\end{theorem}

\begin{example}
Suppose that $M=\infty$.  Consider the sequence $(\ttt_n)_{n \in \bN}$
of finite rooted Harris-Ulam trees given by
$\ttt_n := \{\emptyset, 1, 2, \ldots, n, 21, 211, \ldots, 21^{n-1}\}$,
where the notation $21^k$ indicates $2$ followed by $k$ $1$s.
This sequence of trees converges in the topology of $\bar \SSS$ to
the probability measures on 
$[M]^\infty
\sqcup
\bigsqcup_{k=0}^\infty ([M]^k \times \{\diamond\}^\infty)$
that puts mass $\frac{1}{2}$ at the point $\diamond \diamond \diamond \ldots$
and mass $\frac{1}{2}$ at the point $2111 \ldots$.
\end{example}

\begin{remark}
The calculations of the extended Martin kernel and Doob $h$-transform
transition probabilities associated with a given 
$\mu \in \partial \SSS$ are straightforward but notationally
somewhat cumbersome, so we omit them.  They show that there is
the following ``trickle-up'' construction
of a Markov chain $(W_n)_{n \in \bN_0}$ 
with initial state $\www \in \SSS$
and the $h$-transform transition probabilities
(compare the analogous construction for the Chinese restaurant
process itself in Subsection~\ref{SS:CRPs}).

Let $(V^n)_{n \in \bN}$ be a sequence
of independent, identically distributed $\II_\infty$-valued random
variables with common distribution $\mu$. Suppose that $\SSS$-valued
random variables $\www=:W_0 \subset \ldots \subset W_n$ 
have already been defined.  Put 
$H(n+1):= \max\{h \in \bN : V_1^{n+1} \ldots V_h^{n+1} \in W_n\}$,
with the convention $\max \emptyset = 0$, and 
$M(n+1):= \max\{m \in \bN : V_1^{n+1} \ldots V_{H(n+1)}^{n+1} m \in W_n\}$,
again with the convention $\max \emptyset = 0$.
Set $W_{n+1} := W_n \cup \{V_1^{n+1} \ldots V_{H(n+1)}^{n+1} (M(n+1) + 1)\}$.
For example, if $\www = \emptyset$ and $\mu$ is the unit point mass
at the sequence $\diamond \diamond \ldots$, then $W_n = \{\emptyset, 1, \ldots, n\}$
for $n \ge 1$; that is, $W_n$ consists of the root $\emptyset$ and the first $n$
children of the root.

It is clear from the Kolmogorov zero-one law
that the tail $\sigma$-field of $(W_n)_{n \in \bN_0}$ is
trivial for any $\mu$, and so any $\mu$ is extremal.
\end{remark}

\begin{remark} By analogy with the definition of the 
BST process in Section  \ref{S:BST_and_DST}, we define $T_n$ 
to be
 the set of vertices occupied by time $n$ 
 (so that $T_0 = \{\emptyset\}$).
Put, for each vertex $u$, $T_n(u) := \{v \in T_n : u \le v \}$. The distribution of the random probability measure $R$ on $[M]^\infty$ defined by
%$\lim_{n \to \infty} X_n$ 
$R\{w\in [M]^\infty\, : \, u < w \}  := \lim_{n\to \infty}\# T_n(u)/\#T_n, \quad u\in [M]^\star,$ 
% starting from $X_0 = \{\emptyset\}$
may be derived from known properties of the two-parameter Chinese restaurant process
(see, for example, Theorem 3.2 of \cite{MR2245368}).
For $v \in [M]^\star$ put 
\[
(U_{v1}, U_{v2}, U_{v3}, \ldots) 
:= 
(B_{v1}, (1-B_{v1})B_{v2}, (1-B_{v1})(1-B_{v2})B_{v3}, \ldots),
\]
where the random variables $B_{vk}$, $v \in [M]^\star$, $k \in [M]$, are
independent and $B_{uk}$ has the beta distribution with parameters
$(1-\alpha, \theta + k \alpha)$.  That is, the sequence
$(U_{vk})_{k \in [M]}$ has a {\em Griffiths--Engen--McCloskey (GEM) distribution}
with parameters $(\alpha,\theta)$.
Then, $R$
%$\lim_{n \to \infty} X_n$
is distributed as the random probability measure on $[M]^\infty$ that
for each $u \in [M]^\star$ assigns mass $\prod_{\emptyset < v \le u} U_v$ to the set 
$\{w \in [M]^\infty : u < w\}$.
\end{remark}

\section{Mallows chains}
\label{S:Mallows}

\subsection{Mallows' $\phi$ model for random permutations and the associated tree}
\label{SS:Mallow_and_perms}

The {\em $\phi$ model} of Mallows \cite{MR0087267}
produces a random permutation of the set $[n]$
for some integer $n \in \bN$.
One way to describe the model is the following.

We place the elements of $[n]$ successively into 
$n$ initially vacant ``slots'' labeled by
$[n]$ to obtain
a permutation of $[n]$ (if the number
$i$ goes into slot $j$, then the permutation
sends $i$ to $j$).  To begin with,
each slot is equipped with a Bernoulli
random variable.  These random variables
are obtained by taking $n$ independent Bernoulli
random variables with common success probability $0<p<1$
and conditioning on there being at least $1$ success.  The
number $1$ is placed in the first slot for which the
associated Bernoulli random variable is a success. 
Thus, the probability that there are $k$
vacant slots to the left of $1$ is 
\[
\frac
{(1-p)^k p}
{1 - (1-p)^n},
\quad 0 \le k \le n-1.
\]
Now equip the remaining $n-1$ vacant slots (that is, every slot except
the one in which $1$ was placed) with a set of
Bernoulli random variables that is independent
of the first set.  These random variables
are obtained by taking $n-1$ independent Bernoulli
random variables with common success probability $p$
and conditioning on there being at least $1$ success.
Place the number $2$ in the first vacant
slot for which the associated Bernoulli is a success.
The probability that there are $k$
vacant slots to the left of $2$ is 
\[
\frac
{(1-p)^k p}
{1 - (1-p)^{n-1}},
\quad 0 \le k \le n-2.
\]
Continue in this fashion until all the slots have been filled.

The analogous procedure can be used to produce
a permutation of $\bN$.  Now the procedure
begins with infinitely many
slots labeled by $\bN$, and at each stage there
is no need to condition on the almost sure event
that there is at least one success.  After each $m \in \bN$ is inserted,
the current number of vacant slots to the left of the slot in which
$m$ is placed is distributed as the number of failures
before the first success in independent Bernoulli
trials with common success probability $p$, and these
random variables are independent.  We note that this
distribution on permutations of $\bN$ appears in
\cite{MR2683626} in connection with $q$-analogues of
de Finetti's theorem.

Suppose now that $\pi$ is a permutation of the set $S$, where
$S=[n]$ or $S=\bN$.
Let $I(\pi) := \pi(1)$.  That is, if we think of 
$\pi$ as a list of the elements of $S$ in some order,
then $I(\pi)$ is the index of $1$.
Put $S^L(\pi) := \{i : \pi(i) < \pi(1)\}$
and $S^R(\pi) := \{i : \pi(i) > \pi(1)\}$.  
Note that $\pi$ maps $S^L(\pi)$ to $\{1, \ldots, I(\pi) - 1\}$
and $S^R(\pi)$ to $I(\pi) + \{1, \ldots, n - I(\pi)\}$ 
or $I(\pi) + \bN$, and that
$S^L(\pi)$ (respectively, $S^R(\pi)$) is the set of elements
of $S$ that appear before (respectively, after) $1$ 
in the ordered listing of $S$ defined by $\pi$.  

If $S = [n]$, write
$\psi^L(\pi)$ for the unique increasing bijection
from  $\{1, \ldots, I(\pi) - 1\}$ to $S^L(\pi)$
and
$\psi^R(\pi)$ for the unique increasing bijection
from $\{1, \ldots, n - I(\pi)\}$  to $S^R(\pi)$.
If $S = \bN$, define $\psi^L(\pi)$ and $\psi^R(\pi)$
similarly, except that now $\psi^R(\pi)$ maps
$\bN$ to $S^R(\pi)$.

Define permutations $\sigma^L(\pi)$ and $\sigma^R(\pi)$ of
$\{1, \ldots, I(\pi) - 1\}$ and $\{1, \ldots, n - I(\pi)\}$
(if $S = [n]$) or
$\{1, \ldots, I(\pi) - 1\}$ and $\bN$ (if $S = \bN$)
by requiring that $\pi$ restricted to $S^L(\pi)$
is $\psi^L(\pi) \circ \sigma^L(\pi) \circ (\psi^L(\pi))^{-1}$ and that
$\pi$ restricted to $S^R(\pi)$
is $\psi^R(\pi) \circ \sigma^R(\pi) \circ (\psi^R(\pi))^{-1}$.
In other words, $\sigma^L(\pi)(i)$
is the index  of the $i^{\mathrm{th}}$ smallest element
of $S^L(\pi)$ in the ordered listing of $S$ defined by $\pi$,
and $I(\pi) + \sigma^R(\pi)(i)$
is the index of the $i^{\mathrm{th}}$ smallest element
of $S^R(\pi)$ in the ordered listing of $S$ defined by $\pi$.

Note that $\pi$ is uniquely specified by the objects $I(\pi)$,
$S^L(\pi)$, $S^R(\pi)$, $\sigma^L(\pi)$, and $\sigma^R(\pi)$.

The following lemma is immediate from the construction
of the Mallows model.

\begin{lemma}
\label{L:Mallows_decomposition}
Suppose that $\Pi$ is a random permutation of 
either $[n]$ or $\bN$ that is
distributed according to the Mallows
model with parameter $p$.  Then, conditional
on $(I(\Pi), S^L(\Pi), S^R(\Pi))$, 
the permutations $\sigma^L(\Pi)$ and $\sigma^L(\Pi)$
are independent and distributed according to the
Mallows model with parameter $p$.
\end{lemma}

Recall from the description of the BST process
in the Introduction how it is possible to construct 
from a  permutation $\pi$ of $[n]$
a subtree of the complete binary tree $\{0,1\}^*$ that contains the
root $\emptyset$ and has $n$ vertices.  
The procedure actually produces a tree labeled with the elements
of $[n]$, but we are only interested in the underlying unlabeled tree.
Essentially the same construction produces an infinite rooted binary tree
labeled with $\bN$ from a permutation $\pi$ of $\bN$.  This tree
has the property that if a vertex $u = u_1 \ldots u_k$ belongs
to the tree, then there only finitely many vertices $v$ such that
$u_1 \ldots u_k 0 \le v$.

The following result is immediate from 
Lemma~\ref{L:Mallows_decomposition} and the recursive
nature of the procedure that produces a rooted subtree of $\{0,1\}^*$ from
a permutation.

\begin{proposition}
\label{P:Mallows_as_trickle}
Let $(X_n)_{n \in \bN_0}$ be the Markov chain
that results from the trickle-down construction applied when the
directed graph $\II$ is the infinite complete binary tree
$\{0,1\}^\star$ and all the routing chains have the common transition matrix
$Q$ on the state space $\bN_0 \times \bN_0$, where
\[
Q((i,0), (i+1,0)) := (1-p), \quad \text{for all $i \ge 0$},
\]
\[
Q((i,0), (i,1)) := p, \quad \text{for all $i \ge 0$},
\]
and
\[
Q((i,j), (i,j+1)) := 1, \quad \text{for all $i \ge 0$ and $j \ge 1$}.
\]
We may regard $(X_n)_{n \in \bN_0}$ as a Markov chain taking values
in the set of finite subtrees of $\{0,1\}^\star$ that contain the root $\emptyset$,
in which case $\{\emptyset\} = X_0 \subseteq X_1 \subseteq \ldots$ and $X_\infty := \bigcup_{n \in \bN_0} X_n$ is an infinite subtree of $\{0,1\}^\star$ that contains $\emptyset$.
Then, $X_\infty$ has the same distribution as 
the random tree constructed from
a random permutation of $\bN$ that is 
distributed according to the Mallows model with parameter $p$.
\end{proposition}

We call the Markov chain $(X_n)_{n \in \bN_0}$ 
of Proposition~\ref{P:Mallows_as_trickle} the {\em Mallows tree process}.

\subsection{Mallows urns}
\label{SS:Mallows_urn_compact}

Consider the Markov chain on $\bN_0 \times \bN_0$ with transition matrix
$Q$ introduced in
Proposition~\ref{P:Mallows_as_trickle}.  We call this
chain the {\em Mallows urn}, because its role as a 
routing chain for the Mallows tree process is similar
to that played by the P\'olya urn in the construction of the BST
process.  When started from $(0,0)$, a sample path of the Mallows urn
process looks like $(0,0), (1,0), \ldots, (K,0), (K,1), (K,2), \ldots$,
where $\bP\{K=k\}= (1-p)^k p$ for $k \in \bN_0$.

The probability that the Mallows urn process visits the
state $(k,\ell)$ starting from the state $(i,j)$ is
\[
\begin{cases}
(1-p)^{k-i}, & \text{if $i \le k$, $j=0$ and $\ell = 0$},\\
(1-p)^{k-i} p, & \text{if  $i \le k$, $j=0$ and $\ell \ge 1$},\\
1, & \text{if $i = k$ and $1 \le j \le \ell$,} \\
0, & \text{otherwise.}
\end{cases}
\]
In particular, the probability that the process visits
$(k,\ell)$ starting from $(0,0)$ is
\[
\begin{cases}
(1-p)^{k}, & \text{if $\ell = 0$},\\
(1-p)^{k} p, & \text{if $\ell \ge 1$.}
\end{cases}
\]

Taking, as usual, $(0,0)$ as the reference state, 
the Martin kernel for the Mallows urn process is thus
\[
K((i,j),(k,\ell))
:=
\begin{cases}
(1-p)^{-i}, & \text{if $i \le k$, $j=0$ and $\ell = 0$},\\
(1-p)^{-i}, & \text{if  $i \le k$, $j=0$ and $\ell \ge 1$},\\
(1-p)^{-k} p^{-1}, & \text{if $i = k$ and $1 \le j \le \ell$}, \\
0, & \text{otherwise,}
\end{cases}
\]
or, equivalently,
\begin{equation}
\label{Mallows_urn_kernel}
K((i,j),(k,\ell))
=
\begin{cases}
(1-p)^{-i}, & \text{if $i \le k$ and $j=0$},\\
(1-p)^{-i} p^{-1}, & \text{if $i = k$ and $1 \le j \le \ell$}, \\
0, & \text{otherwise.}
\end{cases}
\end{equation}

It follows that if $((k_n,\ell_n))_{n \in \bN_0}$ is a sequence
for which $k_n + \ell_n \rightarrow \infty$ then, in order for the
sequence $(K((i,j),(k_n,\ell_n)))_{n \in \bN_0}$
to converge, it must either be that  $k_n = k_\infty$ for
some $k_\infty$ for all $n$ sufficiently large
and $\ell_n \rightarrow \infty$, in which case the limit is
\[
\begin{cases}
(1-p)^{-i}, & \text{if $i \le k_\infty$ and $j=0$},\\
(1-p)^{-i} p^{-1}, & \text{if $i = k_\infty$ and $j \ge 1$}, \\
0, & \text{otherwise,}
\end{cases}
\]
or that $k_n \rightarrow \infty$
with no restriction on $\ell_n$, in which case the limit is
\[
\begin{cases}
(1-p)^{-i}, & \text{if $j=0$},\\
0, & \text{otherwise.}
\end{cases}
\]

Consequently, the Doob-Martin compactification 
$\overline{\bN_0 \times \bN_0}$ of the state space of the Mallows urn process
is such that the Doob-Martin boundary
$\partial (\bN_0 \times \bN_0) := \overline{\bN_0 \times \bN_0} \setminus \bN_0 \times \bN_0$
can be identified with $\bN_0 \cup \{\infty\}$, the usual
one-point compactification of $\bN_0$.

With this identification, 
the state space of the 
$h$-transformed process corresponding to the boundary point $k \in \bN_0$
is $\{(0,0), (1,0), \ldots, (k,0)\} \cup \{(k,1), (k,2), \ldots\}$ and
the transition probabilities are
\[
Q^h((i,0), (i+1,0)) = ((1-p)^{-i})^{-1} (1-p) (1-p)^{-(i+1)} = 1, \quad \text{for $0 \le i \le k-1$},
\]
\[
Q((k,0), (k,1)) = ((1-p)^{-i})^{-1} p (1-p)^{-i} p^{-1} = 1, 
\]
and
\[
Q((k,j), (k,j+1)) = ((1-p)^{-i} p^{-1})^{-1} 1 (1-p)^{-i} p^{-1} = 1, \quad \text{for all $j \ge 1$}.
\]
Thus, a realization of the $h$-transformed process starting from
$(0,0)$ is the deterministic path 
$(0,0), (1,0), \ldots, (k,0), (k,1), (k,2), \ldots$.

Similarly, the state space of the
$h$-transformed process corresponding to the boundary point $\infty$
is $\{(0,0), (1,0), (2,0), \ldots\}$
and a realization of the $h$-transformed process starting from
$(0,0)$ is the deterministic path $(0,0), (1,0), (2,0), \ldots$.

\subsection{Mallows tree process}
\label{SS:Mallows_tree_compact}

Recall from Example~\ref{E:dag_tree} that we may identify the state space
$\SSS$ of the Mallows tree process with the set of finite subtrees
of the complete binary tree $\II = \{0,1\}^\star$ that contain the root $\emptyset$, 
and with this identification the partial order $\preceq$ is just subset containment.

Consider $\sss$ in $\SSS$ and a sequence
$(\ttt_n)_{n \in \bN_0}$ from $\SSS$ such that
$\# \ttt_n \to \infty$ as $n \to \infty$.  Given
a vertex $u$ of $\{0,1\}^\star$ write, as in Section~\ref{S:BST_and_DST}, 
$\# \sss(u) := \{v \in \sss: u \le v\}$ and define
$\# \ttt_n(u)$ similarly.  Note that in this setting the consistency
condition \eqref{E:consistency_condition} becomes
$(\# \sss(u) - 1)_+ =\# \sss(u0) +\# \sss(u1)$ and 
$(\# \ttt_n(u) - 1)_+ = \# \ttt_n(u0) + \# \ttt_n(u1)$.

Write 
\[
L(\sss) := \sum_{u \in \{0,1\}^*} \# \sss(u0).
\]
When $\sss \subseteq \ttt_n$, put 
\[
M(\sss, \ttt_n) 
:= \# \{u \in \{0,1\}^* : 
\# \sss(u0) = \# \ttt_n(u0), \; \# \sss(u1) \ge 1\}
\]
and 
\[
I(\sss, \ttt_n) :=
\begin{cases}
1, & \text{if $\# \sss(u0) = \# \ttt_n(u0)$ whenever $\# \sss(u1) \ge 1$,}\\
0, & \text{otherwise.}
\end{cases}
\]

{}From Corollary~\ref{C:Martin_kernel_is_product} and 
\eqref{Mallows_urn_kernel}, the Martin kernel of the Mallows tree process is
\[
K(\sss,\ttt_n)
:=
\begin{cases}
(1-p)^{-L(\sss)} p^{-M(\sss,\ttt_n)} I(\sss,\ttt_n), & \text{if $\sss \subseteq \ttt_n$,}\\
0, & \text{otherwise.}
\end{cases}
\]

Note that if $\sss \subseteq \ttt_n$ and $\# \sss(u0) = \# \ttt_n(u0)$,
then $\{v \in \sss : u0 \le v\} = \{v \in \ttt_n : u0 \le v\}$.
Therefore, when
$\sss \subseteq \ttt_n$, $M(\sss,\ttt_n)$ counts the 
number of vertices of the form
$u0$ such that the subtree below $u0$ in $\sss$ is the same
as the subtree below $u0$ in $\ttt_n$ and $u1 \in \sss$. 
Similarly, $I(\sss,\ttt_n) = 1$
if and only if 
for all vertices of the form $u0$, 
the subtree below $u0$ in $\sss$ is the same
as the subtree below $u0$ in $\ttt_n$ whenever
$u1 \in \sss$.  
Hence, if $\sss \subseteq \ttt_n$, then
\[
p^{-M(\sss,\ttt_n)} I(\sss,\ttt_n) = p^{-N(\sss)} I(\sss,\ttt_n),
\]
where $N(\sss) :=  \# \{u \in \{0,1\}^* : u1 \in \sss\}$.
Thus,
\[
K(\sss,\ttt_n)
=
\begin{cases}
(1-p)^{-L(\sss)} p^{-N(\sss)} I(\sss,\ttt_n), & \text{if $\sss \subseteq \ttt_n$,}\\
0, & \text{otherwise.}
\end{cases}
\]

Suppose that $\# \ttt_n(0) \rightarrow \infty$. For any
$\sss$ such that $1 \in \sss$,  $I(\sss,\ttt_n)$
must be $0$ for all $n$ sufficiently large, 
because the subtree below $0$
in $\sss$ cannot equal the subtree below $0$ in $\ttt_n$
for all $n$.

On the other hand, if $1 \notin \sss$,
then $K(\sss,\ttt_n) = K(\sss,  \tilde{\ttt}_n)$, where
$\tilde{\ttt}_n$ is the tree obtained from $\ttt_n$ by deleting
all vertices $v$ with $1 \le v$.
Consequently, if $\# \ttt_n(0) \rightarrow \infty$, then in order
to check whether $K(\sss,\ttt_n)$ converges for all $\sss \in \SSS$,
it suffices to replace $\ttt_n$ by $\tilde{\ttt}_n$
and restrict consideration to $\sss$ such that 
$1 \notin \sss$.  Moreover,
the limits of $K(\sss,\ttt_n)$ and $K(\sss,\tilde{\ttt}_n)$ are the
same, so the sequences $(\ttt_n)_{n \in \bN}$
and  $(\tilde{\ttt}_n)_{n \in \bN}$
correspond to the same point in
the Doob-Martin compactification.

Now suppose that $\# \ttt_n(0) \not \rightarrow \infty$
(so that $\# \ttt_n(1) \rightarrow \infty$ must hold). 
It is clear that if $K(\sss,\ttt_n)$ converges for all
$\sss \in \SSS$ with $1 \notin \sss$, 
then the sets $\{v \in \ttt_n : 0 \le v\}$ are equal
for all $n$ sufficiently large.

Let $\hat{\ttt}_m$ be the subtree of
$\ttt_m$ obtained by deleting from $\ttt_m$ any vertex $v$
such that $u1 \le v$ for some $u$ with 
$\# \ttt_n(u0) \rightarrow \infty$. 
Applying the above arguments recursively, 
a necessary and sufficient condition
for the sequence $(\ttt_n)_{n \in \bN_0}$
to converge to a point in the Doob-Martin compactification is that
whenever  $ \# \hat{\ttt}_n(u0) \not \rightarrow \infty$ for some $u$,
then the set $\{v \in \hat{\ttt}_n : u0 \le v\}$ are equal for
all $n$ sufficiently large.
Moreover, the sequences
$(\ttt_n)_{n \in \bN_0}$ and $(\hat{\ttt}_n)_{n \in \bN_0}$
converge to the same limit point.

\begin{figure}[htbp]
	\centering
		\includegraphics[width=1.00\textwidth]{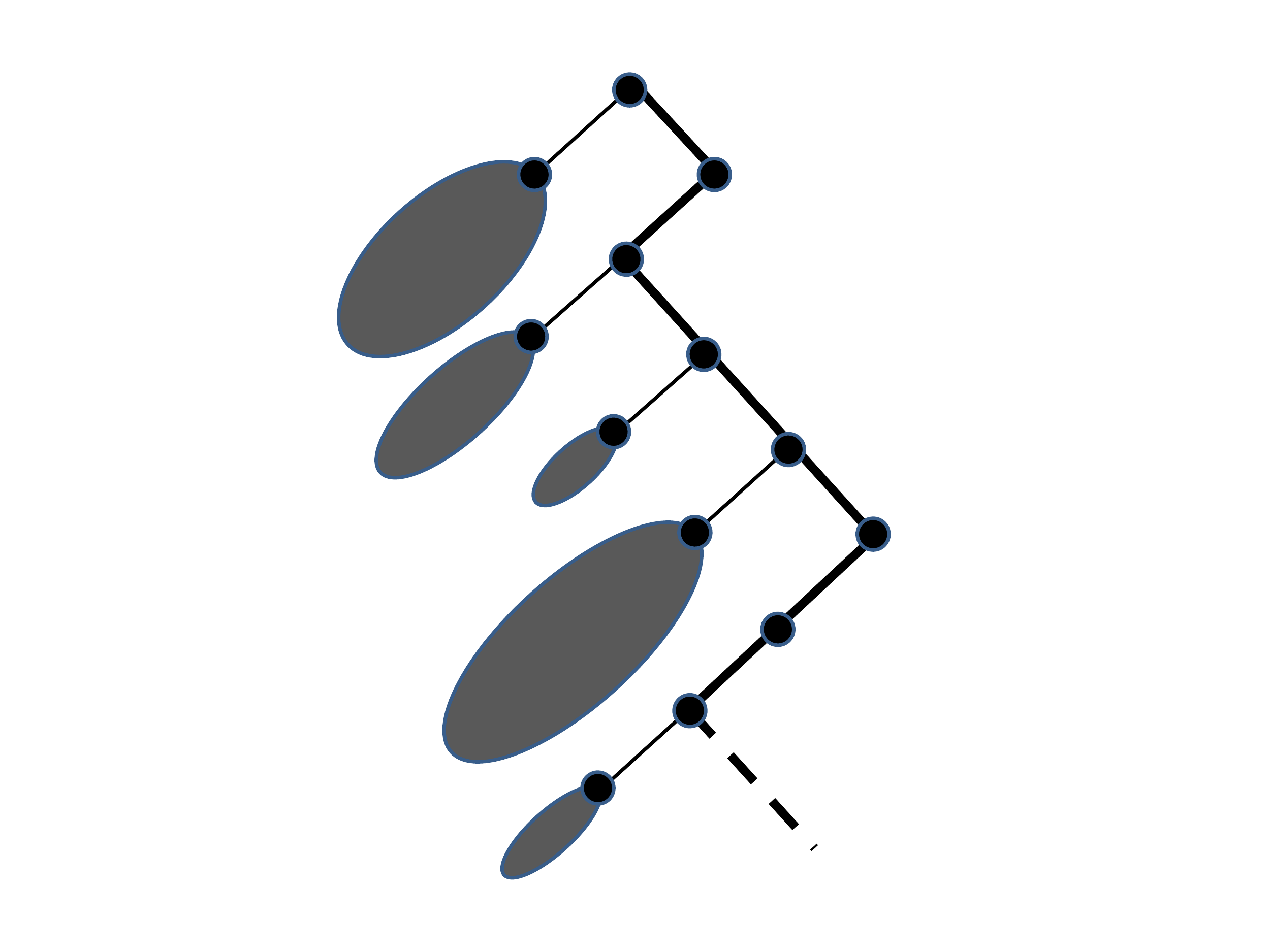}
	\caption{A typical element of the set $\TT$ of infinite rooted binary
	trees with a single infinite spine.  The beginning of the infinite spine is the
	thick line.  The ``blobs'' hanging off the left side of the spine represent
	finite subtrees.  Any vertex that has a ``left'' child
	with infinitely many descendants has no ``right'' child.}
	\label{fig:single_spine_tree}
\end{figure}

Suppose that $(\ttt_n)_{n \in \bN_0}$ and hence 
$(\hat{\ttt}_n)_{n \in \bN_0}$ converges in the 
Doob-Martin compactification.  Set 
\[
\ttt_\infty = \bigcup_{m \in \bN_0} \bigcap_{n \ge m} \hat{\ttt}_n.
\]
Note that $\ttt_\infty$ is an infinite subtree of $\{0,1\}^\star$
containing the root $\emptyset$ and if $\# \ttt_\infty(u0) = \infty$
for some $u \in \{0,1\}^\star$, then $\# \ttt_\infty(u1) = 0$
(that is, $u1 \notin \ttt_\infty$).
Equivalently, there is a unique infinite
path $\emptyset =u_0 \to u_1 \to u_2 \to \ldots$ in $\ttt_\infty$
and this path is such that if $u_n = w_1 \cdots w_{n-1} 0$, then
$w_1 \cdots w_{n-1} 1 \notin \ttt_\infty$.  Let $\TT$
be the set of subtrees with this property.  We can think of a subtree
$\ttt \in \TT$ as consisting of the infinite ``spine'' 
$\emptyset =v_0 \to v_1 \to v_2 \to \ldots$ to which are attached
the finite subtrees $\{v \in \ttt : v_n 0 \le v\}$ 
for those $n \in \bN_0$ such that $v_{n+1} = v_n 1$ -- see Figure~\ref{fig:single_spine_tree}.
 
We have
\[
\lim_{n \rightarrow \infty} K(\sss,\ttt_n)
=
\lim_{n \rightarrow \infty} K(\sss,\hat{\ttt}_n)
=
\begin{cases}
(1-p)^{-L(\sss)} p^{-N(\sss)} I(\sss,\ttt_\infty), & \text{if $\sss \subset \ttt_\infty$,}\\
0, & \text{otherwise,}
\end{cases}
\]
where $I(\sss,\ttt_\infty)$ is defined to be $1$ or $0$
depending on whether or not
for all vertices of the form $u0$ with $u1 \in \sss$
the subtree below $u0$ in $\sss$ is the same
as the subtree below $u0$ in $\ttt_\infty$.

Recall that we write
$|u|$ for the length of a word $u\in \{0,1\}^\star$; that is,
$|u| = k$ when $u=u_1 \ldots u_k$. 
Note that if $\ttt \in \TT$, then the sequence
$(\ttt_n)_{n \in \bN_0}$ in $\SSS$ defined by 
$\ttt_n := \{u \in \ttt : |u| \le n\}$ converges in
the Doob-Martin compactification of $\SSS$ and the tree $\ttt_\infty$
constructed from this sequence is just $\ttt$.

Finally, observe that if we extend $K(\sss,\ttt)$
for $\sss \in \SSS$ and $\ttt \in \TT$ by
\[
K(\sss,\ttt)
:=
\begin{cases}
(1-p)^{-L(\sss)} p^{-N(\sss)} I(\sss,\ttt), & \text{if $\sss \subset \ttt$,}\\
0, & \text{otherwise,}
\end{cases}
\]
then for any distinct $\ttt', \ttt'' \in \TT$ there
exists $\sss \in \SSS$ such that 
$K(\sss,\ttt') \ne K(\sss,\ttt'')$.

The important elements of the above discussion are contained in the
following result.

\begin{theorem}
\label{T:Mallows}
Consider the Mallows tree chain with state space $\SSS$ consisting
of the set of finite rooted binary trees.  Let
$\TT$ be the set of infinite rooted binary trees $\ttt$ 
such that $u1 \in \ttt$ for some $u \in \{0,1\}^\star$
implies $\# \ttt(u0) < \infty$.  
Equip $\SSS \sqcup \TT$ with the topology
generated by the maps $\Pi_n : \SSS \sqcup \TT \to \SSS$, 
$n \in \bN_0$, defined by
$\Pi_n(\ttt) := \{u \in \ttt : |u| \le n\}$, where on
the right we equip the countable set $\SSS$ with the discrete topology. 
The Doob-Martin compactification $\bar \SSS$ is homeomorphic to 
$\SSS \sqcup \TT$, and this homeomorphism identifies
the Doob-Martin boundary $\partial \SSS$ with $\TT$.
\end{theorem}

\begin{remark} 
The limit in the Doob-Martin topology of
the Mallows tree chain $(X_n)_{n \in \bN_0}$
started from the trivial tree $\emptyset$  
is just the $\TT$-valued random variable
$X_\infty := \bigcup_{n \in \bN_0} X_n$ introduced in
Proposition~\ref{P:Mallows_as_trickle}. Almost surely, the 
spine of $X_\infty$ (that is, the unique infinite
path from the root $\emptyset$) 
is equal to the rightmost path $\emptyset\to 1\to 11\to 111\ldots$
in the complete infinite binary tree.
\end{remark}

\begin{remark}
It is straightforward
to check that each of the harmonic functions $K(\cdot, \ttt)$,
$\ttt \in \TT$ is extremal.  If we
order the alphabet $\{0,1\}$ so that $0$ comes before $1$
and equip the set of words $\{0,1\}^\star$ with the 
corresponding lexicographic order, then 
the state space of the $h$-transformed process 
corresponding to an infinite tree
$\ttt \in \TT$ is the set of finite subtrees $\sss$
of $\ttt$ such that if $u \in \sss$, then every predecessor
of $u$ in the lexicographic order also belongs to 
$\sss$.  A realization of the $h$-transformed process started
from $\emptyset$ is the deterministic path that adds the vertices
of $\ttt$ one at a time in increasing lexicographic order.
\end{remark}

\begin{remark}
As in the BST and DST cases, the Mallows tree process can be regarded as a Markov chain 
which moves from a tree $\ttt$ to a tree $\sss$ of the form $\sss= \ttt\sqcup \{v\}$, where 
the new vertex $v$ is an external vertex of $\ttt$ (see the discussion following \eqref{E:BST_prob}). 
This implies that the transition probabilities can be coded by a function $p$ that 
maps pairs $(\ttt,v)$, $\ttt\in\II$ and $v$ an external vertex of $\ttt$, to the probability 
that the chain moves from $\ttt$ to $\ttt\sqcup \{v\}$.

In the BST case one of the $|\ttt|+1$ external vertices of $\ttt$ is chosen
uniformly at random, that is, $p(v|\ttt)=1/(|\ttt|+1)$, whereas we have
$p(v|\ttt)=2^{-|v|}$ in the DST case. For Mallows trees, we 
have the following stochastic mechanism. 
Let $u$ be the vertex of $\ttt$ that is
greatest in the lexicographic order. Denote by $i_1<\cdots<i_\ell$  
the indices at which the corresponding entry of $u$ is
a $0$ (we set $\ell=0$ if every entry of $u$ is a
$1$). Write $v_j$, $1\le j\le \ell$, for the external vertices of $\ttt$ that arise if 
the $0$ in position $i_j$ is changed to $1$. 
Put $v_{\ell+1}:=v1$ and $v_{\ell+2}:=v0$. Then, 
we choose $v_j$ with probability $p^{i_j}$, $j=1,\ldots,\ell$, and $v_{\ell+1}$ and $v_{\ell+1}$ with 
probabilities $rp$ and $r(1-p)$ respectively, where $r:=1-\sum_{j=1}^\ell p^{i_j}$.    

Note that not all Markov chains of the vertex-adding type can be represented as
trickle-down processes. Indeed, a distinguishing feature of the trickle-down chains within 
this larger class is the fact that the restriction of the function $v\to p(v|\ttt)$ 
to the external vertices of the left subtree of $\ttt$ depends on $\ttt$ only via
the number of vertices  in the right subtree of $\ttt$. Similar restrictions hold
with left and right interchanged, and also for the subtrees of non-root vertices.
\end{remark}

\section{$q$-binomial chains}
\label{S:q_chain}

\subsection{$q$-binomial urns}
\label{SS:q_walk}

Fix parameters $0 < q < 1$ and $0 < r < 1$, and 
define a transition matrix $Q$ for the state space
$\bN_0 \times \bN_0$  by
\[
Q((i,j), (i+1,j)) = r q^j
\]
and 
\[
Q((i,j), (i,j+1)) = 1 - r q^j
\]
for $(i,j) \in \bN_0 \times \bN_0$.  
We note that this $2$-parameter family of processes
is a special case of the $3$-parameter family studied
in \cite{MR1452938}, where it is shown to have
a number of interesting connections with graph theory.
In the next subsection,
we use Markov chains with the transition matrix $Q$
as the routing chains for a trickle-down process on $\II = \{0,1\}^\star$
in the same way that we have used the P\'olya and Mallows urn processes.

Note that, by a 
simple Borel-Cantelli argument, almost surely 
any sample path of a Markov
chain $(Y_n)_{n \in \bN_0} = ((Y_n',Y_n''))_{n \in \bN_0}$ 
with transition matrix $Q$ is such that 
$Y_N' = Y_{N+1}' = Y_{N+2}' = \ldots$ for some $N$ (so that $Y_{N+1}'' = Y_N'' +1, \,
Y_{N+2}'' = Y_N'' + 2, \, \ldots$).

We want to compute the probability that the chain goes from $(i,j)$ to $(k,\ell)$
for $i \le k$ and $j \le \ell$.  

Observe that the probability
the chain goes from $(i,j)$ to $(k,\ell)$ via $(k,j)$ is
\[
R((i,j), (k,\ell)) : = (r q^j)^{k-i} (1 - r q^j) (1 - r q^{j+1}) \cdots (1 - r q^{\ell - 1}).
\]

Observe also that if $S(i,j)$ is the probability the chain goes from $(i,j)$ to
$(i+1,j+1)$ via $(i+1,j)$ and $T(i,j)$ is the probability the chain goes from $(i,j)$ to
$(i+1,j+1)$ via $(i,j+1)$, then $T(i,j) = q S(i,j)$.  It follows by repeated applications
of this observation that the probability the chain goes from $(i,j)$ to $(k,\ell)$ along some
``north-east'' lattice path $\sigma$ is
\[
q^{A(\sigma)} R((i,j), (k,\ell)),
\]
where $A(\sigma)$ is the area in the plane above the line segment $[i,k] \times \{j\}$
and below the curve obtained by a piecewise linear interpolation of $\sigma$.
Hence, the probability that the chain hits $(k,\ell)$ starting from $(i,j)$ is
\[
\sum_\sigma q^{A(\sigma)} R((i,j), (k,\ell)),
\]
where the sum is over all ``north-east'' lattice paths $\sigma$ from $(i,j)$ to $(k,\ell)$.

As explained in \cite[Chapter 10]{MR2000g:33001},
the evaluation of the sum is a consequence of the non-commutative 
$q$-binomial theorem of
\cite{MR14:768g} (see also
\cite{MR38:4329}), and
\[
\sum_\sigma q^{A(\sigma)}
=
\frac
{(1-q)(1-q^2) \cdots (1 - q^{(k-i)+(\ell-j)})}
{(1-q)(1-q^2) \cdots (1 - q^{(k-i)}) \times (1-q)(1-q^2) \cdots (1 - q^{(\ell-j)})}.
\]

Taking, as usual, $(0,0)$ as the reference state, the Martin kernel for the 
chain is thus
\[
\begin{split}
K((i,j),(k,\ell))
& =
\frac
{(1-q^{k-i+1}) \cdots (1-q^k) \times (1-q^{\ell-j+1}) \cdots (1-q^\ell)}
{ (1 - q^{(k-i)+(\ell-j)+1}) \cdots (1 - q^{k+\ell})} \\
& \quad \times r^{-i} q^{j(k-i)}
\frac
{1}
{(1 - r)(1 - r q) \cdots (1 - r q^{j-1})},
\end{split}
\]
for $i \le k$ and $j \le \ell$ (and $0$ otherwise).

The Doob-Martin compactification of a chain
with transition matrix $Q$ is identified in \cite[Section 4]{MR2529787},
but for the sake of completeness we present the straightforward 
computations.  If $((k_n,\ell_n))_{n \in \bN_0}$ is a sequence such
that $k_n + \ell_n \rightarrow \infty$, 
then, in order for $K((i,j),(k_n,\ell_n))$
to converge, we must have either that  $k_n = k_\infty$ for
some $k_\infty$ for all $n$ sufficiently large
and $\ell_n \rightarrow \infty$, in which case the limit is
\[
(1-q^{k_\infty-i+1}) \cdots (1-q^{k_\infty})
\times r^{-i} q^{j(k_\infty-i)}
\frac
{1}
{(1 - r)(1 - r q) \cdots (1 - r q^{j-1})}
\]
for $i \le k_\infty$ (and $0$ otherwise), or that $k_n \rightarrow \infty$
with no restriction on $\ell_n$, in which case the limit is
\[
\begin{cases}
r^{-i}, & \text{if $j=0$},\\
0, & \text{otherwise}.\\
\end{cases}
\]
Consequently, the Doob-Martin compactification 
$\overline{\bN_0 \times \bN_0}$ of the state space
is such that 
$\partial (\bN_0 \times \bN_0) := 
\overline{\bN_0 \times \bN_0} \setminus \bN_0 \times \bN_0$
can be identified with $\bN_0 \cup \{\infty\}$, the usual
one-point compactification of $\bN_0$.

With this identification, the
$h$-transformed process corresponding to the boundary point $k \in \bN_0$
has state space $\{0,\ldots,k\} \times \bN_0$,
and transition probabilities 
\[
Q^h((i,j),(i+1,j)) = (1 - q^{k-i}), \quad i < k,
\]
\[
Q^h((i,j),(i,j+1)) = q^{k-i}, \quad i < k,
\]
and
\[
Q^h((k,j),(k,j+1)) = 1.
\]
Similarly, the
$h$-transformed process corresponding to the boundary point $\infty$
has state space $\bN_0 \times \{0\}$
and transition probabilities
\[
Q^h((i,0),(i+1,0)) = 1.
\]

\subsection{$q$-binomial trees}
\label{SS:q_tree}

Suppose that we apply the trickle-down construction with $\II = \{0,1\}^\star$
and all of the routing chains given by the $q$-binomial urn of 
Subsection~\ref{SS:q_walk}, in the same manner that the BST process
and the Mallows tree process were built from the P\'olya urn and the
Mallows urn, respectively.  Just as for the latter two processes, we may
identify the state space $\SSS$ with the set of finite subtrees of
$\{0,1\}^\star$ that contain the root $\emptyset$.  We call
the resulting tree-valued Markov chain the {\em $q$-binomial tree process}.

Recalling Theorem~\ref{T:Mallows}
and comparing the conclusions
of Subsection~\ref{SS:q_walk} with those of 
Subsection~\ref{SS:Mallows_urn_compact}, 
the following result should come as no surprise.
We leave the details to the reader.

\begin{theorem}
\label{T:q-binomial}
Consider the $q$-binomial tree chain with state space $\SSS$ consisting
of the set of finite rooted binary trees.  Let
$\TT$ be the set of infinite rooted binary trees $\ttt$ 
such that $u1 \in \ttt$ for some $u \in \{0,1\}^\star$
implies $\# \ttt(u0) < \infty$.  
Equip $\SSS \sqcup \TT$ with the topology
generated by the maps $\Pi_n : \SSS \sqcup \TT \to \SSS$, 
$n \in \bN_0$, defined by
$\Pi_n(\ttt) := \{u \in \ttt : |u| \le n\}$, where on
the right we equip the countable set $\SSS$ with the discrete topology. 
The Doob-Martin compactification $\bar \SSS$ is homeomorphic to 
$\SSS \sqcup \TT$, and this homeomorphism identifies
the Doob-Martin boundary $\partial \SSS$ with $\TT$.
Moreover, each boundary
point is extremal.  
\end{theorem}

\section{Chains with perfect memory}
\label{sec:memory} 

Recall the Mallows urn model of Subsection~\ref{SS:Mallows_urn_compact}
and the $q$-binomial urn model of Subsection~\ref{SS:q_walk}. 
These Markov chains
have the interesting feature that if we know the state of the chain at some time, then we know
the whole path of the process up to that time.  In this section we examine the Doob-Martin
compactifications of such chains with a view towards re-deriving the results of
Subsection~\ref{SS:Mallows_urn_compact} and Subsection~\ref{SS:q_walk}
in a general context. We also analyze a 
trickle-down process resulting from a composition-valued Markov chain.

We return to the notation of Section~\ref{S:Martin_general}: $X=(X_n)_{n\in\bN_0}$ is a
transient Markov chain with countable state space $E$, transition matrix $P$ and reference
state $e \in E$ such that
\[
   \rho(j) :=  \bP^e\{\text{$X$ hits $j$}\} > 0, \quad \text{for all $j\in  E$}.
\]

We suppose that the chain $X$ has {\em perfect memory}, by which we mean that the sets 
\[
E_n := \{j \in E : \bP^e\{X_n = j\} > 0\}, \quad n \in \bN_0,
\]
are disjoint, and that there
is a map $f:E\setminus\{e\} \to E$ with the property that
\[
\bP^e\{f(X_n) = X_{n-1}\}=1, \quad \text{for all } n\in\bN.
\]
Note that this implies that the tail $\sigma$-field associated with the process $X$ is the
same as the $\sigma$-field $\sigma(\{X_n:\, n\in\bN_0\})$ generated by the full collection 
of variables of the process.

Suppose that we construct a directed graph $T$ that has $E$ as its set of vertices and
contains a directed edge $(i,j)$ if and only if $P(i,j) > 0$.  By the assumption on $e$,
for any $j \in E_n$, $n \in \bN$, there is a directed path $e = i_0 \to \ldots \to i_n = j$.
Also, it follows from the perfect memory assumption
that a directed edge $(i,j)$ must have $i \in E_n$ and $j \in E_{n+1}$ for some $n$.  Moreover, if
$(i,j)$ is such a directed edge, then there is no $h \in E_n$ for which $(h,j)$ is also a 
directed edge.
Combining these observations, we see that the directed graph $T$ is a rooted tree with root $e$.  The function $f$ is simply the
map that assigns to any vertex $j \in E\setminus\{e\}$ its parent.
For $j \in E_n$, $n \in \bN$, the unique directed path from $e$ to $j$ is
$e = f^n(j) \to f^{n-1}(j) \to \ldots \to f(j) \to j$.

Suppose from now on that the tree $T$ is {\em locally finite}; that is, for each $i \in E$,
there are only finitely many $j \in E$ with $P(i,j)>0$.

As usual, we define a partial order $\le$ on $T$ ($=E$) by declaring that $i \le j$ if $i$
appears on the unique directed path from the root $e$ to $j$.

We now recall the definition of the {\em end compactification} of $T$.  This object can be defined
in a manner reminiscent of the definition of the Doob-Martin compactification as follows.  We
map $T$ injectively into the space $\bR^T$ of real-valued functions on $T$ via the map that takes
$j \in T$ to the indicator function of the set $\{i \in T : i \le j\}$.  The closure of the image
of $T$ is a compact subset of $\bR^T$.  We identify $T$ with its image and write $\bar T$ for the
closure.  The compact space $\bar T$ is metrizable and a sequence $(j_n)_{n \in \bN}$ from $T$
converges in $\bar T$ if and only if ${\mathbf 1}_{\{i \le j_n\}}$ converges for all $i \in T$, where
${\mathbf 1}_{\{i \le \cdot\}}$ is the indicator function of the set $\{j \in T : i \le j\}$.  
The boundary $\partial T := \bar T \backslash T$ can be identified with the infinite 
directed paths from the root $e$.
We can extend the function ${\mathbf 1}_{\{i \le \cdot\}}$ continuously to $\bar T$.  We can
also extend the
partial order $\le$ to $\bar T$ by declaring that 
$\xi \not \le \zeta$ for any $\xi \ne \zeta \in \partial T$
and $i \le \xi$ for $\xi \in \partial T$ if and only
if ${\mathbf 1}_{\{i \le \xi\}} = 1$.

%\begin{remark}
%\label{R:end_compact_via_metric}
%Given $i,j\in T$, denote by $i \wedge j \in T$ the most recent
%common ancestor of $i$ and $j$;  that is, $i \wedge j = k$,
%where $k \le i$ and $k \le j$,
%and if $\ell$ is any other vertex with this property, then
%$\ell \le k$.
%It is well-known and easy to check
%that the end compactification of the tree $T$ 
%can also be identified with the completion of the metric space 
%$(T,d)$, where the metric $d$ is defined as follows.
%For $k \in T$, put $|k|=n$ if the unique directed path in $T$ from the
%root to $k$ passes through $n+1$ vertices.
%  Set
%\[
%   d(i,j) := \exp(-|i\wedge j|), \quad i,j\in T.
%\]
%\end{remark}

\begin{theorem}
\label{T:perfect_memory}
Let $X$ be a chain with state space $E$, reference state $e$, 
perfect memory, and locally finite associated tree $T$. 
Then, the associated Martin kernel is
given by
\[
     K(i,j) = 
     \begin{cases} 
     \rho(i)^{-1}, &\text{if $i \le j$,}\\ 
     0, &\text{otherwise,}
     \end{cases} \quad \text{ for } i,j\in E.
\]
The Doob-Martin compactification of $E$ is homeomorphic to the end 
compactification of $T$. 
The extended Martin kernel is given by
\[
     K(i,\zeta) =
     \begin{cases} 
     \rho(i)^{-1}, &\text{if  $i \le \zeta$,}\\ 
      0, &\text{otherwise,}
      \end{cases}
      \quad \text{ for } i\in E, \, \zeta\in \partial E \cong \partial T.
\]
\end{theorem}

\begin{proof}
By definition,
\[ 
   K(i,j) = \frac{\bP^i\{\text{$X$ hits $j$}\}}{\bP^e\{\text{$X$ hits $j$}\}}.
\]  
By assumption, the numerator is $0$ unless $i \le j$.  If $i \le j$, 
then the denominator is 
\[
\bP^e\{\text{$X$ hits $j$}\} 
= \bP^e\{\text{$X$ hits $i$}\} \, \bP^i\{\text{$X$ hits $j$}\}
\]
and the claimed formula for the Doob-Martin kernel follows.

The remainder of the proof is immediate from the  observation that the manner
in which the end compactification is constructed from the functions 
${\mathbf 1}_{\{i \le \cdot\}}$, $i \in E$,
is identical to the manner in which the Doob-Martin compactification is 
constructed from the
functions $K(i,\cdot) = \rho(i)^{-1} {\mathbf 1}_{\{i \le \cdot\}}$, $i \in E$.
\end{proof}

\begin{example} 
The Mallows urns process satisfies the conditions of Theorem~\ref{T:perfect_memory}. 
The tree $T$ has $\bN_0^2$ as its set of vertices, and directed edges
of the form $((i,0),(i+1,0))$ and $((i,j),(i,j+1))$, $i,j\in\bN_0$.
The perfect memory property survives the lift from urn to tree. 
The ``parenthood'' function $f$ 
takes a tree $\ttt$ in the state space of the Mallows tree process
and simply removes the vertex of $\ttt$ that is greatest
in the lexicographic order.

This description of the  state space of the Mallows tree
process as a ``tree-of-trees'' also makes
its Doob-Martin compactification easier to understand. 
 We know from Section~\ref{SS:Mallows_tree_compact} that
points in the Doob-Martin boundary can be identified
with rooted binary trees with a single infinite path --
the ``spine'' -- with nothing dangling off to the right of the spine. 
It is, of course, easy to construct a
sequence of finite rooted binary trees that tries to grow more than one infinite
path: for example, let $\ttt_n$ be the
tree that consists of the two vertices 
$00 \ldots 0, \, 11 \ldots 1 \in\{0,1\}^n$ and the vertices
in $\{0,1\}^\star$ on the directed paths connecting them to the
root $\emptyset$. The sequence $(\ttt_n)_{n\in\bN}$ must have a 
subsequence with a limit point in the compact space
$\bar \SSS$ or, equivalently, it must have a subsequence that
converges to a limit in the end compactification $\bar T$ of the tree $T$.
From the above description of the parenthood function $f$, we 
see for a tree $\sss \in T$ that $\sss \le \ttt_n$ if and only if 
one of the following three conditions hold:
\begin{itemize}
\item
$\sss$ consists of the two vertices $00 \ldots 0 \in\{0,1\}^n$
and $11 \ldots 1\in\{0,1\}^m$ for some $m \le n$ and their
prefixes in $\{0,1\}^\star$;
\item
$\sss$ consists of the vertex  $00\ldots 0 \in\{0,1\}^m$
for some $m \le n$ and its prefixes in $\{0,1\}^\star$;
\item
$\sss$ consists of the single vertex $\emptyset \in \{0,1\}^\star$.
\end{itemize}
It follows that $\sss \le \ttt_n$ for all $n$ sufficiently large
if and only if $\sss$ is the tree consisting of some element of
$\{0\}^\star$ and its prefixes in $\{0,1\}^\star$.
Thus, $\ttt_n$ converges in the end compactification of $T$ to 
$\ttt_\infty \in \bar T \setminus T$ as $n\to\infty$, where we can
regard $\ttt_\infty$ as the single infinite path tree
consisting solely of the infinite spine
$\emptyset \to 0 \to 00 \to \ldots$. 

We note that  the sequence $(\ttt_n)_{n\in\bN}$ 
of finite rooted binary trees 
converges even in the Doob-Martin compactification 
of the binary search tree process to a point in the boundary.
Indeed (see the first paragraph of Section~\ref{S:BST_and_DST}), we can identify this latter point
with the probability measure on $\{0,1\}^\infty$
that puts mass $\frac{1}{2}$ at each of the points
$00\ldots$ and $11\ldots$.
\end{example}

\begin{example}\label{E:composition} 
A {\em composition} of an integer $n\in\bN$ is an element $c =(c_1 ,\ldots,c_k)$ of $\bN^\star$ 
with the property that $\sum_{i=1}^k c_i=n$. We recall the standard proof of the fact 
that there are $2^{n-1}$ such compositions for a given $n$: one thinks of placing 
$n$ balls on a string and defines a composition 
by placing separators into some of the $n-1$ gaps between the balls. A combinatorially 
equivalent bijection arises from deleting the last of these balls, labeling the balls 
to the left of each separator by 1 and labeling the remaining balls by $0$. 
We can now construct a Markov chain $(X_n)_{n\in\bN}$ such that $X_n$ is uniformly
distributed on the set of compositions of $n$ and $X_n$ is a prefix of $X_{n+1}$ for 
all $n\in\bN$:
the state space is $E=\{0,1\}^\star$ and the allowed transitions are of the form
\[
   (u_1,\ldots,u_{n-1})\to (u_1,\ldots,u_{n-1},1),
                     \quad (u_1,\ldots,u_{n-1})\to (u_1,\ldots,u_{n-1},0),
\]
both with probability $1/2$. Here $X_1=\emptyset$ represents the only composition $1=1$ of
$n=1$. Attaching the digit $1$ to the state representing a composition of $n$ means that the
new composition, now of $n+1$, has an additional summand of size $1$ at the end, whereas
adding $0$ corresponds to increasing the last summand of the old composition by $1$.
A construction of this type, which relates random compositions to samples from a 
geometric distribution, has been used in~\cite{MR1871561} -- see also the references 
given there.

The chain $(X_n)_{n\in\bN}$ certainly has the perfect memory property and the associated tree
$T$ is just the complete rooted binary tree structure on $\{0,1\}^\star$ from the
Introduction.  It follows from Theorem~\ref{T:perfect_memory} that the Doob-Martin
compactification is homeomorphic to $\{0,1\}^\star \sqcup \{0,1\}^\infty$, the end
compactification of $\{0,1\}^\star$. 

Note that we can also think of the chain $(X_n)_{n \in \bN}$ as a result of the trickle-down
construction in which the underlying directed acyclic graph $\II$ is the complete rooted
binary tree, the routing instruction chains all have state space $\{(0,0)\} \sqcup (\bN \times
\{0\}) \sqcup (\{0\} \times \bN)$, and transition matrices are all of the form
\[
   \begin{split}
      Q((0,0),(1,0)) & = \frac{1}{2}, \\
      Q((0,0),(0,1)) & = \frac{1}{2}, \\
      Q((i,0),(i+1,0)) & = 1, \; i \ge 1, \\
      Q((0,j),(0,j+1)) & = 1, \; j \ge 1. 
   \end{split}
\]
The chain is of the single trail type described in Example~\ref{E:trail}. For processes
of this type there are usually several possibilities for the underlying directed graph; here we
may take $\II=\bN_0 \times \bN_0$ 
instead of the complete rooted binary tree if we interpret
appending $0$ as a move to the right and appending $1$ as a move up.
\end{example}

\begin{remark}
For several of the chains $(X_n)_{n\in\bN_0}$ that we have considered in the previous
sections there is
a ``background chain'' $(\tilde X_n)_{n\in\bN_0}$ with the perfect memory property in the
sense that there is a function $\Psi:\tilde S\to S$ with $X_n=\Psi(\tilde X_n)$ for
all $n\in\bN$, where $S$ and $\tilde S$ are the respective state spaces. 
For example, random recursive trees are often considered together with their labels
and are then of the perfect memory type -- see Figure~\ref{fig:recursive_tree}. 

Conversely, we can always extend the state space $S$ of a given chain by including the
previous states, taking the new state space $\tilde S$ to be the set of words from the
alphabet $S$, to obtain a background chain of the perfect memory type. For example, the
P\'olya urn then leads to a single trail chain in the sense of Example~\ref{E:trail}, with underlying
directed graph $\bN\times\bN$ and transitions $ Q((i,j),(i+1,j)) = i/(i+j)$ and $Q((i,j),(i,j+1)) =
j/(i+j)$.
\end{remark}

\section{Another approach to tail $\sigma$-fields}
\label{S:freezing}

As mentioned in the Introduction, our initial motivation for
studying the Doob-Martin compactifications of various trickle-down chains
was to understand the chains' tail $\sigma$-fields.  Determining
the compactification requires a certain amount of knowledge about
the hitting probabilities of a chain, and this information may
not always be easy to come by.  In this section we consider
a family of trickle-down chains for which it is possible
to describe the tail $\sigma$-field directly without recourse
to the more extensive information provided by the
Doob-Martin compactification.  The class of processes to which
this approach applies includes the Mallows tree and $q$-binomial
tree process that we have already analyzed, as well as the
Catalan tree process of Section~\ref{S:Catalan_trees} below that we are unable
to treat with Doob-Martin compactification methods.

We begin with a lemma that complements a 
result from \cite{MR699981} on exchanging the order of 
taking suprema 
and intersections of $\sigma $-fields.

\begin{lemma}
\label{L:commuting_sigma-fields}
Suppose that on a probability space $(\Omega, \cF, \bP)$ there is
a collection of independent sub-$\sigma$-fields $\cH_m$, $m \in \bN_0$,
and another collection of sub-$\sigma$-fields 
$\cG_{m,n}$, $m, n \in \bN_0$, with the properties
\[
\cG_{0,n} \subseteq \cG_{1 ,n} \subseteq \ldots, \quad \text{for all $n \in \bN_0$},
\]
\[
\cG_{m,0} \supseteq \cG_{m ,1} \supseteq \ldots, \quad \text{for all $m \in \bN_0$},
\]
\[
\cG_{0,0} \subseteq \cH_0,
\]
and
\[
\cG_{m+1,n} \subseteq \cG_{m,n} \vee \cH_{m+1}, \quad \text{for all $m,n \in \bN_0$}.
\]
Then, the two sub-$\sigma$-fields 
$\bigvee_{m \in \bN_0} \bigcap_{n \in \bN_0} \cG_{m,n}$ 
and 
$\bigcap_{n \in \bN_0} \bigvee_{m \in \bN_0} \cG_{m,n}$
are equal up to null sets.
\end{lemma}

\begin{proof}
We first establish that
\[
\bigvee_{m \in \bN_0} \bigcap_{n \in \bN_0} \cG_{m,n}
\subseteq
\bigcap_{n \in \bN_0} \bigvee_{m \in \bN_0} \cG_{m,n}.
\]
It suffices to check for each $M \in \bN_0$ that
\[
\bigcap_{n \in \bN_0} \cG_{M,n} 
\subseteq 
\bigcap_{n \in \bN_0} \bigvee_{m \in \bN_0} \cG_{m,n},
\]
but this follows from the observation that
\[
\cG_{M,n} \subseteq \bigvee_{m \in \bN_0} \cG_{m,n}
\]
for every $n \in \bN_0$.

We now verify that
\[
\bigvee_{m \in \bN_0} \bigcap_{n \in \bN_0} \cG_{m,n}
\supseteq
\bigcap_{n \in \bN_0} \bigvee_{m \in \bN_0} \cG_{m,n}
\]
up to null sets. For this it suffices to show that  any bounded random variable $Z$ that 
is measurable with respect to $\bigvee_{m \in \bN_0} \bigcap_{n \in \bN_0} \cG_{m,n}$, satisfies 
the equality
\[
\bE\left[Z \, \Big | \, \bigcap_{n \in \bN_0} \bigvee_{m \in \bN_0} \cG_{m,n}\right]= Z \quad \text{ a.s. }
\]
%Because both sub-$\sigma$-fields are
%contained in the sub-$\sigma$-field $\bigvee_{m \in \bN_0} \cH_m$, it suffices
%to show for any bounded $\bigvee_{m \in \bN_0} \cH_m$-measurable
%random variable $Z$ that there is a 
%$\bigvee_{m \in \bN_0} \bigcap_{n \in \bN_0} \cG_{m,n}$-measurable
%version of 
%\[
%\bE\left[Z \, \Big | \, \bigcap_{n \in \bN_0} \bigvee_{m \in \bN_0} \cG_{m,n}\right].
%\]
By a monotone class argument, we may further suppose that $Z$ is
measurable with respect to 
$\bigvee_{m=0}^M \bigcap_{n \in \bN_0} \cG_{m,n}
=
\bigcap_{n \in \bN_0} \cG_{M,n}$
for some $M \in \bN_0$.  
Our assumptions guarantee that for all $n \in \mathbb N_0$ and $m > M$
\[
\cG_{m,n} \subseteq \cG_{M,n} \vee \cH_{M+1} \vee \cdots \vee \cH_m
\]
and 
\[\cG_{M,n} \subseteq \cH_{0} \vee \cdots \vee \cH_M.\]
From these inclusions, the backwards and forwards martingale
convergence theorems and the assumed independence of the $\cH_j$, $j=0,1,\ldots$ 
we see that
\[
\begin{split}
& \bE\left[Z \, \Big | \, \bigcap_{n \in \bN_0} \bigvee_{m \in \bN_0} \cG_{m,n}\right] \\
& \quad =
\lim_{n \rightarrow \infty}
\bE\left[Z \, \Big | \, \bigvee_{m \in \bN_0} \cG_{m,n}\right] \\
& \quad =
\lim_{n \rightarrow \infty}
\lim_{m \rightarrow \infty}
\bE\left[Z \, | \, \cG_{m,n}\right] \\
& \quad =
\lim_{n \rightarrow \infty}
\lim_{m \rightarrow \infty}
\bE\left[ 
\bE\left[Z \, | \, \cG_{M,n} \vee \cH_{M+1} \vee \cdots \vee \cH_m\right]       
\, | \, 
\cG_{m,n}\right] \\
& \quad =
\lim_{n \rightarrow \infty}
\lim_{m \rightarrow \infty}
\bE\left[ 
\bE\left[Z \, | \, \cG_{M,n}\right]       
\, | \, 
\cG_{m,n}\right] \\
& \quad =
\lim_{n \rightarrow \infty}
\bE\left[Z \, | \, \cG_{M,n}\right] \\
& \quad =
\bE\left[Z \, \Big | \, \bigcap_{n \in \bN_0} \cG_{M,n}\right] = Z \quad \text{ a.s. },\\
\end{split}
\]
%and the last random variable is   
%$\bigvee_{m \in \bN_0} \bigcap_{n \in \bN_0} \cG_{m,n}$-measurable,
as required.
\end{proof}

By the assumptions of the trickle-down construction,
$((Y_n^u)^v)_{n \in \bN_0}$ is nondecreasing $\bQ^{u,\xi}$-almost surely
for every $u \in \II$, $v \in \beta(u)$ and $\xi \in \SSS^u$.
Therefore, 
$(Y_\infty^u)^v := \lim_{n \rightarrow \infty} (Y_n^u)^v$
exists $\bQ^{u,\xi}$-almost surely in the usual one-point compactification
$\bN_0 \sqcup \{\infty\}$ of $\bN_0$.

Recall for the Mallows tree  and
$q$-binomial tree processes that $\II = \{0,1\}^\star$
and that the routing chains in both cases all had the property
$(Y_\infty^u)^{u0} < \infty$ and $(Y_\infty^u)^{u1} = \infty$,
$\bQ^{u,\xi}$-almost surely.  We see from the following result that it is
straightforward to identify the tail $\sigma$-field for a
trickle-down process if all of its routing chains exhibit this
kind of behavior.  Another example is the
Catalan tree process defined in Section~\ref{S:Catalan_trees} below -- see Proposition~\ref{P:Catalan_tail}.

\begin{proposition}
\label{P:tail_identify}
Suppose that $\beta(u)$ is finite for all $u \in \II$.
Fix $x \in \SSS$.   Suppose that 
$\#\{v \in \beta(u) : (Y_\infty^u)^v = \infty\} = 1$,
$\bQ^{u,x^u}$-a.s. for all $u \in \II$.
Then, the tail $\sigma$-field
\[
\bigcap_{m \in \bN_0} \sigma\{X_n : n \ge m\}
\]
is generated by $X_\infty := (X_\infty^u)_{u \in \II}$
up to $\bP^x$-null sets.
\end{proposition}

\begin{proof}
By the standing hypotheses on $\II$ and the assumption that
$\beta(u)$ is finite for all $u \in \II$, we can list
$\II$ as $(u_p)_{p \in \bN_0}$ in such a way that
$u_p \le u_q$ implies $p \le q$ (that is, we can put
a total order on $\II$ that refines the partial order $\le$
in such a way that the resulting totally ordered set has the
same order type as $\bN_0$).  For each $p \in \bN_0$, put 
$\JJ_p := \{u_0, \ldots, u_p\}$.
By Remark~\ref{R:hereditary}, each process $((X_n^u)_{u \in \JJ_p})_{n \in \bN_0}$
is a Markov chain.

Now,
\[
\bigcap_{m \in \bN_0} \sigma\{X_n : n \ge m\}
=
\bigcap_{m \in \bN_0} \bigvee_{p \in \bN_0} 
\sigma\{X_n^u : u \in \JJ_p, \, n \ge m\}.
\]
By construction,
\[
\sigma\{X_n^u : u \in \JJ_{p+1}, \, n \ge m\}
\subseteq
\sigma\{X_n^u : u \in \JJ_p, \, n \ge m\}
\vee 
\sigma\{Y_n^{u_{p+1}}:  n \in \bN_0\}.
\]
Thus, by Lemma~\ref{L:commuting_sigma-fields},
\[
\bigcap_{m \in \bN_0} \sigma\{X_n : n \ge m\}
=
\bigvee_{p \in \bN_0} \bigcap_{m \in \bN_0}  
\sigma\{X_n^u : u \in \JJ_p, \, n \ge m\}
\]
up to $\bP^x$-null sets. To show the claimed assertion, it thus suffices to check that 
for all $p\in \mathbb N_0$
\[\bigcap_{m \in \bN_0} \sigma\{X_n^u : u \in \JJ_p, \, n \ge m\}
=
\sigma\{X_\infty^u : u \in \JJ_p\}.
\]
We establish this via induction as follows.

For brevity we 
%To simplify notation,
suppose that $x^u = (0,0,\ldots)$ for all $u \in \II$.
In this way we avoid the straightforward but somewhat tedious notational complications 
of the general case.
%The proof in the general case involves no further new ideas and is left
%to the reader.

By assumption, there is a $\bP^x$-a.s. unique random element 
$V_0 \in \beta(u_0) = \beta(\hat 0)$ such that
$(X_\infty^{u_0})^{V_0} = \infty$.  
With $\bP^x$-probability one,
\[
(X_n^{u_0})^v
=
\begin{cases}
(X_\infty^{u_0})^v, & \text{if $v \ne V_0$},\\
n - \sum_{w \ne V_0} (X_\infty^{u_0})^w, & \text{if $v = V_0$}.
\end{cases}
\]
for all $v \in \beta(u_0)$ and $n$ sufficiently large.
Thus, 
$\bigcap_{m \in \bN_0} \sigma\{X_n^u : u \in \JJ_0, \, n \ge m\}$
is 
%certainly 
generated by $(X_\infty^u)_{u \in \JJ_0} = X_\infty^{\hat 0}$ up to $\bP^x$-null sets.

Suppose we have shown for some $p \in \bN_0$ that 
$\bigcap_{m \in \bN_0} \sigma\{X_n^u : u \in \JJ_p, \, n \ge m\}
=
\sigma\{X_\infty^u : u \in \JJ_p\}$ 
up to $\bP^x$-null sets.

Now,
\[
A_n^{u_{p+1}}
=
\left(\sum_{u \in \alpha(u_{p+1})} (X_n^u)^{u_{p+1}} - 1\right)_+.
\]
Because 
$\alpha(u_{p+1}) \subseteq \JJ_p$,
it follows from our inductive hypothesis that 
\[
\bigcap_{m \in \bN_0} \sigma\{A_n^{u_{p+1}} : n \ge m\} 
\subseteq
\bigcap_{m \in \bN_0} \sigma\{X_n^u : u \in \JJ_p, \, n \ge m\}
=
\sigma\{X_\infty^u : u \in \JJ_p\}
\]
up to $\bP^x$-null sets.
In particular, the $\bN_0 \sqcup \{\infty\}$-valued random variable
\[
A_\infty^{u_{p+1}}
:= 
\lim_{n \rightarrow \infty} A_n^{u_{p+1}}
\]
is $\sigma\{X_\infty^u : u \in \JJ_p\}$-measurable
up to $\bP^x$-null sets.

On the event $\{A_\infty^{u_{p+1}} = \infty\}$,
there is a unique random element $V_{p+1} \in \beta(u_{p+1})$
such that $(X_\infty^{u_{p+1}})^{V_{p+1}} = \infty$ and
\[
(X_n^{u_{p+1}})^v
=
\begin{cases}
(X_\infty^u)^v, & \text{if $v \ne V_{p+1}$},\\
A_n^{u_{p+1}} - \sum_{w \ne V_{p+1}} (X_\infty^{u_{p+1}})^w, & \text{if $v = V_{p+1}$},
\end{cases}
\]
for each $v \in \beta(u_{p+1})$ and $n$ sufficiently large.
Note that 
\[
\{A_\infty^{u_{p+1}} = \infty, \, v = V_{p+1}\}
=
\{(X_\infty^{u_{p+1}})^v = \infty\}
\]
for each $v \in \beta(u_{p+1})$.  It follows that
\[
\begin{split}
& \bigcap_{m \in \bN_0} 
\left[
\sigma\{X_n^u : u \in \JJ_p, \, n \ge m\}
\vee
\sigma\{X_n^{u_{p+1}} \boldsymbol{1} \{A_\infty^{u_{p+1}} = \infty\} : n \ge m\}
\right] \\
& \quad \subseteq
\sigma\{X_\infty^u : u \in \JJ_{p+1}\} \\
\end{split}
\]
up to $\bP^x$-null sets.

Furthermore, on the event $\{A_\infty^{u_{p+1}} < \infty\}$, 
$X_n^{u_{p+1}} = X_\infty^{u_{p+1}}$ for all $n$ sufficiently large,
and so 
\[
\begin{split}
& \bigcap_{m \in \bN_0} 
\left[
\sigma\{X_n^u : u \in \JJ_p, \, n \ge m\}
\vee
\sigma\{X_n^{u_{p+1}} \boldsymbol{1} \{A_\infty^{u_{p+1}} < \infty\} : n \ge m\}
\right] \\
& \quad \subseteq
\sigma\{X_\infty^u : u \in \JJ_{p+1}\} \\
\end{split}
\]
up to $\bP^x$-null sets. This completes the induction step, and thus the proof of the proposition.
\end{proof}

\begin{remark}
\label{R:single_spine}
When $\II$ is a tree and we are in the situation of Proposition \ref{P:tail_identify}, 
then $X_\infty$ may be thought of 
as an infinite rooted subtree of $\II$ with a single infinite 
directed path from the root $\hat 0$.  Regarding $(X_n)_{n \in \bN_0}$
as a tree-valued process, we have $X_\infty = \bigcup_{n \in \bN_0} X_n$.
Equivalently, $X_\infty$ is the limit of the finite subsets $X_n$ of
$\II$ if we identify the subsets of $\II$ with the Cartesian product
$\{0,1\}^\II$ in the usual way and equip the latter space with
the product topology.
\end{remark}

\section{The Catalan tree process}
\label{S:Catalan_trees}

Let $\SSS_n$ denote the set of subtrees of the complete rooted binary
tree $\{0,1\}^\star$ that contain the root $\emptyset$ and have
$n$ vertices.  The set $\SSS_n$ has cardinality $C_n$, where
\[
C_n:=\frac{1}{n+1}\binom{2n}{n}
\]
is the {\em $n^{\mathrm{th}}$ Catalan number}.  A special case
of a construction in \cite{MR2060629} gives a Markov chain 
$(X_n)_{n \in \bN_0}$ with state space the set of finite rooted subtrees
of $\{0,1\}^\star$ such that 
\begin{equation}
\label{uniform_distn}
\bP^{\{\emptyset\}}\{X_n = \ttt\} = C_{n+1}^{-1}, \quad \ttt \in \SSS_n;
\end{equation} 
that is, if the chain begins in the trivial tree
$\{\emptyset\}$, 
then its value at time $n$ is uniformly distributed on $\SSS_{n+1}$.
Moreover, the construction in \cite{MR2060629} is an instance of the
trickle-down construction in which $\II = \{0,1\}^\star$ and all of the
routing chains have the same dynamics.

For the sake of completeness, we reprise some of the 
development from \cite{MR2060629}.  
Begin with the ansatz that there is indeed
a trickle-down process $(X_n)_{n \in \bN_0}$
with $\II = \{0,1\}^\star$ and identical routing chains such that
\eqref{uniform_distn} holds.  Identify
the state spaces of the routing chains with $\bN_0 \times \bN_0$
and write $Q$ for the common transition matrix.
We have
\begin{equation}
\label{eq:catalan}
\begin{split}
Q^n((0,0),(k, n-k))
& =
\bP^{\{\emptyset\}}\left\{\# X_n(0)=k, \, \# X_n(1) =n-k \right\}  \\
& = \frac{C_k C_{n-k}}{C_{n+1}}, \quad n \in \bN,\; k=0,\ldots,n. \\
\end{split}
\end{equation}

Now,  
\[
Q((j,i), (j,i+1)) = Q((i,j),(i+1,j) = 1 - Q((i,j),(i,j+1))
\]
by symmetry,
\[
\begin{split}
Q((0,j),(0,j+1)) 
&=  
\frac
{\bP^{\{\emptyset\}} \left\{\#X_{j+1}(0)=0, \, \#X_{j+1}(1) = j+1\right\}}
{\bP^{\{\emptyset\}} \left\{\#X_{j}(0)=0, \, \#X_{j}(1) = j\right\}} \\
&=
\frac{C_0 C_{j+1}}{C_{j+2}} \bigg / \frac{C_0 C_{j}}{C_{j+1}}  
= \frac{(j+3)(2j+1)}{(j+2)(2j+3)},\\
\end{split}
\] 
and
\[
\begin{split}
& \bP^{\{\emptyset\}} \left\{\#X_{i+j}(0)=i, \, \#X_{i+j}(1) = j\right\} \\
& \quad =
\bP^{\{\emptyset\}} \left\{\#X_{i+j-1}(0)=i-1, \, \#X_{i+j-1}(1) = j\right\}
Q((i-1,j), (i,j)) \\
& \qquad +
\bP^{\{\emptyset\}} \left\{\#X_{i+j-1}(0)=i, \, \#X_{i+j-1}(1) = j-1\right\}
Q((i,j-1), (i,j)) \\
\end{split}
\]
where the appropriate probabilities on the right side are $0$ if $i=0$ or $j=0$,
so that
\[
  Q(i,j-1), (i,j))\, =\frac{2j-1}{j+1}\left(\frac{i+j+2}{2i+2j+1}
                     - \left(1-Q((i-1,j), (i,j))\right)\frac{i+1}{2i-1}\right).
\]

Combining these observations, we can calculate the entries of the
transition matrix $Q$ iteratively and, as observed in \cite{MR2060629},
the entries of $Q$ are non-negative and the rows of $Q$ sum to one.
We refer to the resulting Markov chain as the {\em Catalan urn process}. 
Note that if the random tree $T$ is uniformly
distributed on $\SSS_{n+1}$ then, conditional on 
the event $\{\#T(0) = k, \#T(1) = n-k\}$,
the random trees $\{u \in \{0,1\}^* : 0u \in T\}$ 
and $\{u \in \{0,1\}^* : 1u \in T\}$ are independent and uniformly distributed
on $\SSS_k$ and $\SSS_{n-k}$, respectively.  Thus,
a trickle-down construction with each routing chain given by the
Catalan urn process does indeed give a tree-valued chain
satisfying \eqref{uniform_distn}.

Observe that
\[
\lim_{n\to\infty}\frac{C_{n+1}}{C_n}= 4 .
\]
It follows from \eqref{eq:catalan} that
\[
\lim_{\ell \to \infty} Q^{k+\ell}((0,0), (k,\ell))
=
\lim_{\ell \to \infty} Q^{k+\ell}((0,0), (\ell,k))
=
4^{-(k+1)} C_k
\]
for all $k \in \bN_0$.  Moreover, 
$2 \sum_{k \in \bN_0}  4^{-(k+1)} C_k = 1$ from the well-known fact that
the generating function of the Catalan numbers is
\[
\sum_{k \in \bN_0} C_k x^k = \frac{2}{1 + \sqrt{1 - 4 x}}, 
\quad |x| < \frac{1}{4}.
\]
Hence, if $((Y_n', Y_n''))_{n \in \bN_0}$ is a Markov chain with
transition matrix $Q$ and laws $\bQ^{(y',y'')}$, then 
\[
(Y_\infty', Y_\infty'') 
:= \left(\lim_{n \to \infty} Y_n', \lim_{n \to \infty} Y_n''\right)
\in (\bN_0 \times \{\infty\}) \sqcup (\{\infty\} \times \bN_0), 
\quad \text{$\bQ^{(0,0)}$-a.s.},
\]
with 
\[
\bQ^{(0,0)}\{(Y_\infty', Y_\infty'') = (k,\infty)\}
=
\bQ^{(0,0)}\{(Y_\infty', Y_\infty'') = (\infty,k)\}
=
4^{-(k+1)} C_k, \quad k \in \bN_0.
\]

The following result is immediate from Proposition~\ref{P:tail_identify}.

\begin{proposition}
\label{P:Catalan_tail}
The tail $\sigma$-field
of the Catalan tree process 
$(X_n)_{n \in \bN_0}$ is generated up to null sets by the infinite random tree
$X_\infty := \bigcup_{n \in \bN_0} X_n$ under $\bP^{\{\emptyset\}}$.
\end{proposition}

As we noted in Remark~\ref{R:single_spine}, the tree $X_\infty$ has a single
infinite path from the root $\emptyset$.  Denote this path by
$\emptyset = U_0 \to U_1 \to \ldots$. For $n \in \bN$, define $W_n \in \{0,1\}$ by
$U_n = W_1 \ldots W_n$.
It is apparent from the trickle-down construction and the discussion above
that the sequence $(W_n)_{n \in \bN}$ is i.i.d. with
$\bP\{W_n=0\} = \bP\{W_n=1\} = \frac{1}{2}$.  Moreover, if we set
$\bar W_n = 1 - W_n$ and put
\[
T_n := \{u \in \{0,1\}^\star : W_1 \ldots W_{n-1} \bar W_n u \in X_\infty\},
\]
so that $T_n$ is either empty or a subtree of $\{0,1\}^\star$
rooted at $\emptyset$, then the sequence $(T_n)_{n \in \bN}$
is i.i.d. and independent of $(W_n)_{n \in \bN}$ with
\[
\bP\{\#T_n = k\} = 2 \times 4^{-(k+1)} C_k, \quad k \in \bN_0,
\]
and
\[
\bP\{T_n = \ttt \, | \, \#T_n = k\} 
= \frac{1}{C_k}, \quad \ttt \in \SSS_k, \, k \in \bN.
\]

Note that if $(S_n)_{n\in\bN_0}$ is {\bf any} sequence of random subtrees
of $\{0,1\}^*$ such that $S_n$ is uniformly distributed on $\SSS_{n+1}$
for all $n \in \bN_0$, then $S_n$ converges in distribution to a random
tree that has the same distribution as $X_\infty$, where the notion of
convergence in distribution is the one that comes from thinking of
subtrees of $\{0,1\}^*$ as elements of the Cartesian product
$\{0,1\}^{\{0,1\}^*}$ equipped with the product topology 
--- see Remark~\ref{R:single_spine}.  
The convergence in distribution of such a sequence $(S_n)_{n\in\bN_0}$ 
and the above description of the limit distribution have already been obtained 
in~\cite{MR1940149} using different methods.
For a similar weak convergence result for uniform random trees, see \cite{MR607933}
and the survey \cite[Section~2.5]{MR2023650}. 
Also, if we define rooted finite $d$-ary trees for $d > 2$
as suitable subsets of $\{0,1,\ldots,d-1\}^\star$ 
in a manner analogous to the way we
have defined rooted finite binary trees, then
it is shown in \cite{MR2060629} that it is possible to construct
a Markov chain that grows by one vertex at each step and is uniformly
distributed on the set of $d$-ary trees with $n$ vertices
at step $n$ -- in particular, there is an almost sure (and hence
distributional) limit as $n \to \infty$ in the same sense
as we just observed for the uniform binary trees.  We have
not investigated whether this process is the result of a trickle-down
construction.  Lastly, we note that there are interesting ensembles
of trees that can't be embedded into a trickle-down construction
or, indeed, into any Markovian construction in which a single
vertex is added at each step; for example, it is shown in
\cite{MR2509643} that this is not possible for the
ensemble obtained by taking
a certain critical Galton-Watson tree with offspring distribution
supported on $\{0,1,2\}$ and conditioning the total 
number of vertices to be $n \in \bN$.

\bigskip\noindent
{\bf Acknowledgments.}  The authors thank Heinrich von Weizs\"acker
for his role in bringing them together to work on the problems
considered in the paper.  They also thank Sasha Gnedin 
and the referee for a number
of helpful comments and pointers to the literature.

\def\cprime{$'$}
\providecommand{\bysame}{\leavevmode\hbox to3em{\hrulefill}\thinspace}
\providecommand{\MR}{\relax\ifhmode\unskip\space\fi MR }
% \MRhref is called by the amsart/book/proc definition of \MR.
\providecommand{\MRhref}[2]{%
  \href{http://www.ams.org/mathscinet-getitem?mr=#1}{#2}
}
\providecommand{\href}[2]{#2}

\end{document}